\documentclass{amsart}[11pt]
\usepackage{graphicx,url,amsthm,xspace,enumitem,subfig}
\usepackage[dvipsnames]{xcolor}
\usepackage{amscd}
\usepackage{multirow}
\usepackage[utf8]{inputenc}

\usepackage{amssymb}
\usepackage{graphics}
\usepackage{graphicx}
\usepackage{amscd}
\usepackage{color}
\usepackage{amsmath, amsthm, epsfig,  psfrag}
\usepackage{pinlabel}
\usepackage{hyperref}
\usepackage[margin=3cm, marginpar=2.5cm]{geometry}

\usepackage[all,pdftex,arc,curve,color,frame]{xy}
\usepackage{xy}
\xyoption{all}

\usepackage{pinlabel}

\newcommand{\C}{\mathbb{C}}
\newcommand{\R}{\mathbb{R}}
\newcommand{\Q}{\mathbb{Q}}
\newcommand{\Z}{\mathbb{Z}}
\newcommand{\cptwo}{\C\textup{P}^2}
\newcommand{\cpone}{\C\textup{P}^1}
\newcommand{\cptwobar}{\overline{\C\textup{P}}\,\!^2}
\newcommand{\rptwo}{\R\textup{P}^2}

\newcommand{\ti}{\tilde}
\renewcommand{\d}{\partial}
\DeclareMathOperator{\ord}{ord}
\renewcommand{\Im}{\operatorname{Im}}
\renewcommand{\Re}{\operatorname{Re}}
\DeclareMathOperator{\Ker}{ker}

\newtheorem{theorem}{Theorem}[section]
\newtheorem{lemma}[theorem]{Lemma}
\newtheorem{prop}[theorem]{Proposition}
\newtheorem{cor}[theorem]{Corollary}
\newtheorem{question}[theorem]{Question}

\theoremstyle{definition}
\newtheorem{definition}[theorem]{Definition}
\newtheorem{remark}[theorem]{Remark}
\newtheorem{example}[theorem]{Example}

\numberwithin{equation}{section}
\newtheorem{observation}[theorem]{Observation}

\title[Unexpected Stein fillings and plane curve arrangements]{Unexpected Stein fillings, rational surface singularities, \\ and plane curve arrangements}

\author{Olga Plamenevskaya}
\address{Deparment of Mathematics, University of California, Davis, 1 Shields Avenue, Davis, CA, 95616, U.S.A.}
\email{lstarkston@math.ucdavis.edu}
\author{Laura Starkston}
\address{Department of Mathematics, Stony Brook University, Stony Brook, NY,
11794,  U.S.A.}
\email{olga@math.stonybrook.edu}

\begin{document}

\begin{abstract}
 We compare Stein fillings and Milnor fibers for rational surface singularities with 
reduced fundamental cycle.  Deformation theory for this class of singularities was studied
by de Jong--van Straten in \cite{dJvS}; they associated a germ of a singular 
plane curve to each singularity and described Milnor fibers via deformations of this singular curve.  

We consider links of surface singularities, equipped with their canonical contact structures, 
and develop a symplectic analog of  de Jong--van Straten's construction.  Using planar open books and Lefschetz fibrations, we describe
all Stein fillings of the links via certain arrangements of symplectic disks, related by a homotopy to the plane curve germ of the singularity.  

As a consequence, we show that many rational singularities in this class admit Stein fillings that are not strongly diffeomorphic to any Milnor fibers. 
This contrasts with previously known cases, such as simple and 
quotient surface singularities, where Milnor fibers are known to give rise to all Stein fillings. On the other hand, 
we show that  if for a  singularity with reduced fundamental cycle, the self-intersection of each exceptional curve is at
most $-5$ in the minimal resolution, then the link has a unique Stein filling (given by a Milnor fiber).

\end{abstract}

\maketitle

\section{Introduction}

The goal of this paper is to compare and contrast deformation theory and symplectic topology of certain rational surface singularities.  Using topogical tools, 
we examine symplectic fillings for links of rational surface singularities with reduced fundamental cycle and compare these fillings to Milnor fibers of the singularities. 
Each Milnor fiber carries a Stein structure and thus gives a Stein filling of the link; however, we show that there is a plethora of Stein fillings that do not 
arise from Milnor fibers.  Milnor fibers and deformation theory are studied in the work of de Jong--van Straten \cite{dJvS} for sandwiched surface singularities (this class includes rational singularities with reduced fundamental cycle). The main feature of their construction is a reduction from 
{\em surfaces} to {\em curves}: deformations of a surface singularity in the given class can be understood via deformations of the germ of a reducible plane curve associated to 
the singularity. To describe Stein fillings,  we develop a symplectic analog of de Jong--van Straten's constructions, representing the fillings via arrangements 
of smooth (or symplectic) disks in $\C^2$. Our approach is  purely topological and thus different from de Jong--van Straten's; their algebro-geometric techniques do not apply in our more general symplectic setting.
We work with Lefschetz fibrations and open books, referring to algebraic geometry only for motivation and for the description of smoothings from \cite{dJvS}.  

Let $X\subset \C^N$ be a singular complex surface with an isolated singularity at the origin. For small $r>0$, the intersection $Y= X \cap S^{2N-1}_{r}$ with the sphere $S^{2N-1}_{r}= \{|z_1|^2+|z_2|^2+\dots +|z_N|^2 = r\}$ is 
a smooth 3-manifold called the {\em link of the singularity} $(X,0)$. The induced contact structure $\xi$ on $Y$ is the distribution of complex tangencies to $Y$, 
and is referred to as the {\em canonical} or {\em Milnor fillable} contact structure on the link. The contact manifold $(Y, \xi)$, 
which we will call the {\em contact link}, is independent of the choice of $r$, 
up to contactomorphism.

An important problem concerning the topology of a surface singularity is to compare the Milnor fibers of smoothings of
$(X,0)$ to symplectic or Stein fillings of the link $(Y, \xi)$. A {\em smoothing} is given by a deformation of $X$ to a surface (the Milnor fiber) that is no longer singular. (We discuss smoothings in more detail in Section~\ref{s:picdef}.) 
Milnor fibers themselves are Stein fillings of $(Y, \xi)$, called \emph{Milnor fillings}. An additional Stein filling can
be produced by deforming the symplectic structure on the minimal resolution of $(X, 0)$ \cite{BO}.
For rational singularities, this filling agrees with the Milnor fiber of the Artin smoothing component and need not be considered separately
(see Section~\ref{s:artin}).  An interesting question is whether the collection of these expected fillings, taken for all singularities with the same link $(Y, \xi)$, gives all possible Stein fillings of the link.
In this article, we will use the term \emph{unexpected Stein filling}  to refer to any Stein filling which does not arise as a Milnor fiber 
or the minimal resolution. 

There are very few examples of unexpected Stein fillings in the previously existing literature, none of which are simply-connected. In this article, 
we show that, in fact, unexpected Stein fillings are abundant, and in many cases simply-connected, even for the simple class of \emph{rational singularities with reduced fundamental cycle}. 
These singularities, also known as \emph{minimal singularities} \cite{Ko}, 
can be characterized by the conditions that the dual resolution graph is a tree, where each vertex $v$ corresponds to a curve of genus 0, 
and its self-intersection $v \cdot v$ and valency $a(v)$ satisfy the inequality $ - v \cdot v \geq a(v)$. (See Section~\ref{s:picdef} for more details.)
In low-dimensional topology, such graphs are often referred to as {trees} with no bad vertices. The corresponding plumbed 3-manifolds are $L$-spaces, i.e. they have the simplest possible Heegaard Floer homology~\cite{OSz}. In a sense, links of rational singularities with reduced fundamental cycle are just slightly more complicated than lens spaces. As another measure of low complexity, these contact structures admit planar open book decompositions. In the planar case, the set of Stein fillings satisfies 
a number of finiteness properties \cite{Stip, Plam, Kalo, LiWe}, which makes it rather surprising that these singularities diverge from the expected.

We construct many specific examples of unexpected Stein fillings for rational singularities with reduced fundamental cycle.
Then we show that our examples can be broadly generalized to apply to a large class of singularities with reduced fundamental cycle:
we only require that the resolution graph of the singularity contain a certain subgraph to ensure that the link has many unexpected Stein fillings. 

\begin{theorem} \label{thm:intro-examples} For any $N>0$, there is a rational singularity with reduced fundamental cycle whose contact link $(Y_N, \xi_N)$ admits at 
least $N$ pairwise non-homeomorphic simply-connected Stein fillings, none of which is diffeomorphic to a Milnor filling ({\em rel} certain boundary data). 
Examples of such $(Y_N,\xi_N)$ include Seifert fibered spaces over $S^2$ corresponding to certain star-shaped resolution graphs. 

The statement also  holds for any rational singularity with reduced fundamental cycle 
whose resolution graph has a star-shaped subgraph as above.
\end{theorem}

More precise statements are given in Section~\ref{s:examples}. Our first example which admits simply-connected unexpected Stein fillings 
corresponds to the singularity with resolution graph in Figure~\ref{fig:exampleintro}. More generally, we can find $N$ distinct unexpected
Stein fillings for singularities whose dual resolution graph is star-shaped with at least $2N+5$ sufficiently long legs, $N\geq 4$, the self-intersection of the central vertex is a large negative number, and the self-intersection of any other vertex is $-2$.

\begin{figure}[htb]
	\centering
	\includegraphics[scale=1]{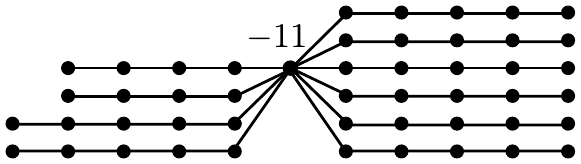}
	\caption{A resolution graph for a singularity whose link admits simply-connected unexpected fillings. (Unlabeled vertices have self-intersection $-2$.) 
	Any graph containing this as a subgraph corresponds to a singularity which also admits simply-connected unexpected fillings.}
	\label{fig:exampleintro}
\end{figure}

By contrast, previous results have indicated that for simple classes of singularities,  
all Stein fillings come from Milnor fibers or the minimal resolution (there are no unexpected fillings).
This is true for $(S^3,\xi_{std})$ \cite{Eliash}, for links of simple and simple elliptic  singularities \cite{OhtaOno1, OhtaOno2},   for lens spaces (links of  
cyclic quotient singularities) \cite{Li, NPP-cycl}, and in general for quotient singularities \cite{PPSU,Bhu-Ono}.
Theorem~\ref{thm:intro-examples} breaks this pattern and provides many unexpected fillings. 
However, we are also able to show that certain classes of rational singularities with reduced fundamental cycle do not admit any unexpected fillings:

\begin{theorem} \label{kollar-fill}
Let $(X, 0)$ 
be a rational singularity with reduced fundamental cycle with link $(Y,\xi)$, and suppose that each exceptional curve in its minimal resolution has self-intersection at most $-5$.  
Then the resolution of $(X,0)$ is the unique weak symplectic filling of $(Y,\xi)$, up to blow-up, symplectomorphism and symplectic deformation. 
\end{theorem}

This theorem proves a symplectic generalization of~\cite[Theorem 6.21]{dJvS}, which establishes a special case of a conjecture of Koll\'ar, 
showing that for singularities as in Theorem~\ref{kollar-fill}, the base space of a semi-universal deformation has one component. Thus, they show there is a unique smoothing, whereas we generalize this to show there is a unique minimal symplectic filling. 
To prove Theorem~\ref{kollar-fill}, we build on the combinatorial argument of \cite{dJvS} and use mapping class group arguments 
to establish the symplectic case.

The bound of $-5$ on the self-intersection of the exceptional curves in Theorem~\ref{kollar-fill} cannot generally be improved.  Indeed, 
any singularity whose minimal resolution contains a sphere of self-intersection $-4$ has
at least two distinct Stein fillings, because a neighborhood of the $(-4)$ sphere can be rationally blown down to produce another
filling with smaller Euler characteristic \cite{Sym}.
This corresponds to the fact that the singularity has at least two smoothing components if a $(-4)$ sphere is present,~\cite{Kol2}; generally, we do not know if unexpected fillings exist in this case.  


Theorem~\ref{kollar-fill} extends the list of singularities with no unexpected Stein fillings. However, 
when complexity of the singularity increases, one should expect the unexpected: as predicted in \cite{Nem2}, more complicated singularities are likely to have 
Stein fillings that do not arise from Milnor fibers.  To our knowledge, the only previous examples of unexpected Stein fillings in the literature  
are detected by their first Betti number. By \cite{GreSt}, Milnor fibers for normal surface singularities always have $b_1=0$. 
An infinite family of Stein fillings with $b_1\neq 0$  was given in \cite{Akh-Ozbagci1, Akh-Ozbagci2} for links of certain non-rational singularities;
these links are Seifert fibered spaces over higher genus surfaces. It follows from \cite{Akh-Ozbagci1, Akh-Ozbagci2} that most of these fillings are different 
from both the Milnor fibers and the resolution of the singularity. The constructions in these papers use surgeries and produce infinite families of exotic fillings 
(which are all homeomorphic but pairwise non-diffeomorphic). Note that for rational singularities, the first Betti number cannot detect 
unexpected fillings: the link is a rational homology sphere, and a homology exact sequence argument 
shows that $b_1=0$ for any Stein filling (see Remark~\ref{b1=0}).

{Note that, in general, known results allow to find many non-rational singularities whose links have infinitely 
many Stein fillings. As an example, consider a normal surface singularity whose resolution has a unique exceptional curve of genus $g \geq 2$ with self-intersection $-d$, for $d > 0$. The resolution is the total space of the complex line bundle of degree $d$ over the corresponding Riemann surface, and the singularity can be thought of as cone point. If $g=\frac12(d-1)(d-2)$, this is the cone over the smooth projective surface of degree $d$ in $\C^3$, so $(X_d, 0)$ is simply the hypersurface singularity $x^d+y^d+z^d=0$.  For each $d\geq 5$, the results of \cite{BMVHM} produce arbitrarily long positive factorizations of the corresponding open book monodromy, which in turn yields infinitely many Stein fillings for the link $(Y_d, \xi_d)$; in particular, there are Stein fillings with arbitrarily large $b_2$. One might hope that most of these Stein fillings are unexpected: indeed, a hypersurface singularity has a unique Milnor fiber, and its topology is well understood~\cite{Milnor, Tju}. 
However, the question is more subtle: because $(X_d, 0)$ is not {\em (pseudo)taut} \cite{Lauf}, there are infinitely many singularities with the same link $(Y_d, \xi_d)$. Milnor fibers of these singularities may yield additional Stein fillings. Describing all such Milnor fibers seems to be out of reach; conceivably, they may produce all the Stein fillings given by the arbitrarily long factorizations of~\cite{BMVHM}. We will discuss
related questions in more detail in Section~\ref{s:artin}, although we do not have any answers for this case.}



Our present work gives the first examples of unexpected Stein fillings for rational singularities, 
{and} for the case where the link $Y$  is a rational homology sphere. In the case of rational singularities, the fillings must be differentiated from Milnor fibers by more subtle means than $b_1$, as all Stein fillings have $b_1=0$ in this case. For singularities with reduced fundamental cycle, the contact link admits a planar open book decomposition, \cite{NT, Scho}. By \cite{Plam, Kalo}, it follows that the number of Dehn twists in any positive monodromy factorization, and thus $b_2$ of Stein fillings, is bounded
above. This means that we cannot generate unexpected fillings by arbitrarily long positive factorizations.
{On the other hand, even though there is typically an infinite collection of singularities with the given link, the reduced fundamental cycle hypothesis, together with the De Jong--Van Straten theory, gives us certain control 
over the topology of all possible Milnor fibers.}

In general, comparing Stein fillings to Milnor fillings is a  two-fold challenge: classification is typically out of reach, both on the deformation theory side 
(smoothings and Milnor fibers) and on the symplectic side (Stein fillings). In the particular case of rational singularities with reduced fundamental cycle,
two important tools facilitate the study of fillings. On the algebraic geometry side, de Jong and van Straten reduce the study of deformations of the surface
to certain deformations of a decorated germ of a reducible singular complex curve $\mathcal{C} \subset \C^2$. 
(The germ $\mathcal{C}$ is associated to the surface as explained in Section~\ref{s:picdef}. For now, we omit the decoration from notation.) 
The construction of \cite{dJvS} works for a more general class of {\em sandwiched} rational singularities; in the case of reduced fundamental 
cycle, the associated plane curve germ has smooth irreducible components. Thus in this case, 
$\mathcal{C}$ is simply the union of smooth complex disks $C_1, C_2, \dots, C_m$, all passing through 0.
The decoration of the germ  is given by marked points, initially concentrated at the origin. 
To encode deformations of the surface singularity, one considers 1-parameter $\delta$-constant deformations of  $\mathcal{C}$, where the 
marked points are redistributed so that all singularities of the deformed curve $\mathcal{C}^s$ are marked (additional ``free'' marked points are also allowed). 
Smoothings of the corresponding singularities are given by {\em picture deformations}, where the only singularities of the deformed 
curve are transverse multiple points. While picture deformations are still hard to classify directly and thus rarely give explicit classification of smoothings, 
they do provide a lot of useful information. In certain examples, they allow us to understand the topology of Milnor fibers and compute their basic invariants.

The following theorem summarizes de Jong--van Straten's results that we use. Detailed definitions and precise statements will be given in Section~\ref{s:picdef}. 
\begin{theorem} \cite[Theorem 4.4, Lemma 4.7]{dJvS} \label{thm:intro-djvs} Let $(X, 0)$ be a rational singularity with reduced fundamental cycle, and $\mathcal{C}\subset \C^2$ its decorated germ 
of a reducible complex curve such that all the branches $C_1, \dots, C_m$ of $\mathcal{C}$ are smooth complex disks. Then smoothings of $(X, 0)$ are in one-to-one correspondence with 
{\em picture deformations} of  $\mathcal{C}$. A picture deformation gives an arrangement $\mathcal{C}^s$  of the deformed branches $C_1^s, \dots , C_m^s$, $s\neq 0$,
with marked points that include all the intersections of the branches. The Milnor fiber  $W_{\mathcal{C}^s}$ of the corresponding smoothing
can be constructed by blowing up at all marked points and taking the complement of the proper 
transforms of $C_1^s, \dots, C_m^s$. 
\end{theorem}

The Milnor fibers described in Theorem~\ref{thm:intro-djvs} are non-compact, but a slight modification yields  compact Milnor fillings of the contact link $(Y, \xi)$ of $(X,0)$.
We consider the germ $\mathcal{C}$ in a small closed ball $B\subset \C^2$ centered at $0$, such that all the branches of $\mathcal{C}$, and thus all the deformed branches for small $s$,
intersect $\partial B$ transversely, and $B$ contains all marked points. To obtain a smooth compact  4-manifold whose boundary is the link $Y$, we
blow up $B$ at the marked points, take  the complement of disjoint tubular neighborhoods of 
the proper transforms of  $C_1^s, \dots, C_m^s$, and smooth the corners.

In turn, on the symplectic side,  contact links of singularities with reduced fundamental cycle are more accessible because they are supported by planar open books, 
\cite{NT, Scho}. By a theorem of Wendl \cite{We}, all Stein fillings of a planar contact manifold are given by Lefschetz 
fibrations whose fiber is the page of the open book. In other words, all these Lefschetz fibrations arise from  factorizations of the monodromy 
of the {\em given} open book into a product of positive Dehn twists. In most cases, such positive factorizations cannot be explicitly classified, but
they give a combinatorial approach to Stein fillings. 

To relate the two sides of the story, we generalize the notion of picture deformation and consider {\em smooth graphical homotopies}
of the decorated germ $\mathcal{C}$ with smooth branches.  A smooth graphical homotopy of $\mathcal{C}$ 
is a real 1-parameter family of embedded disks $C_1^t,\dots, C_m^t$ such that for $t=0$ the disks $C_1^0,\dots, C_m^0$ are the branches of  $\mathcal{C}$, and 
for $t=1$, the intersections between $C_i^1$ and $C_j^1$ 
are transverse and positive for all $i,j$.  There is a collection of marked points on $C_1^1, \dots, C_m^1$, coming from a redistribution of the decoration on $\mathcal{C}$, 
such that all intersection points are marked. (See Definition~\ref{def:homotopy}.)

We prove that just as picture deformations yield smoothings in~\cite{dJvS}, every smooth graphical homotopy gives rise to a Stein filling naturally supported by a Lefschetz fibration.

\begin{theorem} \label{thm:Lefschetz} Let $(Y, \xi)$ be the contact link of a singularity $(X,0)$ with reduced fundamental cycle,
and let $\mathcal{C}$ be a decorated plane curve germ representing $(X, 0)$, with $m$ smooth components $C_1^0, \dots, C_m^0$.
For any smooth graphical homotopy, let $W$ be  the smooth 4-manifold obtained by blowing up at all marked points 
and taking the complement of the proper transforms
of $C_1^1, \dots, C_m^1$. (In the case of a picture deformation $\mathcal{C}^s$, $W$ is the Milnor fiber $W_{\mathcal{C}^s}$ from Theorem~\ref{thm:intro-djvs}). 

Then $W$  carries a planar Lefschetz fibration that supports a Stein filling of $(Y,\xi)$. 
When $W=W_{\mathcal{C}^s}$, the Lefschetz fibration is compatible with the Stein structure on the Milnor fiber.

The fiber of the Lefschetz fibration on $W$ is a disk with $m$ holes,
and the vanishing cycles can be computed directly from the decorated curve configuration $C_1^1,\dots, C_m^1$.
On $(Y, \xi)$, the Lefschetz fibration induces a planar open book decomposition, which is independent of the smooth graphical homotopy of the given decorated 
germ $\mathcal{C}$. 
\end{theorem}

Each rational singularity with reduced fundamental cycle has a distinguished \emph{Artin} smoothing component, which  
corresponds to a picture deformation called the \emph{Scott deformation} (see Section~\ref{s:artin}). 
Applying Theorem~\ref{thm:Lefschetz} to the Scott deformation yields a planar Lefschetz fibration filling $(Y,\xi)$ where the vanishing cycles are disjoint (see Proposition~\ref{artin-ob}).
This gives a natural model for the planar open book decomposition on $(Y,\xi)$. 
This open book is closely related to the braid monodromy of the singularity of $\mathcal{C}$.
Note that we need to consider  all singularities topologically equivalent to $(X,0)$ to describe all Milnor fillings for $(Y, \xi)$, 
since all such singularities have the same contact link. However, topologically equivalent singularities can be represented by topologically equivalent decorated germs and 
produce the same open book decompositions.

The process of computing the monodromy factorization resembles a known strategy for monodromy calculation for a plane algebraic curve \cite{MoTeI,MoTeII}. 
The necessary information can be encoded by a {\em braided wiring diagram} given by the intersection of $\mathcal{C}^s$ with a suitably chosen
copy of $\C\times \R \subset \C^2$. 

A reversal of the above constructions allows us to represent Stein fillings of $(Y, \xi)$  via arrangements of symplectic curves, as follows. 
Let $W$ be an arbitrary Stein filling of the link $(Y, \xi)$. We fix an open book for $(Y, \xi)$ defined by the germ $\mathcal{C}$ as above.
By Wendl's theorem, $W$ can be represented by a Lefschetz fibration with the planar fiber given by the page. 
The Lefschetz fibration corresponds to a factorization 
of the open book monodromy into a product of positive Dehn twists. We reverse-engineer a braided wiring diagram producing this factorization, and then use the diagram
to construct an arrangement $\Gamma$ 
of symplectic disks. (In fact, an arrangement of smooth graphical disks is sufficient for our constructions, but the symplectic condition can be satisfied
at no extra cost.) We require that the disks intersect transversally (multiple intersections are allowed), and equip $\Gamma$ with a collection of marked points that include all 
intersections and possibly additional ``free'' points. We also show that the resulting arrangement of disks and points is related to the decorated germ $\mathcal{C}$ by 
a smooth homotopy, which is graphical in suitable coordinates. (The homotopy moves the disks and the marked points.) This yields a symplectic analog of 
Theorem~\ref{thm:intro-djvs}.

\begin{theorem} \label{thm:intro-symp} Let $(Y, \xi)$ be the contact link of a singularity $(X,0)$ with reduced fundamental cycle that corresponds to a decorated plane curve germ 
$\mathcal{C}$.  Then any Stein filling of $(Y,\xi)$ arises from an arrangement $\Gamma$ of symplectic graphical disks with marked points, as in Theorem~\ref{thm:Lefschetz}.
The arrangement~$\Gamma$ is related to the decorated germ $\mathcal{C}$ by a smooth graphical homotopy. 
%
%
%
%
\end{theorem}

Theorems~\ref{thm:intro-djvs} and~\ref{thm:intro-symp} mean that both Milnor fibers and arbitrary Stein fillings of a given link of rational singularity with reduced fundamental cycle can be constructed in a similar way, starting with the  
decorated plane curve germ $\mathcal{C}$ representing the singularity. Milnor fibers arise from algebraic picture deformations of the branches of $\mathcal{C}$, 
while Stein fillings come from smooth graphical homotopies of the branches.


Once the comparison of Milnor fibers and Stein fillings is reduced to comparison of
arrangements of complex curves or smooth disks with certain properties, we can construct examples 
of arrangements that generate Stein fillings not diffeomorphic to Milnor fibers. 
We need arrangements that are related to a particular plane curve germ by a smooth graphical homotopy but 
not by an algebraic picture deformation. We build {\emph{unexpected line arrangements} satisfying this property} in Section~\ref{s:examples}, using classical projective geometry and a study of analytic deformations. 
We use these to construct unexpected Stein fillings; then we verify that they are not diffeomorphic (relative to the boundary open book data) to Milnor fillings by an argument based on \cite{NPP}.
This leads to the proof of Theorem~\ref{thm:intro-examples} and other similar examples.

At first glance, the difference between algebraic and smooth plane curve arrangements seems rather obvious.  However, because we are in 
an open situation, working with germs of curves and smooth disks with boundary as opposed to closed algebraic surfaces, the question is quite subtle. 
In particular, we cannot simply use known examples of topological or symplectic line arrangements in $\cptwo$ not realizable by complex lines. 
Indeed, in many cases the smooth surfaces can be closely approximated by high-degree polynomials, so that a Lefschetz fibration on the corresponding Stein filling 
can be realized by a Milnor fiber. We discuss the relevant features of the picture deformations and smooth (or symplectic) graphical homotopies in detail in Section~\ref{s:further}, 
and explain what makes our examples work.  

It is worth stating that while Stein fillings and Milnor fillings are the same for certain small families of singularities, the two notions are in fact fundamentally 
different. A Milnor filling is given by a smoothing of a singular complex surface, so there is a family of Stein homotopic fillings of $(Y, \xi)$ that degenerate to
the singular surface. A Stein filling of the link has no {\em a priori}  relation to the singular surface and is not part of any such family. This distinction becomes  apparent in our 
present work, by the following heuristic reasoning. A picture deformation $\mathcal{C}^s$ of the decorated germ $\mathcal{C}$ gives, for any $s\neq 0$, a Milnor filling 
$W_{\mathcal{C}^s}$, so that all these fillings are diffeomorphic and even Stein homotopic. The Milnor fillings look the same for all $s \neq 0$ because 
the arrangements of deformed branches $\{C_1^s, \dots, C_m^s\}$ have the same topology. By contrast, if the germ $\mathcal{C}$  is homotoped via a family of 
smooth disk arrangements~$\Gamma^t$, the topology of the arrangement $\{\Gamma_1^t, \dots, \Gamma_m^t\}$ may change during the homotopy.
Under certain conditions we can construct a family of Lefschetz fibrations $W_t$ that includes the given Stein filling 
and changes its diffeomorphism type at finitely many discrete times as it connects to the minimal resolution.
In other cases, at some time $t$  the homotopy gives an arrangement $\Gamma^t$ which produces an {\em achiral} Lefschetz fibration, so the 4-manifolds in the corresponding family do not necessarily carry a Stein structure. 
We return to this discussion in Section~\ref{s:further}.


One can also ask whether unexpected fillings exist for rational singularities with reduced fundamental cycle that are not covered by Theorem~\ref{thm:intro-examples} or
Theorem~\ref{kollar-fill}. For certain additional simple examples, we can use Theorem~\ref{thm:intro-symp} and pseudoholomorphic curve 
arguments to verify that there are no unexpected fillings,
even though the smoothing may not be unique. This approach only works when the germ of the singularity is a pencil of lines satisfying certain restrictive constraints.
Namely, we can consider
(1)~arrangements of 6 or fewer symplectic lines, or
(2)~arrangements of symplectic lines where one of the lines has at most two marked points where it meets all the other lines in the arrangement.
Since the boundary behavior of symplectic lines is controlled, we can cap off symplectic lines in a ball to symplectic projective lines in $\cptwo$, together with the line at infinity. The corresponding arrangements in $\cptwo$ are shown to have a unique symplectic isotopy class and are symplectically isotopic
to an actual complex algebraic line arrangement in $\cptwo$, \cite[Lemma 3.4.5]{StarkstonThesis}. It follows that every symplectic arrangement 
as above can be obtained as picture deformation of a pencil of complex lines, and therefore, the corresponding Stein fillings are given by Milnor fibers.
The links of the corresponding singularities are Seifert fibered spaces for which Stein fillings were completely classified and
presented as planar Lefschetz fibrations in \cite[Chapter 4]{StarkstonThesis}. The line arrangements appearing in that
classification precisely coincide with the symplectic disk arrangements from the perspective of this article. (Here,
gluing in the deleted neighborhood of the disk provides an embedding of the Stein filling into a blow-up of $\C^2$. In \cite{StarkstonThesis},
gluing on the cap, which augments the configuration of lines by the additional line at infinity, provides an embedding of the Stein filling in a blow-up of $\cptwo$.)
In general, Theorem~\ref{thm:intro-symp} seems to have limited applications to classification of fillings, due to complexity of arrangements of curves.

It is interesting to note that while de Jong--van Straten describe deformations of sandwiched singularities, our constructions only work for the subclass of rational singularities 
with reduced fundamental cycle. Indeed, a planar open book decomposition of the contact link plays a key role in our work because we need Wendl's theorem to describe Stein fillings.
By~\cite{GGP} the Milnor fillable contact structure on the link of a normal surface singularity is planar {\em only if} the singularity is rational 
and has reduced fundamental cycle. This means that our methods in the present paper cannot be used for classification for any other  surface singularities. However, for future work, we are investigating extensions of these methods to produce examples of unexpected fillings for more general surface singularities. Finally, recall that 
 all weak symplectic fillings of a planar contact 3-manifold are in fact given by planar Lefschetz fibrations, up to blow-ups and symplectic deformation \cite{NiWe}. 
 It follows that Theorem~\ref{thm:intro-symp} and related results apply to describe all minimal weak symplectic fillings. 
 However, we focus on Stein fillings and will give all statements, with the exception of Theorem~\ref{kollar-fill}, 
 only for the Stein case. 

\subsection*{Organization of the paper} In Section~\ref{s:picdef} we review the definitions of rational singularities with reduced fundamental
cycle as well as their deformation theory from \cite{dJvS}, and prove some of their properties from the topological perspective.
In Section~\ref{s:homotopy} we prove the first direction of the symplectic correspondence, namely Theorem~\ref{thm:Lefschetz}.
In Section~\ref{s:artin} we explain the smoothing in the Artin component from the perspective of symplectic topology, discuss the corresponding open books, and also {raise some questions related to open book factorizations and non-rational singularities}. 
In Section~\ref{find-curvettas-for-filling} we prove the other half of the correspondence, establishing Theorem~\ref{thm:intro-symp} 
using braided wiring diagrams and Wendl's theorem~\cite{We}. In Section~\ref{topology} we prove Theorem~\ref{kollar-fill} and explain how to calculate algebraic topological
invariants of the fillings, which we will use to distinguish our examples of unexpected Stein fillings from Milnor fillings.
In Section~\ref{s:examples} we prove that there are many examples of unexpected Stein fillings for links of rational surface 
singularities with reduced fundamental cycle, establishing Theorem~\ref{thm:intro-examples}. Finally, in Section~\ref{s:further} 
we explain what key differences between picture deformations and smooth graphical homotopies contributed to the distinction between expected and unexpected Stein fillings.

\subsection*{Acknowledgments} We are grateful to Stepan Orevkov for suggesting Example~\ref{example-orevkov} to us. 
This example played a crucial role in our understanding of arrangements that produce unexpected fillings. 
We thank Roger Casals, Eugene Gorsky, and Marco Golla for their interest to this project and numerous motivating and 
illuminating discussions at the early stages of this work. In particular, Eugene helped us understand some of the results of \cite{dJvS}.
We are grateful to Eugene and Marco for their comments on the preliminary version of this article, and to Jonathan Wahl, Jeremy Van Horn-Morris, and Patrick Popescu-Pampu for interesting correspondence. Many thanks to Inanc Baykur for illuminating correspondence and discussion on the higher genus case.
LS would also like to thank Sari Ogami who learned and explained 
to me a great deal about the monodromy of braided wiring diagrams. OP is also grateful to John Etnyre, Jonny Evans,  Mark McLean, and Oleg Viro for a few helpful discussions. We thank the referees for their thoughtful comments and suggestions. 
OP was supported by NSF grants DMS-1510091 and DMS-1906260 and a Simons Fellowship. LS was supported by NSF grant DMS-1904074.  

\section{Rational singularities with reduced fundamental cycle,  \\ their decorated curve germs, and relation to deformations} \label{s:picdef}

In this section, we collect some facts about rational singularities with reduced fundamental cycle and state de Jong--van Straten's results on their smoothings, 
\cite{dJvS}. De Jong--van Straten's results are in fact more general: they fully describe deformation theory for a wider class of {\em sandwiched} singularities. We state only 
the results we need. Some of our statements are slighly different from \cite{dJvS}: we describe their constructions from the topological
perspective and set the stage for our work.
Although we aim for a mostly self-contained discussion, the reader may find it useful to consult~\cite{Nem} 
for a general survey on topology of surface singularities. {The survey~\cite{PPP} focuses on the interplay between singularity theory and contact topology and provides very helpful background.} 
Additionally, a brief survey of the key results of~\cite{dJvS} from the topological perspective can be found in~\cite{NPP}.

\subsection{Resolutions and smoothings.} We begin with some general facts about surface singularities.
{Let $(X, 0)$ be a normal surface singularity. Its resolution $\pi: \ti{X} \to X$ is a proper birational morphism such that $\ti{X}$ is smooth. The {\em exceptional divisor} $\pi^{-1} (0)$ is the inverse image of the singular point. For a given singularity $(X, 0)$, the resolution is not unique, as one can aways make additional blow-ups; however, for a surface singularity, there is a unique {\em minimal} resolution \cite{Lauf2}. 
The minimal resolution is characterized by the fact that $\ti X$ contains no embedded smooth complex curves of genus 0 and self-intersection $-1$ (thus it does not admit a blow-down).} 

{After performing additional blow-ups if necessary, we can assume that 
the exceptional divisor $\pi^{-1} (0)$ has normal crossings. This means that $\pi^{-1} (0)= \cup_{v \in G} E_v$, where
the irreducible components $E_v$ are smooth complex curves that intersect transversally at double points only. 
A resolution with this property is  called a {\em good} resolution. For a surface singularity, a minimal good resolution is also unique~\cite{Lauf2}.}

The topology of a {good} resolution is encoded by the (dual) resolution graph $G$.  
The vertices  $v\in G$ correspond to the exceptional curves $E_v$ and are weighted by the {genus and} self-intersection $E_v \cdot E_v$ of the corresponding curve. We will often refer 
to $E_v \cdot E_v$   as the self-intersection of the vertex $v$ and use notation $v \cdot v$ for brevity. 
The edges of $G$ record intersections of different irreducible components. 
Note that the link of the singularity is the boundary of the plumbing of disk bundles over surfaces according to $G$.
{In this paper, we focus on {\em rational} singularities; in this case
 $G$ is always a tree, and each exceptional 
curve $E_v$ has genus 0. {(Genus $0$ curves are also called \emph{rational curves}.)} Therefore we will typically omit the genus from the markings on the vertices and only record the self-intersection numbers.}
  
{The link of the singularity determines the dual graph of the minimal good resolution, and vice versa. 
By a result of W. Neumann~\cite{Neu}, the links of two normal surface singularities have the same oriented diffeomorphism type if and only if their dual resolution 
graphs are related  by a finite sequence of blow-ups/blow-downs along rational $(-1)$ curves. Moreover, 
the links of two normal surface singularities are orientation-preserving diffeomorphic if and only if their  minimal good resolutions have the same dual graphs. Minimal good resolutions are easy to recognize: if a good resolution is not minimal, its graph will have a vertex representing a genus 0 curve with self-intersection $-1$.
(This follows from~\cite{Neu}; see also~\cite[Lemma 5.2]{GGP} for a direct proof 
that any possible blow-downs can be seen directly from the graph.)}   
  
{The local topological type of the singularity $(X, 0)$ can be understood from its link $Y$, as a cone on the corresponding 3-manifold. We will say that two singularities are {\em topologically equivalent} if they have the same link. It is important to note that the analytic type of the singularity is not uniquely determined by the link; typically, 
many analytically different singularities have diffeomorphic links.
It is known that the canonical contact structures are all isomorphic for different singularities of the same topological type~\cite{CNPP}; thus, the dual resolution graph encodes the canonical contact structure. Indeed, 
this contact structure can be recovered as the convex boundary of the plumbing, according to the graph, of the standard neighborhoods of the corresponding symplectic surfaces.}

{We now turn attention to deformations and Milnor fibers.} {A deformation of a surface singularity $(X,0)$ is any flat map $\lambda: (\mathcal{X},0)\to (\mathcal{T},0)$ 
such that $\lambda^{-1}(0) = (X,0)$. A versal (or semi-universal) deformation $f:(\mathcal{X},0)\to (B,0)$ parameterizes 
all possible deformations of $(X,0)$. The base space $(B,0)$ generally has multiple irreducible components, which may have different dimensions.
It is generally difficult to understand the space $B$, its irreducible components, and the dimensions of these components.

A deformation $\lambda: (\mathcal{X},0)\to (D,0)$ over the disk in $\C$ is called a (1-parameter) \emph{smoothing} of
$(X,0)$ if $X_s:=\lambda^{-1}(s)$ is smooth for all $s\neq 0$. 
{For any smoothing all such $X_s$ are diffeomorphic,
and we call $X_s$ the \emph{Milnor fiber} of the smoothing. 
For example, for a hypersurface $X=\{f(x,y,z)=0\} \subset \C^3$ with $f(0)=0$ and $df(0)=0$, a smoothing of the singularity at $0$ 
can be given by $f:\C^3\to \C$, with Milnor fiber  $X_\epsilon =\{f(x,y,z)= \epsilon\}$ for a small $\epsilon\neq 0$.}
Each Milnor fiber is endowed with a Stein structure,
and for different $t_0,t_1\in D\setminus 0$, $X_{s_0}$ and $X_{s_1}$ are Stein homotopic (the Stein homotopy is obtained
by choosing a path from $s_0$ to $s_1$ in $D$ which avoids $0$). 

We will consider the compact version of the Milnor fiber to obtain a Stein filling 
of the link $Y$ of $(X,0)$, as follows. For a sufficiently small radius $r>0$,  the surface $X \subset \C^N$ is transverse to the sphere $S^{N-1}_r$. We fix an
ball $B^N_r\subset \C^N$ centered at $0$, sometimes called {\em a Milnor ball}, and consider $X \cap B^N_r$ as the {\em Milnor representative} of $X$. The boundary  
$\partial(X_r \cap B^N_r)$ is the link $Y$ of $(X,0)$, and the complex  structure on $X$ induces the canonical contact structure $\xi$ on $Y$. For sufficiently small $s \neq 0$, 
the Milnor fiber $X_s$ also intersects $S^{N-1}_r$ transversally. Then $X_s \cap B^N_r$ is a compact 4-manifold with boundary,  $\partial (X_s \cap B^N_r)$ is diffeomorphic  to $Y$, 
and the contact structure induced by the Stein structure on $X_s$ is isomorphic to $\xi$, so that  $X_s \cap B^N_r$ is a Stein filling of $(Y, \xi)$.

For  a semi-universal deformation $f: (\mathcal{X},0) \to (B,0)$ of the surface singularity $(X,0)$, an irreducible component $B_i$ of $B$ 
is called a \emph{smoothing component} of $(X,0)$ if the general fiber over $B_i$ is smooth. We note that $B_i$ may have lower
(complex) dimensional strata where the fibers over these strata are not smooth. For example, these non-general strata could arise
from singularities in the component $B_i$ or intersections of $B_i$ with other irreducible components of $B$. Nevertheless, these non-general strata have positive complex co-dimension, so the subset of $B_i$ over which the fiber is smooth will be connected. Any $1$-parameter smoothing of $(X,0)$ lies in a unique smoothing component $B_i$.

In general, not every surface singularity admits a smoothing. However, for rational singularities every irreducible component of $B$ is a smoothing component.
Moreover, there is one distinguished component, called the \emph{Artin component}. This component is associated to the minimal resolution
$\widetilde{X}$ of $(X,0)$. (For rational singularities, deformations of $\widetilde{X}$ come from deformations of $(X,0)$ 
and these deformations of $(X,0)$ form the Artin component.) 
We discuss Milnor fibers in this component in greater detail in Section~\ref{s:artin}.}

{In this paper, we study Stein fillings for the contact link $(Y, \xi)$ of a surface singularity, and compare them to Milnor fillings.  As explained above, in general the link determines only the topological, but not the analytic, type of the singularity. Normal surface singularities whose topological type admits a unique analytic type 
are called {\em taut}; if there are only finitely many analytic types, the singularity is {\em pseudotaut}. Taut and pseudotaut singularities were classified by Laufer~\cite{Lauf}: there are several very restrictive lists for the dual resolution graphs, in particular, the graphs cannot have any vertices of valency greater than 3. Thus, most singularities are not (pseudo)taut,  even if we restrict to a  very special kind that we consider in this paper, rational singularities with reduced fundamental cycle.
If we are to compare Stein fillings and Milnor fillings of the link,  
we need to consider Milnor fibers for {\em all} possible singularities of the given topological type. In principle, it is quite possible that topologically equivalent singularities have non-diffeomorphic Milnor fibers: for example, the hypersurface singularities $x^2+y^7+z^{14}=0$ and $x^3+y^4+z^{12}=0$ have the same topological type, but their (unique) Milnor fibers have different $b_2$, \cite{Lauf3}, see also the discussion in \cite[Section 6.2]{PPP}. Fortunately, in the case of reduced fundamental cycle we will have some control over the topology of Milnor fibers for different analytic types, thanks to the de Jong--van Straten construction.}

\subsection{Sandwiched singularities, extended graphs, and decorated germs.}

\begin{definition} \label{def-RFC} $(X, 0)$ is a {\em rational singularity with reduced fundamental cycle} if it admits a normal crossing resolution such that all exceptional curves have genus 0, 
the dual resolution graph $G$ is a tree, and for each vertex $v \in G$,  the valency $a(v)$ of $v$ and the self-intersection $v \cdot v$ satisfy the inequality
\begin{equation} \label{valency-weight}
a(v) \leq - v \cdot v.
\end{equation}
\end{definition}

It follows from~\eqref{valency-weight} that the graph as above can only have vertices with self-intersection $-1$ as the leaves of the tree. Blowing down all such vertices, we obtain a graph that still satisfies~\eqref{valency-weight} and represents the minimal resolution of  $(X, 0)$. 

To explain the terminology of Definition~\ref{def-RFC}, we recall the definition of a fundamental cycle. For a given resolution, consider the set of divisors
{$$
\{ Z= \sum_{v \in G} m_v E_v \mid Z>0,\textrm{ and }  Z \cdot E_v \leq 0 \text{ for all } E_v \}.
$$}
This set has a partial order, defined by    $\sum m_v E_v    \geq \sum n_v E_v$   if $m_v \geq n_v$ for all $v$. There is a minimal element 
with respect to this partial order, denoted $Z_{min}$ and called Artin's fundamental cycle. {Since the resolution dual graph is connected, different components $E_v$ intersect positively, and $Z>0$, any element in the set has $m_v>0$. Therefore, $Z_{min} \geq \sum_{v \in G} E_v$.} It is easy to see 
that $(\sum_{v \in G} E_v) \cdot E_v \leq 0$ for all $E_v$ if and only if Condition~(\ref{valency-weight}) is satisfied. In this case 
$Z_{min}=  \sum_{v \in G} E_v$, and since each exceptional curve enters with multiplicity 1, we say that the fundamental cycle $Z_{min}$ is reduced.  

In \cite{dJvS}, de Jong and van Straten work with {\em sandwiched} singularities. By definition, a sandwiched singularity $(X, 0)$ is analytically isomorphic to the germ of an algebraic singular surface which admits a birational morphism to $(\C^2, 0)$.
For a resolution $\pi: \ti{X} \to X$,  we get a diagram $(\ti{X}, \pi^{-1}(0)) \dashrightarrow (X, 0) \dashrightarrow (\C^2, 0)$. In particular, X is sandwiched between
two smooth spaces via birational maps. Sandwiched singularities are rational and can be characterized by their  
resolution graphs as follows, by translating the sandwiched condition. The graph $G$ is {\em sandwiched} if we can add to it 
a number of edges and their end vertices with self-intersections $(-1)$, so that the resulting graph $G'$ gives a 
plumbing whose boundary represents $S^3$. In other words, $G'$ gives a configuration of rational curves that can be blown down to a smooth point. 
The choice of the graph $G'$ is not unique.  
It is not hard to see that every rational singularity with reduced fundamental cycle is sandwiched.  In Proposition~\ref{p:resolution} below, we discuss in detail the construction of the 
possible graphs $G'$ for this case.

Any sandwiched singularity can be associated to a  (germ of) a complex plane curve singularity, constructed as follows. The choice of the graph $G'$ corresponds 
to an embedding of the tubular neighborhood of the exceptional set of the resolution $\ti{X}$ into some blow-up of $\C^2$.  This blown-up surface also has a distinguished  
collection of $(-1)$ curves, so that the configuration of these $(-1)$ together with the exceptional set can can be completely blown down. For each distinguished $(-1)$ curve, 
choose a transverse complex disk (called a {\em curvetta}) through a generic point. Now, contract the curve configuration corresponding to $G'$. The union of the curvettas 
becomes a germ of a reducible curve $\mathcal{C}$ in $\C^2$, with components passing through 0. Let $C_i$, $i=1, 2, \dots, m$ be the irreducible components of 
$\mathcal{C}$; following \cite{dJvS}, we also refer to $C_i$ as {\em curvettas}. 
We emphasize that only the germ of $\mathcal{C}$ at the origin is defined; when we use the notation
$\mathcal{C}\subset \C^2$, we only consider a small neighborhood of $0\in \C^2$. In particular,  we are only interested in the singularity of the reducible curve 
$\mathcal{C}$ at 0. In this paper, we will focus on the case where the components $C_i$ are smooth at $0$, so that locally $C_i$ is a smooth disk. This suffices to study rational singularities with reduced fundamental cycle, as we will soon see. This disk 
may be locally parameterized by a high-degree algebraic curve in $\C^2$, but the global topology of this curve is unimportant to us, because we only use the part of the curve in a neighborhood of the origin.

\begin{figure}[htb]
		\centering
		\includegraphics[scale=.75]{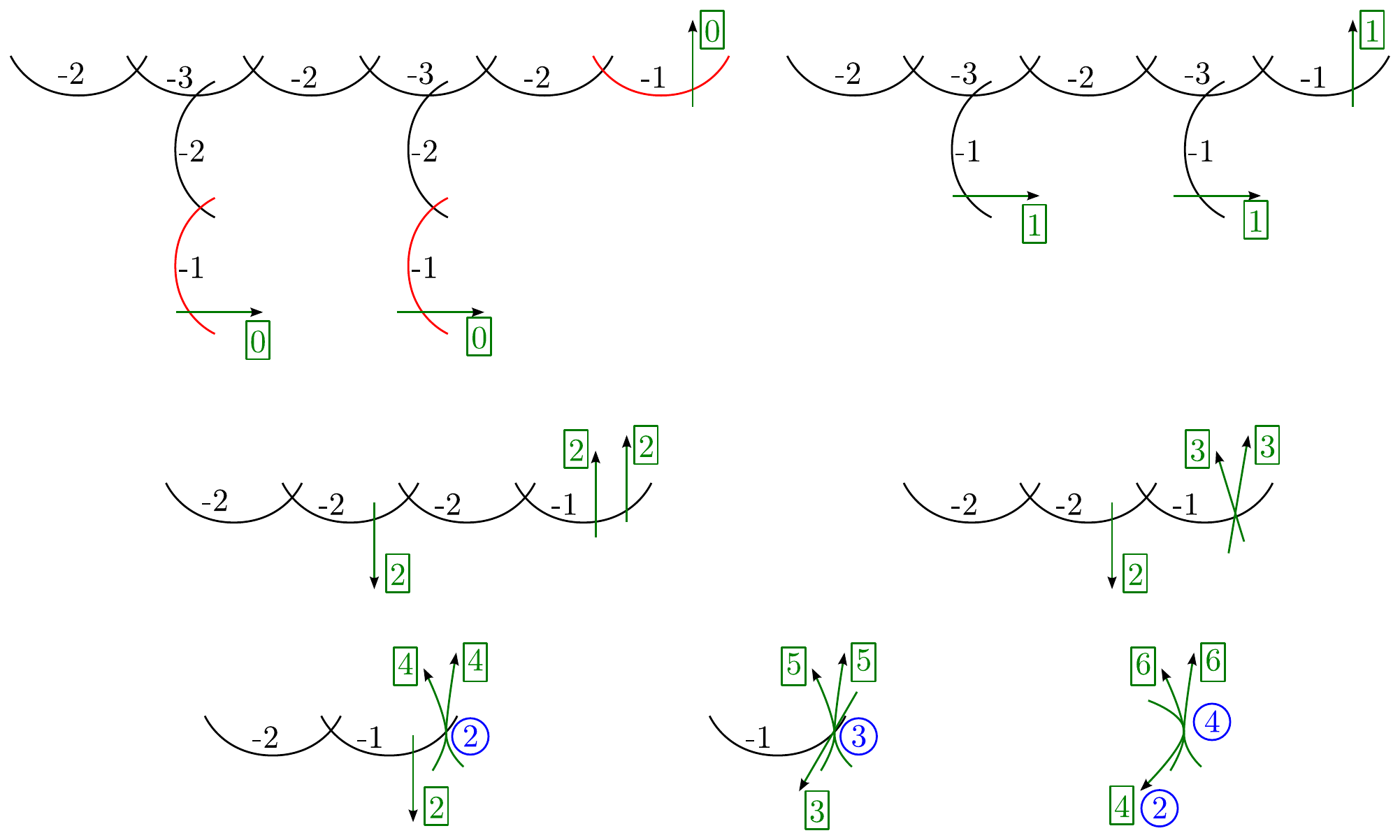}
		\caption{An example of a sandwiched singularity and a choice of corresponding curvettas (green arrows).
		The first diagram shows the resolution curves together with extra (red) $(-1)$ exceptional curves attached.
		Then there is a sequence of blow-downs. We keep track of the weights $w(C_i)$ in rectangular boxes next to each green curvetta arrow.
		The multiplicities of tangencies between bunches are recorded in blue circled numbers.}
		\label{fig:example}
	\end{figure}

Each  curvetta $C_i$  comes with a weight $w_i=w(C_i)$,
given by the number of exceptional spheres that intersect the corresponding curve in the blow-down process from $G'$ to the empty graph. In other words, $w_i$ is the number
of blow-down steps that affect the corresponding curvetta before it become $C_i$. The weighted curve $(\mathcal{C}, w)$ is called a {\em  decorated germ} corresponding 
to $(X, 0)$. An example of this process, and the resulting decorated germ for the given singularity, is shown in Figure~\ref{fig:example}.

It is convenient to start the process with the minimal normal crossings resolution of $(X, 0)$. For rational singularities with reduced fundamental cycle, it is easy to see that the graph of the
minimal normal crossings resolution has no $(-1)$ vertices. (From~(\ref{valency-weight}), 
only vertices of valency 1 can have self-intersection $-1$ in any resolution graph, and these can be blown down to get the minimal graph.) If $G$ has no $(-1)$ vertices, then {\em all} the 
$(-1)$ vertices of $G'$ are those that come from the extension: each $(-1)$ vertex is a leaf of $G'$, connected by an edge to a unique vertex of $G$. The transverse 
curvetta slices are added to all these $(-1)$ vertices.

In what follows, we will only consider decorated germs that arise from the above construction. 
(These are called  {\em standard} decorated germs  in \cite{NPP}. Some statements in~\cite{dJvS} allow for more general decorated germs.)

The singularity $(X, 0)$ can be reconstructed from $(\mathcal{C}, w)$. We iteratively blow up  points infinitely near 0 on proper transforms of curvettas $C_1, \dots, C_k$ until 
we obtain a minimal embedded resolution of $\mathcal{C}$. Then we perform additional blow-ups at the intersection of $C_i$ with the corresponding exceptional curve, so that the sum of multiplicities of proper transforms of 
$C_i$ at the blow-up points is exactly $w_i$. The union of the exceptional curves that {\em do not} meet the proper transforms of the curvettas is then contracted to form 
$(X, 0)$. 

We emphasize that $\mathcal{C}$ depends on the choice of the graph $G'$, i.e. on the particular extension of the resolution graph of $(X, 0)$ by $(-1)$ curves. Any of these choices can be used to classify Milnor fillings as in \cite{dJvS}.
In general, the branches of $\mathcal{C}$ are singular curves. However, if $(X, 0)$ is a rational singularity with reduced fundamental cycle, an appropriate choice of 
$G'$ ensures that $\mathcal{C}$ has smooth branches. We will always work in this setting and only consider decorated germs with smooth components. In the following 
proposition, we establish a necessary and sufficient condition for smoothness purely in terms of the graph $G'$. 
Although similar questions were studied in \cite{dJvS,dJvSmin}, we formulate the condition here in a way that seems simplest from the topological point of view. In the 
next section, we will reinterpret the statement for open book decompositions.

\begin{prop} \label{p:resolution} Let the graph $G'$ be a negative definite plumbing tree, and $P'$ 
the corresponding plumbing of disk bundles over rational curves. Suppose that the boundary of the plumbing $P'$  
is $S^3$; equivalently, $G'$ encodes a configuration of rational curves that can be blown down to a smooth point.
For each $(-1)$ vertex, let $\ti C_j$ be a complex disk intersecting the corresponding $(-1)$ sphere in $P'$ 
transversally once. Let $C_1,\dots, C_m$ be the {images} of $\ti C_1, \dots, \ti C_m$ {under} blowing down the configuration $G'$.
Then the following are equivalent:
	\begin{enumerate}
		\item Each $C_j$ is smooth.
		\item \label{cond:valence} There exists exactly one $v'_0 \in G'$ such that $v'_0 \cdot v'_0 +a(v_0')=-1$,
		 and $v'\cdot v'+a(v')=0$ for all $v'\neq v_0'$.
	\end{enumerate}
As before,  $v' \cdot v'$ denotes 
the self-intersection of a vertex $v' \in G'$, and $a(v')$ its valence.
\end{prop}

\begin{proof}	
	Consider  $\mathcal{C}=C_1 \cup \cdots \cup C_m$ with smooth branches $C_j$.
	We obtain $G'$ as described above, by blowing up repeatedly at 
	intersections of the $C_j$ with each other and with the exceptional divisors. We stop  
	when the resulting configuration of curves has the following property:  if an exceptional divisor intersects a 
	proper transform $\tilde C_j$ then it is disjoint from all other proper transforms $\tilde C_{j'}$, $j'\neq j$
	(in particular, different $\tilde C_j$ are disjoint from each other), 
	and the total number of blow-ups performed on (proper tranforms of) $C_j$ is exactly $w_j$, {the weight on $C_j$}. 

	We will show that $G'$ has the structure of a rooted tree by repeatedly applying the following procedure. For the  root $v_0'$, we 
	will have $v'_0 \cdot v'_0 +a(v_0')=-1$, and for all other vertices $v' \neq v'_0$, $v'\cdot v'+a(v')=0$. We show that this condition is satisfied at every 
	stage of the process.

	Blow up at the common intersection point of all $C_j$. The resulting exceptional divisor (and its future proper transforms) 
	gives the root of the tree. If proper transforms of all $C_j$ still have a common point, we repeatedly blow up at the same point 
	until some of the proper transforms  $\tilde C_j$ become disjoint from each other. (With a slight abuse of notation,  $\tilde C_j$ will denote 
	the proper transform of $C_j$ at any stage of the process.) Additional blow-ups create a chain of exceptional $(-2)$ spheres
	with the root at one end and the most recent exceptional $(-1)$ sphere at the other end.
	 Up to relabeling, we can assume there are distinct intersection points 
	 $\tilde C_1\cap \cdots \cap \tilde C_{a_1}=p_1^1$, $\tilde C_{a_1+1}\cap \cdots \cap \tilde C_{a_2}=p_2^1, \cdots, \tilde C_{a_{r_1}^1}\cap \cdots \cap 
	 \tilde C_m = p_{r_1}^1$ lying on the most recently introduced exceptional divisor $B_1$. 
	 
	 Assuming $m>1$, since all the $\ti C_j$ intersect $B_1$, we must blow up exactly once at each $p_i^1$ to make them all 
	 disjoint from $B_1$. Here we use smoothness of the curvettas $C_j$ (and thus of their proper transforms)
	 to ensure that they become disjoint from $B_1$ after a single blow-up:  every point on $C_j$ has multiplicity $1$, 
	 thus $\tilde C_j$ intersects each exceptional divisor with multiplicity at most $1$.
	 Note that once  $\tilde C_1, \dots, \tilde C_m$ are all disjoint from $B_1$, we will not blow up at any point on $B_1$ again, therefore  at this stage we can already 
	 compute the self-intersection and valency of the correspondung vertex in $G'$.
	 The self-intersection of the proper transform of $B_1$ in $G'$ (which we will also denote $B_1$) is $-r_1-1$. 
	 If $B_1$ is not the root, it has valency $r_1+1$, and if it is the root it has valency $r_1$. 
	 Thus, Condition \ref{cond:valence} is satisfied for the vertex of $G'$ given by $B_1$. 
	 All the other vertices in the graph at this stage are either $(-2)$ spheres in a chain of valency $2$ (if not the root) or 
	 valency $1$ (if the root), or newly introduced $(-1)$ vertices of valency $1$, so  Condition \ref{cond:valence} is satisfied at this stage.
	
	In order to obtain $G'$ we repeat this process iteratively, replacing the first exceptional sphere 
	with the exceptional sphere obtained by blowing up at some $p_i^s$. (The points $p_1^1, \dots, p_{r_1}^1$ were introduced above; after blowing up at each of these new points, 
	the new exceptional curves intersect the proper transforms of the curvettas at points $p_1^2, \dots, p_{r_2}^2$; similarly, points 
	$p_i^s, \dots, p_{r_s}^s$ are the intersections that appear at step $s$.)  Each time, Condition \ref{cond:valence} is preserved, since each curve $\ti C_j$ intersects each exceptional divisor with multiplicity at most $1$. 
	Repeating sufficiently many times, eventually all of the $\tilde C_j$ will intersect only disjoint exceptional spheres.
	After potentially blowing up more times at the intersection of $\tilde C_j$ with its intersecting exceptional sphere until the number of blow-ups is $w_j$, 
	we obtain $G'$. (The additional blow-ups create a chain of $(-2)$ vertices connecting to the last $(-1)$ vertex.)
	 Since Condition~\ref{cond:valence} is preserved at each step of this procedure, $G'$ satisfies Condition~\ref{cond:valence}.
	
	Conversely, if $G'$ satisfies Condition \ref{cond:valence}, the only $(-1)$ vertices are leaves of 
	the rooted tree (valency $1$). Blowing down a leaf preserves Condition \ref{cond:valence} because 
	it decreases the valency of the adjacent vertex by $1$ and increases the self-intersection by $1$. 
	The $\tilde C_j$ are disks which transversally intersect the $(-1)$ leaves of $G'$ with multiplicity $1$. 
	Therefore each $\tilde C_j$ intersects each exceptional divisor with multiplicity at most~$1$. 
	This property is preserved under blowing down a $(-1)$ leaf, because a multiplicity $1$ intersection of $\tilde C_j$ on a $(-1)$ 
	leaf becomes a multiplicity $1$ intersection on the adjacent exceptional divisor after blowing down.
	Blowing down an exceptional divisor which intersects $\tilde C_j$ with multiplicity $1$ preserves smoothness of $\tilde C_j$. 
	Therefore after blowing down all leaves of $G'$ and finally the root, the resulting proper transforms $C_j$ are still smooth.	
\end{proof}

\begin{remark} Another way to see that $G'$ must satisfy Condition $\ref{cond:valence}$ is to consider what happens if $G'$ has a vertex with $a(v')>-v' \cdot v'$. 
	After blowing down, eventually the vertex $v'$ will correspond to a $(-1)$ sphere with valency $\geq 2$, with at least one $\tilde C_j$ intersecting 
	it with multiplicity at least $1$. (The existence of the intersecting $\tilde C_j$ comes from the fact that intersections are transferred under 
	blow-down to the adjacent vertices. Initially, every $(-1)$ sphere in $G'$ has an intersecting curvetta. Each time that a $(-1)$ sphere is blown down,
	the curvetta intersection is transferred to the adjacent vertices, whose self-intersections are correspondingly increased. For $v'$ to reach self-intersection $-1$, one must have blown down $(-1)$ vertices adjacent to it. 
	Throughout the process of blowing down, we maintain the condition that $(-1)$ vertices always have at least one intersecting curvetta.)
	After blowing down the $(-1)$ sphere of valency $\geq 2$, we obtain a point where at least two exceptional divisors
	intersect at the same point with a $\tilde C_j$. Eventually one of these exceptional divisors will be blown down, 
	forcing $\tilde C_j$ to intersect the other exceptional divisor with multiplicity $\geq 2$. 
	Once this other exceptional divisor is blown down, the proper transform of $\tilde C_j$ becomes singular.
\end{remark}

\begin{figure}[hbt!]
	\centering
	\includegraphics[scale=.75]{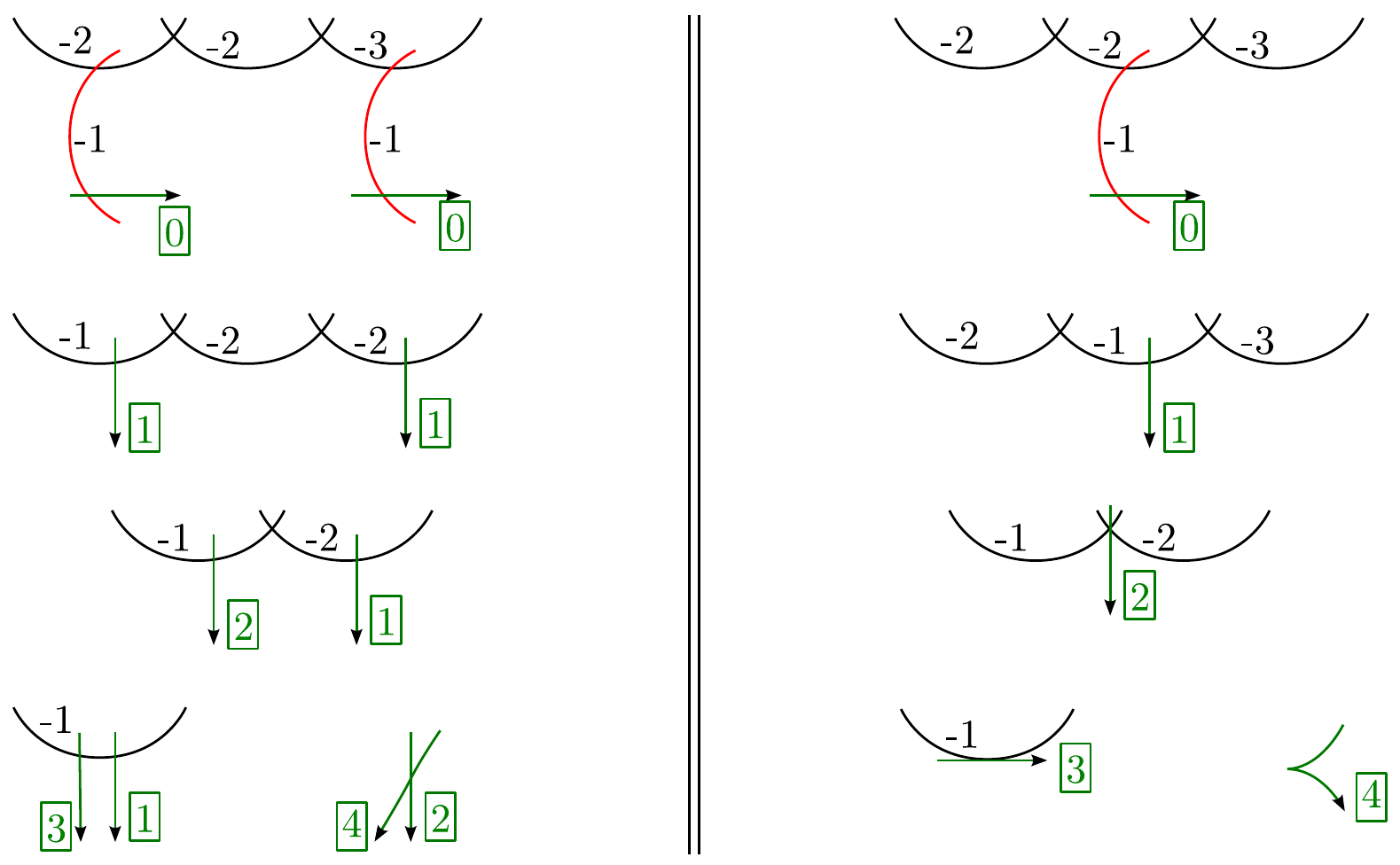}
	\caption{Two possible choices to add $-1$ curves to the same resolution graph, resulting in different curvettas, one with smooth components and another with a singular (cuspidal) component.}
	\label{fig:choices}
\end{figure}

Note that it is possible to have different choices of extension for $G$, such that one choice yields
smooth curvettas and another yields singular curvettas (see Figure \ref{fig:choices} for an example). 
In other words, some sandwiched resolution graphs $G$ have extensions both to a graph which does
satisfy Condition \ref{cond:valence} of Proposition \ref{p:resolution} and to a graph which does not.
For our classifications, we will always work with a choice of extension of $G$ which does satisfy Condition~\ref{cond:valence} and the corresponding smooth curvettas.

We can also deduce some basic numerical properties from Proposition \ref{p:resolution}. It turns out that for a rational singularity $(X, 0)$ with reduced fundamental cycle,
the multiplicity of the singular point determines the number of choices for the defining plane curve germ $\mathcal{C}$ with smooth branches, as well 
as the number of curvetta branches in each such germ. 
Assuming that $(X, 0) \subset (\C^N, 0)$ for some large $N$, recall that the multiplicity $mult\;\! X$ can be defined geometrically as the 
number of intersections $\# X\cap L$ of $X$ with a generic complex $(N-2)$-dimensional affine subspace $L\subset \C$, passing close to the origin. For rational singularities, 
multiplicity is a topological invariant, which can be computed from the resolution graph  by the formula $mult\;\! X = -Z_{min}^2$, see e.g.~\cite{Nem}. 
The statements {of the two lemmas} below are also discussed in 
\cite{dJvS} from the algebro-geometric perspective but they follow easily from the combinatorics of the resolution graph. 

\begin{prop} \label{p:choices} Let $(X, 0)$ be a rational singularity with reduced fundamental cycle, and $\mathcal{C}$ a plane curve germ corresponding to $(X, 0)$. 
If $\mathcal{C}$ has smooth branches, the number of branches is given by~$mult\;\! X -1$. 
\end{prop}
\begin{proof} The minimal normal crossings resolution graph $G$ for $(X, 0)$ has no $(-1)$ vertices. Then $G$ is obtained from (any choice of) the graph $G'$ by deleting all 
vertices $v' \in G'$ with $v' \cdot v'=-1$. The curvetta branches are obtained by putting transverse slices on each $(-1)$ sphere $v' \in G'$,
thus the number $m$ of curvetta branches is given by the number of the $(-1)$ vertices in $G'$. 
By Condition \ref{cond:valence} of Proposition \ref{p:resolution}, 
	$$\sum_{v'\in G'} (v'\cdot v' +a(v')) = -1.$$
Again by Condition~\ref{cond:valence}, each $(-1)$ vertex has valency $1$ in $G'$, so each  addition of a  $(-1)$ vertex to $G$ increases 
the sum $\sum_{v\in G} (v \cdot v+a(v))$ by $1$, thus we have
$$
\sum_{v\in G}(v \cdot v +a(v)) = \sum_{v'\in G'}(v' \cdot v' +a(v')) -m  =-1-m. 
$$
Finally, we relate this quantity to the fundamental cycle $Z_{min}$, which is the sum of homology classes of the exceptional divisors, $Z_{min}= \sum_{v \in G} E_v$:
$$
\sum_{v\in G}(v \cdot v +a(v))= \sum_v E_v^2 + \sum_{v \neq u}  E_v\cdot E_{u}= Z_{min}^2, 
$$
so $m= -1 - Z_{min}^2= mult\;\! X-1$.
\end{proof}

Decorated germs representing a given $(X, 0)$ are obtained from extensions $G'$ of the resolution graph $G$ as above. These can be thought of as combinatorial choices for the gecorated germ; in the next lemma, we compute the number of such extensions. Then, we show that the combinatorial choice, namely  the choice of vertices of $G$ on which 
the additional $(-1)$ vertices are placed to form $G'$, determines the {\em topological type} of the resulting decorated germ.  By definition, the topological 
type of a germ of a singular curve $\mathcal{C} \subset \C^2$ is given by  its link, which is the intersection of $\mathcal{C}$ with a
sufficiently small 3-sphere $S^3 \subset \C^2$ centered at the 
origin. For a decorated germ, we additionally record the weights of the curvetta components. 
Later on, we will see that the different choices of $G'$ correspond to natural different choices of data on the open book decomposition we construct in Section~\ref{s:artin}.
{
\begin{lemma} \label{lem:toptype}  Up to topological equivalence, there are at most $mult\;\! X$ choices of plane curve germs with smooth branches representing $(X, 0)$. 
\end{lemma}

\begin{proof} We first show that there are at most $mult\;\! X= -Z_{min}^2$ possible {\em combinatorial choices} for  germs with smooth components representing $(X, 0)$. These correspond to choices of extensions of
$G$ to $G'$ by adding $(-1)$ vertices. If we have a minimal graph $G$ with an extension $G'$   satisfying Condition \ref{cond:valence} 
	of Proposition \ref{p:resolution}, then we can add another $(-1)$ sphere leaf adjacent to the root to get a
	new graph $G''$, such that the valency of each vertex of $G''$ equals its negative self-intersection.  All the other possible extensions of $G$ to 
	a graph satisfying Condition \ref{cond:valence}  can be obtained by deleting one of the $(-1)$ vertices of $G''$. 
	(Indeed, adding a $(-1)$ vertex to any other position in $G$ would violate Condition \ref{cond:valence}.) Since $G'$ has $(mult\;\! X-1)$ vertices
	of self-intersection $-1$, 
	we know that $G''$ has exactly $mult\;\! X$  vertices with this property, one of which must be deleted. 
Note that because of potential symmetries in the graph $G''$, some of the 
choices of $G'$ will result in isomorphic germs  $C$, but $mult\;\! X$ gives an upper bound on the number of combinatorially different curvetta configurations.
 
  Once the choice of the extension $G'$ of the graph $G$ is made, the topological 
type of the decorated germ $\mathcal{C}$ can be read off directly from $G'$. In particular, we can 
compute the relevant numerical invariants, such as linking numbers between the components of $\mathcal{C} \subset \C^2$.
As before, we assume that $G'$ satisfies Condition \ref{cond:valence} of Proposition \ref{p:resolution}, so that 
$\mathcal{C}$ has smooth branches.

Following  \cite[Definition 4.14]{dJvS}, we define the {\em length} and {\em overlap} functions on the vertices of the graph $G$. For  
$v_0, v_i \in G$, let the length $l(v_0, v_i)$ be the number
of vertices in the path from $v_i$ to $v_0$ in the tree $G$ (including endpoints). For $v_0, v_i, v_j \in G$, 
let the overlap  $\rho( v_i, v_j; v_0)$ be the number 
of common vertices in the paths from $v_i$ to $v_0$ and  $v_j$ to $v_0$.

Let $v_0 \in G \subset G'$ be the root.
Now, if the curvetta $C_i$ comes from the transverse slice on a $(-1)$ sphere corresponding to a leaf of $G'$, and this leaf is attached to the vertex 
$v_i\in G$, then the blow-down process gives $w(C_i)=1+l(v_0, v_i)$. If $C_i$, $C_j$ are the curvettas at the $(-1)$ vertices attached to $v_i$, $v_j$, 
the order of tangency $tang(C_i, C_j)$ between the corresponding 
branches of $\mathcal{C}$ is given by $tang(C_i, C_j)=\rho( v_i, v_j; v_0)$. 

The topological  type of $\mathcal{C} \subset \C^2$ is described via its link, given by the intersection
 $\mathcal{C}\cap S^3$, where $S^3$ is a small sphere cetered at the origin. As each of the curvettas $C_1, \dots, C_m$ is a smooth disk, the intersection of $C_i$ with $S^3$ is an unknot; $\mathcal{C}\cap S^3$ is a link with $m$ components
$C_1 \cap S^3, \dots, C_m \cap S^3$, each of them 
unknotted. {The components of $C_i\cap S^3$ are oriented as boundaries of $C_i\cap B^4$.}
Then, the linking number between two link components equals the order of tangency between the corresponding curvettas, 
$$ 
lk (C_i \cap S^3, C_j\cap S^3)= tang(C_i, C_j). 
$$

The topological equivalence of germs follows from the above calculations, by construction of the links of the germs that we consider; it can also be seen more directly. Any decorated germ for $(X, 0)$ comes, after a blow-down, from a particular placement of the transverse curvetta slices on the $(-1)$ curves corresponding to vertices that we added to $G$ to from the graph $G'$.  This gives a configuration of curvetta slices together with the curve configuration corresponding to the graph $G'$, embedded in a blow-up of $\C^2$. Clearly, for two different choices of the generic curvetta slices for the same graph $G'$, the two configurations of curvettas+curves can be identified by an ambient homeomorphism (in the blown-up $\C^2$). After the blow-down, the induced ambient homeomorphism will identify the links of the resulting germs, showing that the germs are topologically equivalent. 
We already know that the weights will be same, so the decorated germs have the same topological type. 
\end{proof}

The following observation will also be useful later. Let $t(C_i)= \max_j tang(C_i, C_j)$ be
the maximal order or tangency between $C_i$ and another branch of $\mathcal{C}$. Then it follows that 
\begin{equation} \label{weight-ineq}
t(C_i)< w(C_i)
\end{equation}
for all curvettas $C_i$.

\begin{remark} \label{rem:toptype}
De Jong--van Straten \cite{dJvS} study deformation theory of the surface singularity $(X, 0)$; in particular, they are interested in the analytic type of the singularity and 
its deformations. To encode the analytic type of $(X, 0)$, one needs the analytic type of 
the corresponding decorated germ $\mathcal{C}$. By contrast, our focus is on the contact link $(Y, \xi)$ of $(X, 0)$ and its Stein fillings. 
A priori there may be another surface singularity $(X', 0)$ whose link is $Y$, and  by \cite{CNPP}, the singularities $(X, 0)$ and ($X', 0)$ have contactomorphic links.   By Neumann's results \cite{Neu}, all singularities with the same link 
have the same dual graph of minimal resolution, so both $(X, 0)$ and $(X', 0)$ correspond to the same minimal graph $G$. (Note that by \cite{Lauf}, if $G$ has any vertices 
of valency greater than 3, the analytic type of the singularity is not uniquely determined, so indeed $(X, 0)$ and $(X',0)$ may be analytically different in the above scenario.)
We can compare the decorated germs that describe singularities $(X, 0)$ and $(X', 0)$: any choice of the decorated germ for $(X, 0)$   arises from an extension $G'$ of the graph 
$G$ and the corresponding placement of the curvettas. Although analytically 
the exceptional divisors of resolutions of $(X, 0)$ and $(X', 0)$ may be different, topologically they look the same, and we can choose the same extension $G'$ 
and the corresponding placement of curvettas for $(X', 0)$. By the argument above, the resulting germ for $(X', 0)$ will be topologically equivalent to the germ for $(X,0)$, even if the two germs may be analytically different.
This fact will play an important role in the proof of Theorem~\ref{thm:nonMilnor}.
In particular, the two germs will have the same number of branches, the same 
weights and the same pairwise orders of tangency for the branches. 

Of course, if we only know the combinatorics of the graph $G$, we lose analytic information on the plane curve germ $\mathcal{C}$ (such as, for example,
the angles between its transverse branches), but we will never need the analytic information.
The contact 3-manifold~$(Y, \xi)$ is fully determined by the weights and pairwise orders of tangency of the branches of the decorated germ $\mathcal{C}$.
\end{remark}
}

{\subsection{De Jong--van Straten's theory: Milnor fibers from germ deformations.} The main result of \cite{dJvS} says that deformations of the sandwiched singularity can be encoded via
deformations of the  germ $(C,0)$ satisfying certain hypotheses.  We  will state a special case of their theorem that will be relevant to us, but first we introduce some notation.}

We have defined the weights as positive integers $w_i$ associated to the irreducible components 
(curvettas) $C_i$ of $\mathcal{C}$. It will be convenient to interpret the weight $w_i$ as a collection of $w_i$ marked points concentrated at $0 \in C_i$.  
More formally, we consider a subscheme $w(i)$ of length $w_i$ at 0 in $C_i$. The normalization $\ti{\mathcal{C}}$ of the reducible curve $\mathcal{C}$ with smooth components 
is given by the disjoint union of the components $C_i$; thus we can think of the decoration $w=(w_1, w_2, \dots, w_m)$ as a subscheme of $\ti{\mathcal{C}}$, with components 
$w(i)\subset C_i$ as above. (We use notation $\tilde{\mathcal C}$ for normalization here and in the discussion below. Similar notation $\tilde C_j$ had different meaning in 
Proposition~\ref{p:resolution}, though in a sense, both uses refer to resolutions of the curve $C_j\subset \mathcal{C}$. This should not lead to confusion as normalization is only mentioned in the next few paragraphs.)

De Jong and van Straten prove that for sandwiched singularities, 1-parameter smoothings correspond to {\em picture deformations}, 
which are 1-parameter deformations of the germ $\mathcal{C}$ together with the subscheme $w$.  
In fact, de Jong--van Straten describe all deformations of $(X,0)$, but in this paper we  are only interested in  smoothings. Since we do not use their results in full generality,
we omit some technical points and give simpler versions of the definitions and statements from~\cite{dJvS}. 

Informally, picture deformations look as follows. The deformation $\mathcal{C}^s$ is given by individual deformations $C_i^s$ of the curvetta components, so that 
the deformed germ $\mathcal{C}^s$ is reduced and has irreducible smooth components $C^s_i$ corresponding to the original curvettas. 
(In the case of plane curves, any deformation is given by {\em unfolding}, i.e. by deforming the defining equation of the curve.)
The deformation is required to eliminate tangencies between 
the curvettas, so that for $s \neq 0$ all deformed curvettas $C_i^s$ intersect transversally. Thus, the only singularities of the deformed germ 
{$\mathcal{C}^s= \cup_i C_i^s$ for $s\neq 0$}
are transverse multiple points. For $s=0$, the decoration $w$ consists of $w_i$ marked points on the curvetta $C_i$ for each $i=1, \dots, m$, concentrated at $0$. During the 
deformation, these {marked points} move along the curvettas, so that for $s\neq 0$, the deformed curvetta $C^s_i$ contains exactly $w_i$ {\em distinct} marked points, and all 
intersection points $C^s_i \cap C^s_j$ {for $j\neq i$} are marked.  

More formally, deforming the curvettas $C_i$  individually means that we consider $\delta$-constant deformations of the reducible germ $\mathcal{C}= \cup_i C_i$. 
Intersection points between deformed curvettas define the total multiplicity scheme $m^s$ on the normalization $\ti{\mathcal{C}}^{s}$ for $s \neq 0$; if all intersections are transverse, 
the corresponding divisor is reduced, i.e. each point enters with multiplicity one.  The requirement that all intersection points are marked means that 
the deformation  $w^S \subset \ti{\mathcal{C}}\times S$ of the decoration $w$ must satisfy the condition $m^s \subset w^s$. The requirement that 
all marked points be distinct on each $C_i^s$ for $s \neq 0$ is the same as saying that the divisor given by {$w_i^s$} is reduced 
for $s \neq 0$. The condition $m^s \subset w^s$   then implies  automatically that all singularities of the deformed germ $\mathcal{C}^s$ are ordinary multiple points 
(i.e. the deformed curvettas intersect transversally).

\begin{definition}\label{pic-def} A picture deformation $\mathcal{C}^S$ of the decorated germ $(\mathcal{C}, w)$ with smooth components $C_1, \dots, C_m$ over a germ 
of smooth curve $(S, 0)$ is given by a  $\delta$-constant deformation $\mathcal{C}^S \to S$ of $\mathcal{C}$ and a  flat deformation
$w^S \subset  \widetilde{\mathcal{C}^S}= \ti{\mathcal{C}} \times S$ 
of the scheme $w$, such that for $s \neq 0$, the divisor $w^s$ is reduced, the only singularities of $\mathcal{C}^s$ are ordinary multiple points, and $m^s \subset w^s$. 
\end{definition}

Strictly speaking, $w^S$ lives in the normalization, but for $s\neq 0$ we can think of $w^s$ as the set of marked points $\{p_1, p_2, \dots, p_n\}\subset \cup_{i=1}^m C_i^s$, 
such that {\em all} intersection points $C_i^s \cap C_j^s$ are marked. We say that $p_i$ is a {\em free} marked point if it lies on a single $C_i^s$ (away from the intersections). ({Note that these points, and the number of such points $n$, can generally be different for different picture deformations.})

With these definitions in place, de Jong and van Straten's results on smoothings are as stated in Theorem~\ref{thm:intro-djvs}:  every picture 
deformation of $(\mathcal{C}, w)$ gives rise to a smoothing of the corresponding surface singularity $(X, 0)$, and every smoothing arises in this way. 
Specifically, the Milnor fiber of the smoothing that corresponds to the picture deformation $\mathcal{C}^s=\cup_{i=1}^m C_i^s \subset \C^2$ with marked points 
$\{p_1, p_2, \dots, p_n\}$ is obtained by blowing up $\C^2$ at all points $p_1, p_2, \dots, p_n$ and taking the complement of the proper transforms of $C_1^s, \dots, C_m^s$ 
in $\C\#_{j=1}^n \cptwobar$. Picture deformations of $\mathcal{C}$ generate all Milnor fibers, that is, each Milnor fiber of $(X,0)$ arises from some
picture deformation of $(X,0)$ via this construction. Note that Theorem~\ref{thm:intro-djvs} makes no claim of a 
precise one-to-one the correspondence between picture deformations and smoothings: one
expects  that isomorphic smoothings only come from isomorphic picture deformations (in the appropriate sense), but this has not been established.  
In certain cases, one can distinguish Milnor fibers by their topological invariants, or by comparing incidence matrices of the corresponding curvetta
arrangements, \cite[Section 5]{dJvS}, \cite{NPP}. We discuss this in Section~\ref{topology} and use similar
technique to distinguish Stein fillings.

\begin{remark} \label{rmk:milnor-ball} To be more precise, we need to consider the compact version of the construction of Milnor fibers, as follows. Fix a closed Milnor ball $B \subset \C^2$ 
for the germ $\mathcal{C}$. For sufficiently small~$s\neq 0$, the deformed arrangement
$\mathcal{C}^s$ will have a representative in $B$ which meets $\partial B=S^3$ transversally, and all marked points $p_1, \dots, p_n$ are contained 
in the interior of $B$. Let $\ti{B}$ be the blow-up of $B$ at $p_1,\dots, p_n$. Because in the picture deformation all the intersections between deformed curvettas are transverse,
the proper transforms of $C_1^s, \dots, C_m^s$ in $\ti{B}$ will be disjoint smooth disks. 
Let $T_1, \dots, T_m$ be pairwise disjoint tubular neighborhoods 
of these proper transforms.   As a compact 4-manifold with boundary, the Milnor fiber that corresponds to   $\mathcal{C}^s$ is given by 
$W=\ti{B} \setminus \cup_{i=1}^m T_i$, after corners are smoothed, and the Stein structure is homotopic to the complex structure induced from the blow-up. 
\end{remark}

\section{Graphical deformations of curvettas yield fillings} \label{s:homotopy}

Let $(X, 0)$ be a rational surface singularity with reduced fundamental cycle, and consider the associated decorated germ $(\mathcal{C}, w)$ of a reducible plane curve as in the previous
section, with smooth branches $C_1, C_2, \dots, C_m$ equipped with weights.  
Our goal is to build an analog of \cite{dJvS} in the symplectic category:  it turns out that Stein fillings the link of $(X, 0)$ can be obtained from certain
{\em smooth} homotopies of the branches of the decorated germ $\mathcal{C}$. We will restrict to graphical homotopies to streamline our definition and constructions.
(In our setting, one can always choose an appropriate coordinate system, so the graphical hypothesis leads to no loss of generality.)
%

Fix a closed Milnor ball $B$ for $\mathcal{C}$ as in Remark~\ref{rmk:milnor-ball}, so that each branch $C_i$ intersects $\partial B$ transversally. If $B$ is small enough,
the complex coordinates $(x,y)$ in $\C^2$ can be chosen so that all branches  $C_1, C_2, \dots, C_m$ are graphical in $B$: $C_i=\{y=f_i(x)\}$.
We will consider smooth graphical arrangements $\Gamma=\{\Gamma_1, \Gamma_2, \dots, \Gamma_m\}$ such that each $\Gamma_i$ is a smooth graphical 
disk, so that $\Gamma_i=\{y=g_i(x)\}$ for a smooth function $g_i$, and  $\Gamma_i$ intersects $\partial B$ transversally.

The following definition is given for homotopies of the branches defined for a real parameter $t \in [0, 1]$. Sometimes we will use the same 
notion for homotopies defined in a parameter interval $t \in [0, \tau]$, with obvious notational changes.  We assume that coordinates $(x,y)$ are chosen as above.

\begin{definition}~\label{def:homotopy} Let $(\mathcal{C}, w)$ be a decorated plane curve germ, with weights $w_i=w(C_i)$ of its smooth graphical branches $C_1, C_2, \dots , C_m$. 
		A {\em smooth graphical homotopy} of $(\mathcal{C}, w)$   is a smooth homotopy $C^t_i$
	of the branches of $\mathcal{C}$, so that   $\mathcal{C}=  \cup_{i=1}^m C^0_i$,  together with distinct marked points $p_k$, $k=1, \dots, n$ {(for some $n$)}, 
	on~$\cup_{i=1}^m C^1_i$. We assume that in a Milnor ball $B$ the following conditions are satisfied: 
	
	 (1) Each branch is given by  $C_i^t=\{y=f_i^t(x)\}$ for a function $f_i^t(x)=f_i(x,t)$ smooth in $(x,t)$, and $C_i^t$ 
	 intersects $\partial B$ transversally for all $t$. 
	
	 (2) Intersections between the branches remain in the interior of $B$ during the homotopy.
	
	 (3) At $t=1$, all intersections of any two branches $C^1_i$, $C^1_j$ are positive {and transverse}.
	
	 (4) At $t=1$, all intersection points on each branch $C^1_i$ are marked, and there may be additional free marked points.
	 Each free point lies in the interior of $B$ on a unique 
	 branch $C^1_i$. The total number of marked points on $C^1_i$ is $w_i$.
\end{definition}

The choice of Milnor ball $B$ is unimportant as all our considerations are local. For brevity, we will often 
omit $B$ from notation and talk about decorated germs 
and their homotopies in $\C^2$. In that case, we implicitly work in a fixed neighborhood of the origin, and assume that all intersections between branches which begin in this neighborhood remain in this neighborhood 
during the homotopy, and thus the components of the arrangement have controlled behavior near the boundary of the neighborhood.

Conditions (1) and (2) are automatically satisfied for ``small''  homotopies.  Indeed, if $t$ is close to~$0$, $C_i^t$ is $C^1$-close to $C_i$. 
The reducible curve $\mathcal{C}$ with smooth branches has a finite set of tangent directions at the origin, and the branches $C^t_i$ will have tangent spaces 
lying in a small neighborhood of these directions in the Grassmannian of symplectic planes in $\C^2$. Therefore we can choose coordinates so that the fiber of the projection avoids these directions. We only include the intersection of the branches of $\mathcal{C}$ at $0$ in the Milnor ball $B$, so for small $t$ intersections will remain in $B$. For larger homotopies, we require these conditions non-trivially.

Picture deformations satisfy all of the conditions (1)-(4), so a picture deformation is a special case
a smooth graphical homotopy of the germ (in appropriate coordinates). In contrast to picture deformations
of \cite{dJvS}, condition (4) on the marked points and the weight restrictions is only required at $t=1$ for homotopies. For a closer analogy with Definition~\ref{pic-def}, 
we can consider marked points $\{p_i^j(t)\}_{j\in \{1,\dots, w_i\}}$ on $C_i^t$ for all $0\leq t \leq 1$. For $t=0$, the marked points are concentrated at the origin on each branch, giving 
the decoration of $(\mathcal{C}, w)$. Suppose that $p_i^j(t)$, $0\leq t\leq 1$ are smooth functions describing the motion of marked points during homotopy, 
so that $p_i^j(t)\in C_i^t$ for all $t$.  For $t=1$, the points $p_i^j(1)=p_i^j$  satisfy condition (4) above. This implies, in particular, that at $t=1$, the branch $C^t_i$
has no more than $w_i$ intersection points with other branches.  However, for $0<t<1$, the marked points $p_i^j(t)$ are not subject 
to any restrictions and have little significance. The homotoped curvettas $C_i^t$ can have an arbitrary number of intersections, 
and intersections may be positive or negative. By contrast, for picture deformations the weights control the number of intersection points between
deformed curvettas at all times, the intersections between branches are always marked during deformation, and all intersections are positive because curvettas are deformed through complex curves.

Let $(Y, \xi)$ be the link of the singularity $(X,0)$ with the decorated germ $(\mathcal{C}, w)$.  We will show that every smooth graphical 
homotopy of the germ $\mathcal{C}$ gives rise to a Stein filling of $(Y, \xi)$. 

First, we focus on the curvetta arrangement $\{C^1_1, C^1_2, \dots, C^1_m\}$ with marked 
points, produced at the end of homotopy at the time $t=1$. Lemma~\ref{construct-fibration} below produces a certain Lefschetz fibration from this input. The lemma 
applies to any arrangement of smooth graphical disks $\{ \Gamma_1, \Gamma_2, \dots, \Gamma_m \}$ satisfying the stated hypotheses; the homotopy 
is not used at this stage. We use different notation to emphasize that $\{ \Gamma_i\}$ need not be related to $\mathcal{C}$. Then, Lemma~\ref{same-monodromy} uses the 
homotopy between the decorated germ $(\mathcal{C}, w)$ and the curvetta arrangement $\{C^1_1, C^1_2, \dots, C^1_m\}$
with its marked points $p_1, p_2, \dots, p_n$ to show that the open book on the boundary
of the Lefschetz fibration supports $(Y, \xi)$. It follows that our construction produces a Stein filling of $(Y, \xi)$.




As a smooth 4-manifold, the filling produced by Lemma~\ref{construct-fibration} is constructed   similarly to the Milnor fibers in Theorem~\ref{thm:intro-djvs}. 
Namely, we blow up at each of the intersection points of the homotoped curvettas, 
as well as at the free marked points, and  then take the complement of the 
proper transforms of the curvettas. {Even though  $C_i^1$ are smooth disks (rather than complex curves), we will assume that they are locally modelled on complex curves near the intersection point, 
so the blow-up and the proper transforms can be understood in the usual sense. Alternatively, 
one could also think about the proper transform in the smooth sense, 
as the closure of the complement of the blown-up point (see \cite[Definitions 2.2.7 and 2.2.9]{GS}).} To obtain a 4-manifold with given boundary, we consider a compact version 
of the construction in a Milnor ball, as explained in Remark~\ref{rmk:milnor-ball}. It is convenient to consider the Milnor ball of the form $B=D_x \times D_y \subset \C^2$, with corners smoothed,
where $D_x$ and $D_y$ are disks in the coordinate planes $\C_x$ and $\C_y$. 
{For every $x_0 \in D_x$,}  the graphical disks $\Gamma_i$ intersect {$\{x_0\}\times D_y$ transversally,
and the intersection with $\partial (D_x\times D_y)$ lies as a braid in $\partial D_x\times D_y$}. To simplify 
notation, we do not mention the Milnor ball $B$ explicitly in the first part of the lemma.


\begin{lemma} \label{construct-fibration}
	Let $\Gamma_1,\dots, \Gamma_m$ be smooth disks in $\C^2$ which are graphical with respect to
	the projection  $\pi_x$, $\Gamma_i=\{y=f_i(x)\}$. Assume that 
	{at each intersection point of two or more $\Gamma_i$, there exists a neighborhood $U$ of the intersection such that $\cup_i\Gamma_i$ are cut out by complex linear equations inside $U$. (Up to graphical isotopy, this only requires the $\Gamma_i$ to intersect transversally and positively with respect to the orientation on the graph $\Gamma_i$ induced from the natural orientation on $\C$.)}
	Let $p_1,\dots, p_n$ be points on the disks {$\Gamma_i$} which include all intersection points, 
	and let $\alpha:\C^2\#n\cptwobar\to \C^2$ be the blow-up at the points $p_1,\dots, p_n$. Let $\widetilde{\Gamma}_1,\dots, \widetilde{\Gamma}_m$ 
	denote the proper transforms {of $\Gamma_1,\cdots, \Gamma_m$}. 
	
	Then $\pi_x\circ \alpha: (\C^2\#n\cptwobar)\setminus (\widetilde{\Gamma}_1\cup \cdots \cup \widetilde{\Gamma}_m) \to \C$ is a Lefschetz fibration 
	whose regular fibers are punctured planes, where each puncture corresponds to a component $\widetilde{\Gamma}_i$. 
	There is one vanishing cycle for each point $p_j$, which is a curve in the fiber enclosing the punctures that
	correspond to the components~$\Gamma_i$ passing through $p_j$.
	
	Similarly,  if $B=D_x \times D_y$ is a Milnor ball that contains all the points $p_i, \dots, p_n$ {and contains $(D_x\times \C)\cap (\cup_i \Gamma_i)$}, 
	and  $T_i$ is a small tubular neighborhood of $\widetilde{\Gamma}_i$, 
	then $\pi_x\circ \alpha: (\alpha^{-1}(D_x\times D_y))\setminus (T_1\cup \cdots \cup T_m) \to D_x$ is 
	a Lefschetz fibration with compact fiber. The fiber is a disk with holes corresponding to the components $\Gamma_i$. 
	The vanishing cycles correspond to the points $p_j$ in the same way.
	
	If the curvettas $C_1^s,\dots, C_m^s$ with marked points are the result of picture deformation of the germ $(C, w)$ associated to a surface singularity, 
	the Lefschetz fibration constructed as above is compatible with the complex structure on the Milnor fiber of the corresponding smoothing.
\end{lemma}

\begin{proof}
	Before blowing up, the projection $\pi_x:\C^2\to \C$ is clearly a fibration, and the smooth disks $\Gamma_i$ are sections of this fibration. 
	If they were disjoint sections, then their complement would be a fibration whose fiber is $\C$ with $m$ punctures. 
	Since the sections intersect, we blow up at each of the intersection points, along with blow-ups at other chosen points on the curves. 
	For each fiber containing one of the $p_j$ where we blow up, the corresponding fiber in the blow-up is the total transform, which is a 
	nodal curve containing the exceptional sphere and the proper transform of the fiber. More specifically, translating the 
	coordinates $(x,y)$ on $\C^2$ to be centered at $p_j$, the coordinates on the blow-up are
	$$\C^2\#\cptwobar_{p_i}= \{((x,y),[u:v]) \mid xv=yu  \}$$
	The {singular} fiber is the total transform of $F=\{x=0\}$ which has two irreducible components:
	$$\left(E=\{((0,0),[u:v]) \}\right)\cup \left(\widetilde{F} = \{((0,y),[0:1]) \}\right)$$
	The node occurs at the intersection of these two components at $((0,0),[0:1])$. Therefore in a neighborhood of the node we can take $v=1$, so we have local coordinates on the blow-up given by $(y,u)\in \C^2$ where $x=yu$. The projection $\pi_x\circ\alpha$ is given in these coordinates by 
	$$\pi_x\circ\alpha(y,u) = yu$$
	which is exactly the model for a Lefschetz singularity at $(y,u)=(0,0)$.
	
	In the coordinate chart on $\C^2$ centered at $p_j$, let $\Gamma_i=\{(x,f_i(x))\}$. The total transforms of the curves $\Gamma_i$ which pass 
	through $p_j$ (i.e. which have $f_i(0)=0$) are given by 
	$$\left(E=\{((0,0),[u:v]) \}\right)\cup \left(\widetilde{\Gamma}_i = \left\{\left((x,f_i(x)),\left[1:\lim_{a\to x}\frac{f_i(a)}{a}\right]\right) \right\}\right)$$
	and those which do not pass through $p_j$ ($f_i(0)\neq 0$) lift isomorphically to the blow-up:
	$$\{((x,f_i(x)),[x:f_i(x)]) \}.$$
	Note that the proper transforms do not pass through the node $((0,0),[0:1])$. Moreover, since the intersections between the 
	$\Gamma_i$ were assumed to be transverse, $\lim_{a\to 0}\frac{f_i(a)}{a}$ have different values for different values of $i$ where $f_i(0)=0$. 
	Therefore, the $\widetilde{\Gamma}_i$ are disjoint sections of the Lefschetz fibration from the blow-up of $\C^2$ to $\C$, so their
	complement gives a Lefschetz fibration with punctured fibers. Moreover, in the singular fibers, 
	the sections which intersect the exceptional sphere part of the fiber are precisely the proper transforms $\widetilde{\Gamma}_i$ such that $\Gamma_i$ 
	passed through $p_j$.
	
	Regular neighborhoods $T_i$ of the $\widetilde{\Gamma}_i$ can be chosen sufficiently small to be disjoint from each other and the
	Lefschetz singular points, thus yielding the compact Lefschetz fibration. This changes the fiber (converting the punctures into holes) but does not change 
	the fibration structure and the vanishing cycles. The total space of  a Lefschetz fibration over a disk is a compact  4-manifold with  boundary;
	the fibration induces a planar open book decomposition on the boundary.   
	
	In the case of a picture deformation deformation of the germ $(C, w)$, the deformed curvettas $C_1^s, C_2^s, \dots, C_m^s$ are smooth complex disks with marked points 
	satisfying the  hypotheses of the lemma. Then Stein structure induced by the Lefschetz fibration is compatible
	with the complex structure on the Milnor fiber, because $\pi_x\circ \alpha$ is holomorphic.
\end{proof}


Consider a smooth graphical arrangement $\Gamma=\{ \Gamma_1, \dots, \Gamma_m \}$ in a Milnor ball $B=D_x \times D_y$, such that each $\Gamma_i$ transversally intersects the
vertical part $\partial D_x  \times D_y$ of $\partial B$ and is disjoint from $D_x \times \partial D_y$. Taking the boundaries of the graphical disks,  we have an
$m$-braid $\d \Gamma= \partial \Gamma_1 \cup \d \Gamma_2 \cup \dots \cup \d \Gamma_m \subset \d B = S^3$. (Each component $\d \Gamma_i$ is an unknot, but the components are linked.) 
The monodromy of this braid is  called the monodromy of the arrangement $\Gamma$.  We can interpret the braid group on $m$ strands as the mapping class group 
$MCG(\C_m)$ of the $m$-punctured plane. Then the braid $\d \Gamma$ is identified with the monodromy $\phi_\Gamma$ of 
the $\C_m$-bundle over $S^1$, given by the projection $\pi_x: \C^2\setminus  \cup_{i=1}^m \Gamma_i \to \C$ restricted to the preimage $\pi_x^{-1}(\d D_x)$ of the 
circle  $\d D_x \subset \C$. 

To construct the Lefschetz fibration corresponding to $\Gamma$ in Lemma~\ref{construct-fibration}, we perform blow-ups at points $p_i$ 
that project to the interior of $D_x$. These blow-ups do not affect the bundle over $\partial D_x$. Therefore, the non-compact version of the Lefschetz fibration 
(with fiber $\C_m$) has the monodromy $\phi_\Gamma$ given by the braid $\d \Gamma$.

For the compact version of the Lefschetz fibration from Lemma~\ref{construct-fibration},  the general fiber $P_m$ is the disk $D_y$ with $m$ holes. The fibration 
induces an open book on its boundary, with page $P_m$.
The boundary of the total space of the fibration $\mathcal{L}^s$ is the union of two parts: the horizontal boundary $\partial P_m \times D$, which forms the binding of the open book, and the vertical boundary,  
a fiber bundle over $S^1= \partial D_x$ with fiber $P_m$,  which forms the mapping torus for the open book. This fiber bundle is given by the projection 
$(\pi_x\circ \alpha)^{-1}(\partial D_x) \to D_x$, which is the same as the projection $\pi_x: B \setminus  \cup_{i=1}^m \Gamma_i \to D_x$ 
restricted to $\pi_x^{-1}(\d D_x)$, because  the blow-up map $\alpha$ is the identity over $\partial D_x$.  Let $\phi:P_m \to P_m$ denote the monodromy of this fiber bundle (i.e. the monodromy of the open book). We then have a commutative diagram  
$$
\begin{CD}
 P_m @>\phi>> P_m\\
 @VVV @VVV\\
 \C_m @>{\phi_\Gamma}>> \C_m,
\end{CD}
$$
where the vertical maps are inclusions. This proves

\begin{lemma} \label{lem:braid-monodromy} Let $\Gamma=\{\Gamma_1, \dots, \Gamma_m\}$ be a smooth graphical arrangement with marked points $\{p_j\}$, 
and $B= D_x \times D_y$ a Milnor ball whose interior 
contains all marked points, {such that $\Gamma_i \cap (D_x\times \C) \subset B$ and} $\Gamma_i$ is transverse to $\partial B$ for all $i=1, \dots, m$. Let $\phi_\Gamma$ be the monodromy of the braid 
$\d \Gamma = \partial \Gamma_1 \cup \dots \cup \partial\Gamma_m \subset \d B=S^3$.   

Let $\phi:P_m \to P_m$ be the monodromy of the open book induced by the Lefschetz fibration constructed for $(\Gamma, \{p_j\})$ in  Lemma~\ref{construct-fibration}.   
Then $\phi_\Gamma$ is the image of $\phi$ under the projection  $$\eta: MCG(P_m)\to MCG(\C_m)$$ induced by the inclusion $P_m \hookrightarrow \C_m$ of
the compact disk with $m$  holes into 
	the $m$-punctured plane. 
\end{lemma}

When the arrangement $(\Gamma, \{p_j\})$ is related to the decorated germ $(\mathcal{C}, w)$ by a smooth graphical homotopy, the monodromy $\phi_\Gamma$ of the braid 
$\partial \Gamma$ is the same as the monodromy of the braid $\partial \mathcal{C} = \d C_1 \cup \dots \cup \d C_m$, because the homotopy between disks 
gives an isotopy of the two boundary braids. By definition, the braid monodromy of $\partial \mathcal{C}$ is the monodromy $\varphi$ of the singular point of $\mathcal{C}$.

We next examine the monodromy of the open book corresponding to $\Gamma$ in the case of the compact fiber, and find its relation to the monodromy of the singular curve 
$\mathcal{C}$.


\begin{lemma} \label{same-monodromy} Let $\{\Gamma_1,\dots, \Gamma_m\}$ and $\{\Gamma_1',\dots, \Gamma_m'\}$ be two smooth graphical arrangements, such that the boundary braid of are braid-isotopic (respecting labels) and the weights on the corresponding components agree.
	Let $\mathcal{L}$ and $\mathcal{L}'$ be the corresponding Lefschetz fibrations constructed in Lemma~\ref{construct-fibration}. Then the induced open book decompositions on the boundary have the same page and same monodromy.
%
\end{lemma}

We will prove the lemma after pointing out its consequences. 
Since the plane curve {arrangements at either} end of a smooth graphical homotopy have braid-isotopic boundaries, and the weights on the components remain constant during the smooth graphical homotopy, it follows that the open book decomposition induced on the boundary for any Lefschetz fibration arising in this way is independent of the choice of smooth graphical homotopy. 

	For the case where $(\Gamma, w)$ is the end point of a {\em picture deformation} of a plane curve germ $(\mathcal{C}, w)$, $\mathcal{L}'$ is a Lefschetz fibration on the (compactified) Milnor fiber of the associated smoothing of the surface singularity $(X, 0)${, as in Theorem~\ref{thm:intro-djvs}}. In this case the boundary of the Milnor fiber is the link $Y$ of the singularity $(X, 0)$,
	and the Milnor fiber gives a Stein filling of the canonical contact structure $\xi$ on the link, so 
	the open book supports  $\xi$. Because every rational singularity has a picture deformation yielding a Milnor fiber arising in such a manner (see Section~\ref{s:artin} in our case), the open book on the boundary of any Lefschetz fibration arising from the endpoint of a smooth graphical homotopy of the same germ must support the canonical contact structure on the link of the singularity.

Combining Lemma~\ref{construct-fibration} and Lemma~\ref{same-monodromy} with this discussion completes the
proof of Theorem~\ref{thm:Lefschetz} which we summarize in the following corollary. 
\begin{cor} \label{cor:stein-fill}
	A smooth graphical homotopy of the decorated germ $(\mathcal{C}, w)$
	gives rise to a Stein filling of the link $(Y, \xi)$ of the corresponding singularity.
\end{cor}
\begin{proof}[Proof of Lemma~\ref{same-monodromy}] Applying the previous discussion and Lemma~\ref{lem:braid-monodromy} to the arrangement 
$\Gamma=\{\Gamma_1, \Gamma_2, \dots, \Gamma_m\}$, we see that the  the homomorphism $\eta: MCG (P_m) \to MCG (\C_m)$  sends the
open book monodromies $\phi$ and $\phi'$ to the same braid monodromy $\varphi\in MCG(\C_m)$. The kernel of the map $\eta: MCG (P_m) \to MCG (\C_m)$ 
is generated by the boundary-parallel Dehn twists around the holes in the fiber $P_m$. (Recall that  
	the monodromy of an open book is considered {\em rel} boundary of the page, so while the twists around individual strands are trivial in the braid case, 
	the boundary twists  become non-trivial for open books.)
It follows that the monodromies $\phi$, $\phi'$ of the open books on the boundaries of $\mathcal{L}$ and $\mathcal{L}'$ can differ {\em only} by boundary twists, since $\eta(\phi)=\eta(\phi')= \varphi$.

	Let $T_i$ denote a positive Dehn twist around the $i$-th hole. Then we have 
	\begin{equation} \label{extra-twists}
	\phi'= \phi \circ T_1^{\alpha_1} \circ T_2^{\alpha_2} \circ \dots \circ T_m^{\alpha_m}
	\end{equation}
	for some integers $\alpha_1, \alpha_2, \dots, \alpha_m$. The order is unimportant since the boundary twists are in the center of $MCG(P_m)$.
	
	It remains to pin down the boundary twists around each hole, i.e. to show that $\alpha_i=0$ for every $i=1, \dots, m$. To do so, we need to take into account the blow-ups 
	at the free marked points $p_i$ (the marked points that lie on the branches away from the intersections). These correspond to boundary twists. 
	Recall a basic fact about diffeomorphisms of a planar surface {\em rel} boundary: for any two factorizations
	$\Psi$, $\Psi'$ of $\psi: P_m \to P_m$, the number of Dehn twists  that enclose a given hole $h$ is the same for $\Psi$ and $\Psi'$. (Here, we count all twists, not 
	only the boundary ones.) The above statement 
	easily follows from the fact that lantern relations generate all relations in the mapping class group of a planar surface~\cite{MM}, 
	and the number of Dehn twists enclosing a given hole is unchanged under a lantern 
	relation. This implies that the number of Dehn twists enclosing the $i$-th hole is well-defined for a monodromy $\psi: P_m \to P_m$; let $n_i= n_i (\psi)$ denote this number.  
	If two monodromies $\phi$, $\phi'$ are related by (\ref{extra-twists}), we have
	\begin{equation} \label{alpha}
	n_i(\phi')= n_i(\phi)+ \alpha_i.
	\end{equation}
	On the other hand, the number $n_i$ is determined by the vanishing cycles of the Lefschetz fibration. By construction of the fibration $\mathcal{L}^1$ associated 
	to the homotopy $\mathcal{C}^t$,  
	the number of Dehn twists enclosing the $i$-th hole is given by the number of blow-ups at the marked points on $C^1_i$, which in turn equals the weight $w_i$ of the component
	$C_i$ of the original germ $\mathcal{C}$. So  $n_i(\phi)=w_i=n_i(\phi')$, and
	$\alpha_i=0$ from~(\ref{alpha}).
\end{proof}

\begin{remark} Our description of the open book monodromy for an arrangement is somewhat similar to E.~Hironaka's results on the monodromy of complexified real 
	line arrangements in $\C^2$ \cite{Hir}. An important difference is that we consider Lefschetz fibrations on the complement of the proper transform of the 
	curves in a blow-up of $\C^2$, while Hironaka computes the monodromy of the fiber bundle over $S^1$ obtained by projecting the complement of the complex lines 
	in $\C^2$ to a circle of large radius (compare with the proof of Lemma~\ref{same-monodromy}). She also considers the setting with compactified fibers, by taking 
	the complement of tubular neighborhoods of the lines, and computes the 
	monodromy of line arrangements as an element of $MCG(P_m)$. It is important to note that even in the compactified setting, her answers are different 
	from the monodromy of the corresponding Lefschetz fibrations that we consider. (The 
	difference is given by some boundary twists.) The discrepancy appears because when the tubular neighborhoods of $C^1_i$'s are removed from $\C^2$, 
	their parametrization is 
	induced from~$\C^2$. When we blow up and take proper transforms of $C^1_i$, the parametrization of tubular neighborhoods is induced by the Lefschetz fibration structure on the blown-up 
	manifold.
	These two parametrizations are different in the two settings, affecting the choice of the meridian of the tubular neighborhood of a line and the framing of the 
	boundary of the corresponding hole.
\end{remark}

\section{The Lefschetz fibration for the Artin smoothing} \label{s:artin}

\subsection{The Scott deformation.} We can now use a specific deformation to describe the monodromy of the open book decomposition of $(Y, \xi)$. We will use a canonical deformation, 
called the \emph{Scott deformation} in \cite{dJvS}, which yields a smoothing in the Artin component. This deformation yields a particularly nice arrangement
of curvettas where the associated Lefschetz vanishing cycles are disjoint. This in turn yields a model factorization for the monodromy of the boundary open book decomposition. {In Proposition~\ref{artin-fill}, we will show that the corresponding Stein filling is uniquely recognizable from its combinatorics.}
Recall that $tang(C_i, C_j)$ stands for the order of tangency between branches $C_i, C_j$ of $\mathcal{C}$, and  
$t(C_i)= \max_j tang(C_i, C_j)$.

\begin{prop} \label{artin-ob} Let $(X, 0)$ be a rational surface singularity with reduced fundamental cycle, and $(\mathcal{C},w)$ one of its decorated reducible plane curve germs with 
	$m$ smooth irreducible components.  
	Let $(Y, \xi)$ be the contact link of $(X,0)$. Then $(Y, \xi)$ has a planar open book decomposition
	whose page is a disk with $m$ holes $h_1, \dots, h_m$, corresponding to the branches of $\mathcal{C}$. The open book monodromy admits a factorization 
	into {\em disjoint} positive Dehn twists with the following properties:
	
	1) for any  two branches $C_i, C_j$, the corresponding holes $h_i, h_j$ are enclosed by 
	exactly $tang(C_i, C_j)$ of these Dehn twists;
	
	2) there are   $w(C_i) - t(C_i)>0$ boundary Dehn twists around the hole $h_i$;
	
	3) there is at least one positive Dehn twist about the outer boundary component of the page. 
	
\end{prop}

\begin{proof} 
	We use the picture deformation of $(\mathcal{C}, w)$ referred to as the \emph{Scott deformation} in \cite[Proposition 1.10]{dJvS}.
	This deformation arises from iteratively applying the following procedure. (We refer the reader to \cite{dJvS, ACampo} for details, including the explanation why the procedure below can be actually realized by a 1-parameter deformation.)
	
	The input of the procedure is an isolated singular point $p$ of a plane curve $C$ with multiplicity $m$. 
		In our case $C$ is a union of smooth components, 
	and the multiplicity $m$ is the number of components through the point $p$. The output of the procedure is a deformation $C'$ whose singularities are:
	\begin{enumerate}
		\item \label{i:mtuple} one $m$-fold point where $m$ branches intersect transversally, and
		\item \label{i:others} the collection of singularities occurring on the proper transform of $C$ in the blow-up of $\C^2$ at~$p$.
	\end{enumerate}
	The idea of the deformation is to blow up at $p$, perform a small deformation of the curves so that the singularities of the proper transform
	become disjoint from the exceptional divisor, and then blow down the exceptional divisor to return to the plane and obtain the curve $C'$.
	
	We demonstrate this process in an example in Figure~\ref{fig:ScottLoc}. The initial configuration in the bottom left consists of five curves.
	The  curves $C_1$ and $C_2$ are tangent with multiplicity $3$, and these two curves are tangent to $C_3$ with multiplicity $2$. 
	The curves $C_4$ and $C_5$ are transverse to $C_1, C_2, C_3$ but tangent to each other with multiplicity $4$. 
	After blowing up at the common intersection point, we obtain the proper transforms together with an exceptional divisor as 
	shown in the top left of Figure~\ref{fig:ScottLoc}. Now $C_1$ and $C_2$ are tangent with multiplicity 2 and 
	transversally intersect $C_3$ at the same point on the exceptional divisor. The curves $C_4$ and $C_5$ 
	become disjoint from $C_1$, $C_2$, $C_3$,  and they are tangent to each other with multiplicity 3 at another point on the exceptional divisor. 
	Next we perform the deformation of the curves, fixing the exceptional divisor, but translating the proper transforms $\ti{C}_1, \ti{C}_2, \dots, \ti{C}_5$
	of the curvettas slightly so that the intersection of the exceptional divisor with the proper transforms now occurs away from the intersections of the proper 
	transforms with each other, as shown in the top right of Figure~\ref{fig:ScottLoc}.
	Finally, we blow down the exceptional divisor, which results in a transverse intersection of the
	resulting curvettas $C_1^s, C_2^s, \dots, C_5^s$ together with the singularities (intersections) of the proper transforms as required.
	
	\begin{figure}
		\centering
		\bigskip
		\labellist
		\small\hair 2pt
		\pinlabel $C_1$ at 151 148
		\pinlabel $C_2$ at 149 106
		\pinlabel $C_3$ at 136 33
		\pinlabel $C_4$ at 79 174
		\pinlabel $C_5$ at 115 174
		\pinlabel $\ti{C}_2$ at 147 296
		\pinlabel $\ti{C}_3$ at 147 244
		\pinlabel $\ti{C}_1$ at 150 345
		\pinlabel $\ti{C}_5$ at 150 404
		\pinlabel $\ti{C}_4$ at 125 427
		\pinlabel ${C}^s_1$ at 395 165
		\pinlabel ${C}^s_2$ at 395 108
		\pinlabel $C^s_3$ at 395 50
		\pinlabel $C^s_4$ at 375 213
		\pinlabel $C^s_5$ at 395 190
		\endlabellist
		\includegraphics[scale=.7]{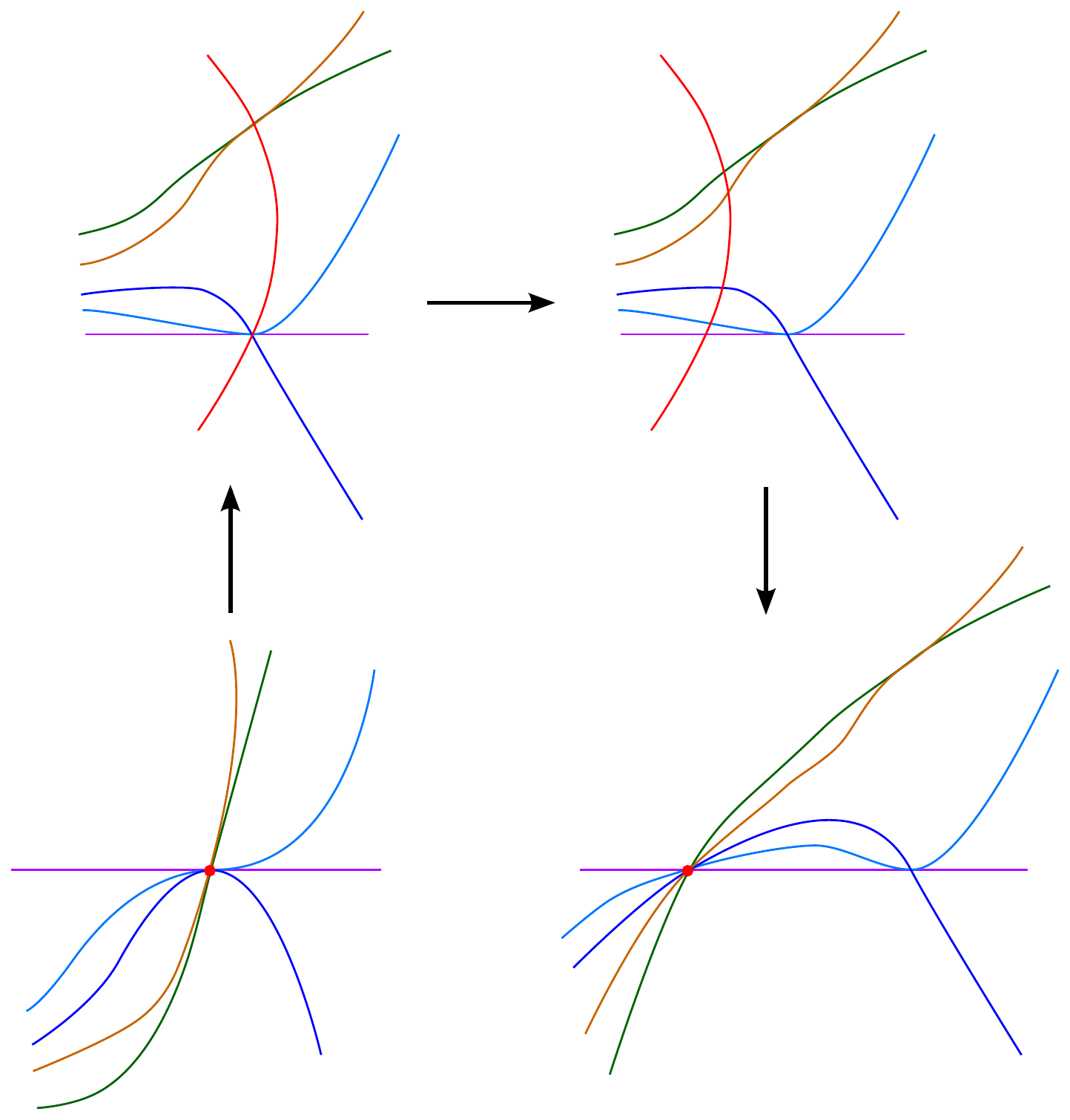}
		\caption{One iteration of the Scott deformation in an example.}
		\label{fig:ScottLoc}
	\end{figure}
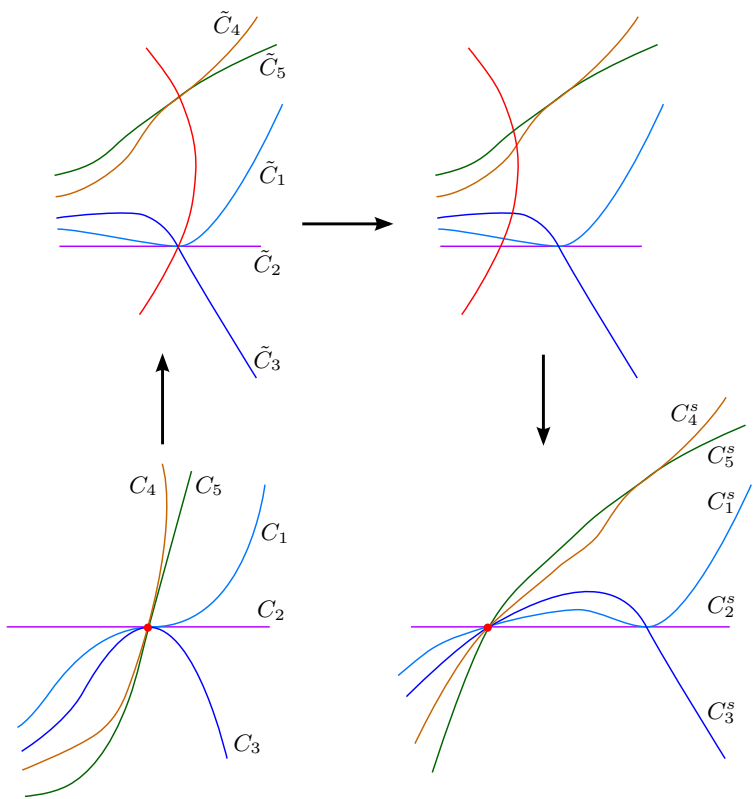
	
	Since the multiplicity of the orders of tangency between components decreases each time we take the proper transform, 
	applying this procedure iteratively to the singularities of type \ref{i:others} eventually yields a deformation to a plane curve
	with only transverse intersections. See Figure \ref{fig:artin} for the iterations of the Scott deformation in our example, until all of the singularities are transverse intersection points.
	When working with a decorated germ $(\mathcal{C},w)$, with the marked points of $w$ initially concentrated at 0,  the same blow-up
	procedure will separate the marked points. Indeed, if there are additional marked points which increase the weight, they can be separated 
	by additional iterations of the blow-ups and translations, so that at the end all marked points are disjoint. (In this sense the scheme $w^s$ is reduced.)
	Note that the total weight $w(C_i)$ of each component is equal to the total number of marked points on that component (including the intersection points).

	\begin{figure}
		\centering
		\includegraphics[scale=.55]{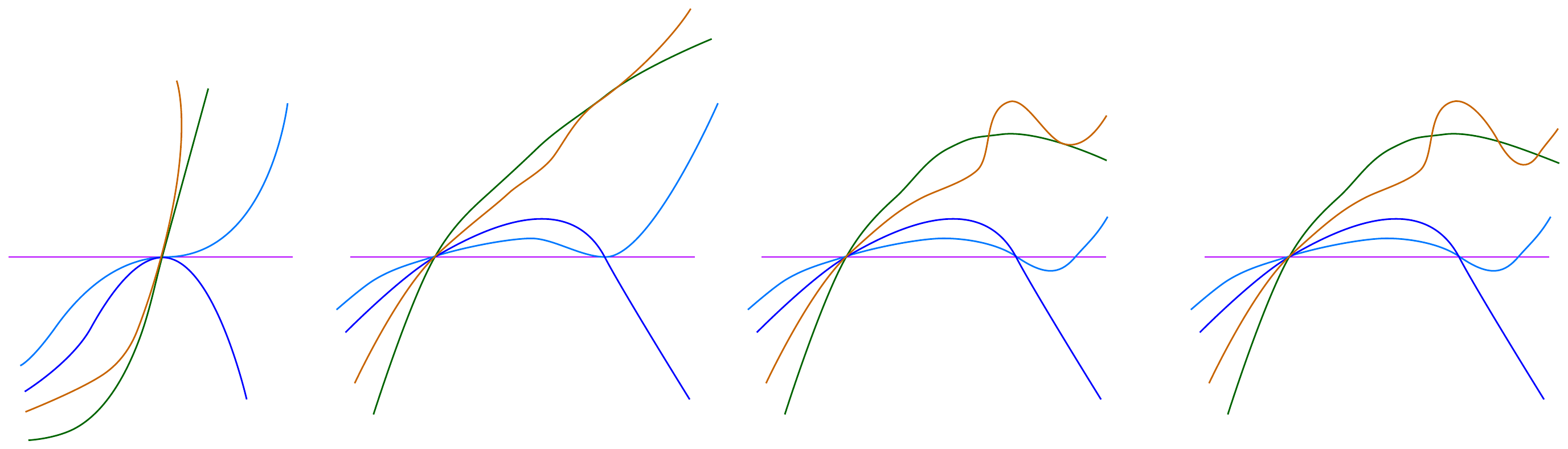}
		\caption{A Scott deformation applied iteratively until all intersections are transverse.}
		\label{fig:artin}
	\end{figure}

	When the components of $\mathcal{C}$ are smooth,  the result of this deformation is as follows. If some components of $\mathcal{C}$  were tangent 
	to order $r_1$ before the deformation, they  will all pass through the same $r_1$ transverse multi-points $p_{i_1},\dots, p_{i_{r_1}}$.
	If another component of $\mathcal{C}$ intersect these components with multiplicity $r_2<r_1$ before the deformation, 
	this component will pass through $r_2$ of these points afterwards.  The total number of intersection points appearing
	on the Scott deformation of a component $C_i$ is precisely $t(C_i)$, the highest possible order of tangency between $C_i$ 
	and another branch in the original germ $\mathcal{C}$. In this sense, the intersection points are used as efficiently as possible. The number of additional marked points on $C_i$ is $w(C_i)-t(C_i)$.


	
	Now, consider the Lefschetz fibration constructed from the Scott deformation via Lemma \ref{construct-fibration}. 
	We claim that up to a curve isotopy, the vanishing cycles of this fibration are 
	disjoint curves on the planar page. The reason for this is built into the iterative nature of 
	the Scott deformation, which results in a nesting of the vanishing cycles as follows.
	
	Consider the equivalence relations on the components $C_1,\dots, C_m$ 
	of the germ $\mathcal{C}$ defined by $C_i\sim_{l} C_j$ if $C_i$ and $C_j$ 
	intersect at $0$ with multiplicity at least $l$. The transitivity of
	this relation comes from the fact that if $C_1$ intersects $C_2$ with 
	multiplicity $r$ at $0$ and $C_2$ intersects $C_3$ with multiplicity $s$ at $0$, 
	then $C_1$ must intersect $C_3$ with multiplicity at least $\min\{r,s\}$. 
	These equivalence relations induce partitions of the components of $\mathcal{C}$, and $\sim_l$ refines $\sim_{l'}$ for $l>l'$.
	
	If we apply the Scott deformation procedure iteratively, on the first iteration, we obtain one transverse intersection of all 
	of the branches (the singularity of type \ref{i:mtuple}) which groups the components of $\mathcal{C}$ according 
	to the (unique) block of the partition induced from $\sim_1$. Applying the Scott deformation 
	procedure to all the singularities of type \ref{i:others}, generates a transverse multi-point of type \ref{i:mtuple} 
	for every block in the partition induced by $\sim_2$. Iterating this procedure, we obtain a transverse intersection 
	for every block of each partition $P_l$ induced by $\sim_l$, for $l\geq 1$. 
	For sufficiently large $l$, each block will consist of a single smooth component, 
	and thus no new transverse intersections of type \ref{i:mtuple} will result from the procedure. 
	When a block contains a single element, there may or may not be additional marked points placed.
	Instead of using the partition and Scott deformation to place additional marked points, 
	we can simply use the formula that $C_i$ must have $w(C_i)-t(C_i)$ total additional marked points.
	
	Recall that there is one vanishing cycle in the Lefschetz fibration for each marked point of the Scott deformed curve, 
	and this vanishing cycle encircles the punctures/holes corresponding to the components of curves which pass through
	the given marked point. Because the equivalence relations~$\sim_l$ refine each other as $l$ increases,
	the subsets of $C_i$ which intersect at the $(l+1)^{st}$ iteration are nested within the subsets of $C_i$ 
	which intersect at the $l^{th}$ iteration. Moreover, because the isotopy in the blow-up procedure 
	can be made arbitrarily small, we can assume that there is no braiding of the components $C_i$ between the $l^{th}$
	and $(l+1)^{st}$ iterations (see Section \ref{find-curvettas-for-filling} for more details on how braiding of
	the curves can occur and be understood in general).
	More specifically, observe that in the Scott deformation procedure, 
	as in Figure~\ref{fig:ScottLoc}, the deformation from right to left in the blow-up (at the top of the figure) can be performed 
	by an arbitrarily small translation of the exceptional divisor. 
	By making the translation 
	sufficiently small, we can ensure that in each subset intersecting at the $(l+1)^{st}$ iteration, the curves stay close together and 
	do not interact with another such subset. (In the language of Section~\ref{find-curvettas-for-filling},   non-trivial braiding would correspond to a crossing of
	the wires, and a small translation ensures that the wires cannot cross in between  the singularities produced 
	iteratively by the Scott deformation.) 
	Then, the vanishing cycles corresponding
	to the intersections of type \ref{i:mtuple} which are introduced at the $(l+1)^{st}$ iteration will be nested inside
	(and thus disjoint from) the vanishing cycles corresponding to the intersections of type \ref{i:mtuple} 
	introduced at the $l^{th}$ iteration. We can also assume that any two vanishing cycles introduced in this
	way at the $l^{th}$ iteration are disjoint because the application of Lemma \ref{construct-fibration} to the 
	Scott deformation actually realizes these Lefschetz singularities simultaneously in the same fiber 
	(we can later perturb so they arise in different fibers if desired). Finally,
	the additional marked points at smooth points of the $C_i$ correspond to vanishing cycles which
	are boundary parallel to the $i^{th}$ hole, and thus can be realized disjointly from each other and all other vanishing cycles. 
	Thus we conclude that the Scott deformation yields a Lefschetz fibration with disjoint vanishing cycles.  This means that the compatible planar open book for the link $(Y, \xi)$ has monodromy which 
	is a product of positive Dehn twists about the \emph{disjoint} curves described above.
	Because at the first step we get a transverse intersection of 
	{\em all} deformed curvettas, the corresponding vanishing cycle encloses {\em all} holes, i.e. we have a Dehn twist about the outer boundary component of the page. 
	%
\end{proof}

\subsection{Symplectic resolution and Lefschetz fibrations.}
It is noted in \cite{dJvS} that the Scott deformation corresponds to the Artin smoothing, which in this situation is diffeomorphic to the resolution of the singularity.
In fact, we can see more directly, through symplectic topological means, that the Lefschetz fibration corresponding
to this Scott deformation gives a plumbing which necessarily corresponds to the resolution of the singularity.

We recall the procedure of \cite[Theorem 1.1]{GayMark}. Starting with the plumbing graph $G$,
this procedure produces a planar Lefschetz fibration compatible with the symplectic resolution of a rational  singularity with reduced fundamental 
cycle. (The sympectic structure on the plumbing can be deformed to the corresponding Stein structure.)  In fact,
\cite[Theorem 1.1]{GayMark} applies to a wider class of singularities (see Subsection~\ref{unexp-non-r} below), but we first describe it for this particular case.
To construct the fiber of the Lefschetz fibration, 
take a sphere $S_v$ for each vertex  $v \in G$ and cut out  $-a(v) - v \cdot v \geq 0$ disks out of this sphere. (As before, $a(v)$ is the valency
of 
the vertex $v$; the number of disks is non-negative by~\eqref{valency-weight}.) 
Next, make a connected sum of these spheres with holes  by adding a connected sum neck for each 
edge of $G$.  For a sphere $S_v$
corresponding to the vertex $v$, the number of necks equals the number of edges adjacent to $v$, i.e. its valency $a(v)$. 
The resulting surface $S$ has genus 0 because $G$ is a tree. See {the} top of Figure~\ref{fig:outer} for an example.

\begin{prop} \cite[Theorem 1.1]{GayMark} \label{gay-mark-fibration}
	The surface $S$ constructed above is the fiber of a Lefschetz fibration on a symplectic
	neighborhood of symplectic surfaces intersecting $\omega$-orthogonally according to the graph $G$. 
	The vanishing cycles are given by the curves parallel to the boundaries of the holes (one curve for each hole) and the {cores} of the necks of the connected sums.
\end{prop}

Let $\ti{X}$ be the Milnor fiber of the Artin smoothing component for a rational $(X, 0)$ with reduced fundamental cycle; $\ti{X}$ is a Stein filling for the contact link 
$(Y, \xi)$. We now have several different Lefschetz fibration structures on $\ti{X}$. First, because $\ti{X}$ is diffeomorphic to the minimal resolution of $(X, 0)$, 
a Lefschetz fibration is produced by the Gay--Mark construction of Proposition~\ref{gay-mark-fibration}. Second, for each choice of the decorated germ $(\mathcal{C}, w)$ with smooth branches,  the proof 
of Proposition~\ref{artin-ob} also gives a Lefschetz fibration on $\ti{X}$. All these Lefschetz fibrations have planar fibers. In our construction of the Lefschetz fibration from
the curvetta arrangement, the general fiber has  a distinguished ``outer'' boundary component coming from the fibration $\pi: B \to \C$ on the Milnor ball 
$B=D_x \times D_y \subset \C^2$. 
In the Gay-Mark construction, there is no distinguished boundary component of the fiber. On the other hand, the decorated germ is not uniquely defined:
recall from Proposition~\ref{p:choices} that there are {$mult\;\! X$} choices of decorated germs with smooth branches representing $(X, 0)$,
where some of these germs may coincide  due to symmetries in the extension of the resolution graph.
Of course, since the link of the singularity is independent of the choice of curvetta germs, the Stein filling arising from the Artin smoothing should not depend on these choices.
We now show that the choice of curvettas corresponds precisely to the choice of the outer boundary component, so this choice 
only affects the presentation of the Lefschetz fibration.  See Figure~\ref{fig:outer} for an example.
\begin{figure}
	\centering
	\includegraphics[scale=1.2]{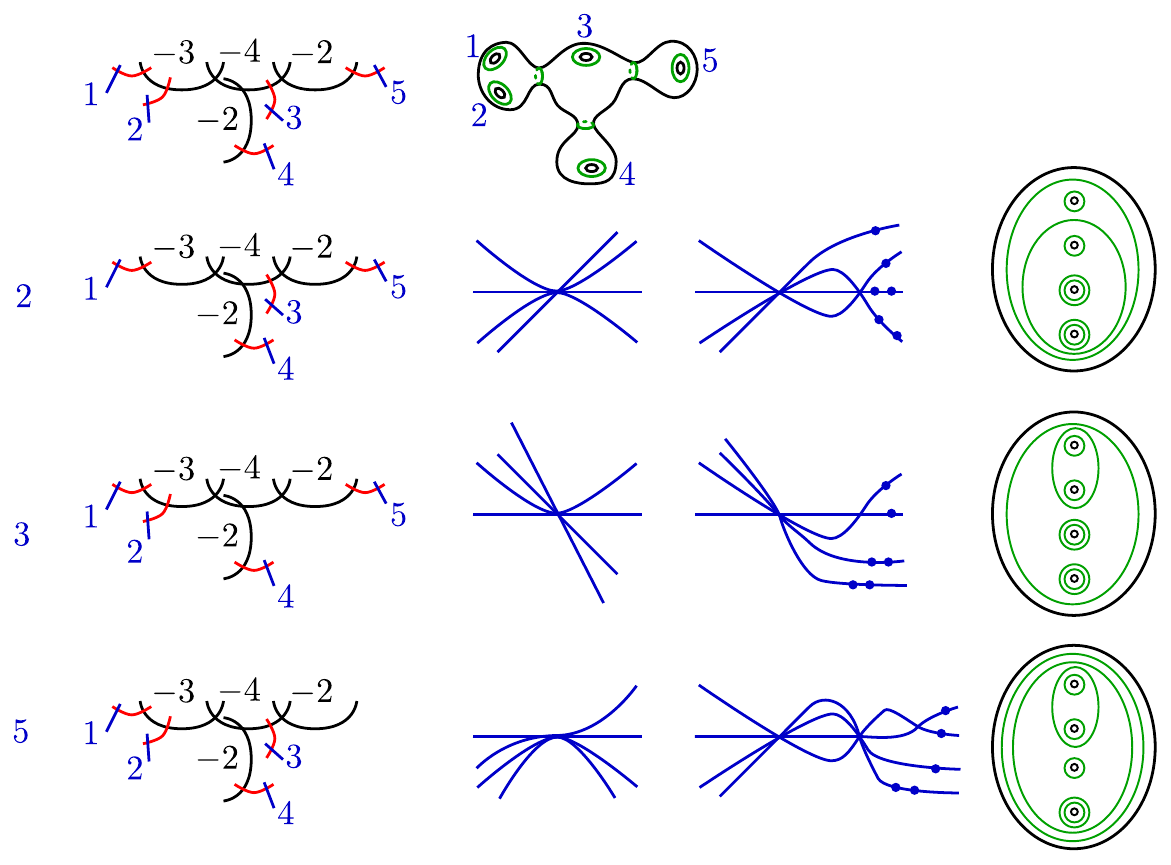}
	\caption{An example demonstrating different choices of curvettas correspond to different choices of outer boundary component
		for the fiber of the Lefschetz fibration. At the top we have the resolution configuration 
		and the corresponding Gay--Mark Lefschetz fiber with vanishing cycles. The resolution configuration is augmented with red $(-1)$ curves 
		and blue curvettas. For each choice of curvettas we delete exactly one of these red $(-1)$ curves and the corresponding curvetta. 
		We show the resulting curvettas, their Scott deformation, and the corresponding planar Lefschetz fibration obtained 
		from Lemma~\ref{construct-fibration} in the cases of excluding the $(-1)$ curves labeled $2$, $3$, and $5$. 
		Note that because of symmetries in the graph, the exclusion of $1$ or $2$ yield very similar looking cases, and similarly with the exclusion of $4$ or $5$.}
	\label{fig:outer}
\end{figure}

\begin{lemma} Let $\mathcal{L}$ be the planar Lefschetz fibration on $\ti{X}$ provided by  Proposition~\ref{gay-mark-fibration}. Then 
	the  {$mult\;\! X$} different choices of 
	smooth curvetta germs for $(X, 0)$ produce, via the Scott deformation, planar Lefschetz fibrations on $\ti{X}$ with a distinguished boundary component of the fiber. 
	The choices of smooth curvetta germs are in one-to-one  correspondence with the different choices of outer
	boundary component of the general fiber of $\mathcal{L}$.
\end{lemma}

\begin{proof}
	As before, we associate to each vertex of the resolution graph $G$ for the singularity the quantities $v\cdot v$ for the self-intersection
	and $a(v)$ for the valency. In the Gay-Mark Lefschetz fibration $\mathcal{L}$,  each vertex $v \in G$  contributes  $-v\cdot v-a(v)$ boundary components to the fiber. 
	On the other hand, recall from the proof of Propositions~\ref{p:resolution} and~\ref{p:choices} that the germ $\mathcal{C}$ of smooth curvettas is obtained from an extension of the resolution graph $G$ to a graph $G'$. We attach $-v\cdot v-a(v)$ vertices with self-intersection $-1$ and valency $1$ to each vertex $v$ to obtain a graph $G''$ and then delete exactly one of these $(-1)$ vertices to get the graph $G'$. This shows 
	that the number of choices for the germ matches the number of boundary components of the fiber of $\mathcal{L}$, {and this number is exactly $mult\;\! X=-\sum(v\cdot v+a(v))$}. 
	The curvetta branches of the germ $\mathcal{C}$ are obtained by taking disks dual to the remaining $(-1)$ vertices and considering 
	their proper transform after blowing down all exceptional divisors; thus the curvettas are in one-to-one correspondence with the $(-1)$ vertices of $G'$.
	In turn, in the Lefschetz fibration constructed by Lemma~\ref{construct-fibration}, the ``inner'' boundary components of the fiber are in one-to-one 
	correspondence with the curvettas. 
	The deleted $(-1)$ vertex in $G''$ still corresponds 
	to a boundary component in the fiber of the Gay-Mark Lefschetz fibration $\mathcal{L}$, thus we can say that it corresponds to the outer boundary 
	component of the fiber of the planar Lefschetz fibration produced by Lemma~\ref{construct-fibration}. 
	Note also that if we enumerate the $(-1)$ vertices of the graph $G'$ by $1, 2, \dots, m=mult\;\! X-1$, we get an enumeration of the components of $\mathcal{C}$, which in turn gives an enumeration of the holes of the fiber.
\end{proof}

Recall from Remark~\ref{rem:toptype} that there may be different analytic types of singularities with the same link $(Y, \xi)$. These singularities are all topologically 
equivalent and have the same graph $G$, so that decorated germs for each these singularities are obtained from extensions of $G$. 
A particular choice of extension gives topologically equivalent decorated germs for all singularities with link $Y$. Topologically equivalent germs yield the same open book
decompositions of $(Y, \xi)$ as in Proposition~\ref{artin-ob}, since the weights and the orders of tangency between branches are encoded by the topological type. Together with the previous proposition, 
this gives

\begin{cor} Let $(Y, \xi)$ be a link of surface singularity with reduced fundamental cycle. Then for  any singularity $(X, 0)$ whose link is $Y$ and any choice of 
the decorated germ $\mathcal{C}$ for $(X, 0)$ with smooth branches, the open book decomposition of $(Y, \xi)$ defined by $\mathcal{C}$ is the same; namely, 
the open book induced by the Gay-Mark Lefschetz fibration. Different extensions $G'$
of the resolution graph $G$ used to construct $\mathcal{C}$ correspond to different choices of the outer boundary of the page of the open book. Enumeration of the branches of 
$\mathcal{C}$ (or equivalently, of the $(-1)$ vertices of $G'$) corresponds to enumeration of the holes in the page. 
\end{cor}

It is interesting to note that the Milnor fiber of the Artin 
smoothing is the {\em only} Stein filling with disjoint vanishing cycles in its Lefschetz fibration.  

\begin{prop}\label{Artin-factor} 
	Suppose a planar Lefschetz fibration has disjoint vanishing cycles, with at least one boundary parallel vanishing cycle for each boundary component.
	Then this is a Lefschetz fibration for the Artin smoothing of a rational singularity with reduced fundamental cycle.
	In particular, the induced open book decomposition on the boundary supports the contact link of a rational singularity with reduced fundamental cycle.
\end{prop}

\begin{proof} 
	As in \cite{GayMark}, if the vanishing cycles are disjoint, we can realize all Lefschetz singularities simultaneously in the same fiber. The unique singular fiber is thus a configuration of spheres intersecting transversally according to a graph. 
	Note that the boundary parallel twists are important to ensure that the only non-closed components of the singular
	fiber are disks which retract to a point. (These disks come from the small annuli around the holes.) The non-singular fibers provide a regular neighborhood for the configuration, so the entire $4$-manifold is a symplectic plumbing. 
	This $4$-manifold gives a symplectic filling for a contact structure supported by a planar open book, thus by \cite{Etn-planar} its intersection form is negative definite,
	i.e. the plumbing graph $G$ is negative definite.  Thus, the graph can be thought of as the resolution graph of a normal surface singularity $(X, 0)$.
	
	As in \cite{GayMark}, $-v \cdot v \geq a(v)$ for each vertex $v \in G$, so $(X, 0)$ is a rational singularity with reduced fundamental cycle. To see this, observe that each vertex $v \in G$ corresponds to 
	a closed component $\hat{S}_v$ of the singular fiber. Alternatively, $\hat{S}_v$ can be viewed as the union of a component $S_v$ of the complement of the vanishing cycles in a regular fiber capped off by thimbles for each of its boundary vanishing cycles. Then,  $v \cdot v=\hat{S}_v \cdot \hat{S}_v$ equals 
	the negative number of thimbles in $\hat{S}_v$, or equivalently the negative number of vanishing 
	cycles on the boundary of $S_v$ (see \cite[Proposition 2.1]{GGP}). The valency $a(v)$ is the number of other spheres in the singular fiber intersecting $\hat{S}_v$. 
	Put differently, $a(v)$ is the number of closed surfaces $\hat{S}_{v'}$, $v'\neq v$, such that $S_v$ and $S_{v'}$ share a vanishing cycle in their boundaries; thus $a(v)$ is the number 
	of the vanishing cycles in $\partial S_v$ that are not adjacent to a boundary component in the fiber. Then $-v \cdot v - a(v)$ is the number of vanishing cycles adjacent to a boundary component in 
	$\partial S_v$, so $-v \cdot v - a(v) \geq 0$, as required. Note also that $v \cdot v \leq -2$ as 
	each $S_v$ has at least 2 vanishing cycles on the boundary, so $G$ is the graph of the {\em minimal} resolution.
	
	The above discussion implies that if we run the construction of Proposition~\ref{gay-mark-fibration} for the graph $G$, we recover the given Lefschetz fibration. It follows that 
	our Lefschetz fibration is compatible with the symplectic structure on the minimal resolution.
	For a rational singularity, the resolution is diffeomorphic to the Milnor fiber  of the Artin smoothing (and the symplectic structure on the symplectic plumbing deforms to the corresponding 
	Stein structure).  This shows that the Lefschetz fibration produces the same filling as the Artin smoothing.
\end{proof}

{\subsection{A digression: some non-rational singularities and potential unexpected fillings.} \label{unexp-non-r}
Although we stated Proposition~\ref{gay-mark-fibration} for rational singularities, 
Theorem 1.1 of \cite{GayMark} is more general: the same construction works when the normal crossings resolution has exceptional curves of higher genus, as long as 
condition~(\ref{valency-weight}) is satisfied. The fiber of the corresponding Lefschetz fibration is formed by taking the connected sum of surfaces given by the exceptional 
curves and cutting $-v \cdot v - a(v) \geq 0$ holes in the surface corresponding to $v \in G$. As before, the vanishing cycles are given by the boundary parallel curves around the holes 
and the curves around the connected sum necks. We can use this construction together with monodromy factorizations of~\cite{BMVHM} {to construct infinite collections of Stein fillings for links of certain non-rational singularities.} 
\begin{figure}[htb]
	\centering
	\includegraphics[scale=.6]{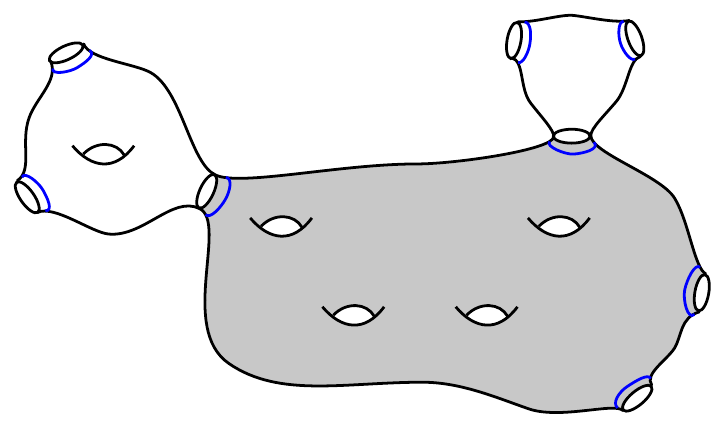}
	\caption{The Gay-Mark Lefschetz fibration for the resolution of a non-rational singularity which admits infinitely many unexpected fillings. The subsurface of genus~$4$ with $d=4$ used to produce infinitely many monodromy factorizations is shaded. {Vanishing cycles are drawn in blue.}}
	\label{fig:highgenus}
\end{figure}

Indeed, suppose that a normal surface singularity $(X,0)$ has a good resolution such that 
one of the exceptional curves has genus $g\geq 2$ and
self-intersection $-d$, with $d \leq 2g -4$. As before, we assume that the resolution graph has no bad vertices, i.e. satisfies~(\ref{valency-weight}).
Then the fiber of the Lefschetz fibration from \cite[Theorem 1.1]{GayMark} has a subsurface of 
genus $g$ with some necks and holes, and a vanishing cycle around each neck and each 
hole. (See Figure~\ref{fig:highgenus}.) The total number of these vanishing cycles is $d$.  
We can cut out this subsurface along the curves parallel to the vanishing cycles to get a surface of genus $g$ with $d$ holes, so that the product of the Dehn twists around the 
vanishing cycles is the boundary multi-twist. For $d\leq 2 g -4$, \cite[Theorem A]{BMVHM} establishes that the boundary multi-twist has infinitely many positive factorizations as products of Dehn twists about non-separating curves. These factorizations can consist of arbitrarily many Dehn twists.   
It follows that the monodromy of the corresponding open book on the link $(Y, \xi)$ has infinitely many positive factorizations, 
each of which produces a positive allowable Lefschetz fibration (see~\cite{AkbOz}) and thus a Stein filling; these Stein fillings can have arbitrarily high Euler characteristic. We ask

\begin{question}
 Does the above construction produce any unexpected Stein fillings?
\end{question}

To answer this question, one would need to contrast these Stein fillings and the Milnor fibers of {\em all} surface 
singularities with the given link.  Each fixed singularity can only have finitely many Milnor fibers. (Indeed, the Milnor fibers correspond to the components of the base of miniversal deformation; the base  is 
a germ of an analytic space,  and as such it can only have finitely many components, see e.g. \cite[Theorem 4.10 and discussion in section 7]{PPP}.) However, because of the presence of a higher-genus surface in the resolution, 
every singularity as above is not (pseudo)taut \cite{Lauf}, which means that there exist infinitely many analytic types of singularities with the same dual resolution graph, and thus the same contact link. We are interested in the Stein topology of the Milnor fibers, which is more coarse than the analytic type;  in principle, it is possible  that the infinite collection of analytic types of the singularity would only give rise to finitely many Stein homotopy types for the Milnor fibers. Thus, we have the following dichotomy: either (1) there are only finitely many Stein homotopy types (or diffeomorphism types) of the Milnor fibers, which would imply existence of unexpected fillings, or (2)~an infinite collection of possible analytic types gives rise to an infinite collection of pairwise distinct Stein fillings. Establishing either outcome would be extremely interesting, even for a single example.  

It should also be noted that in the non-rational case, one should in principle consider non-normal singularities as well, as these might generate additional Stein fillings, see \cite{PPP2} for a detailed discussion of this issue 
(which doesn't arise in the rational case).}  

\begin{remark} In a related direction, it is interesting to give a closer look at a family of examples given by cones over curves. Consider a normal surface singularity whose resolution
has a unique exceptional curve of genus $g\geq 2$ with self-intersection $-d$, for $d>0$.  The resolution is the total space of the complex line bundle of degree $d$ over the corresponding Riemann surface,
and the singularity can be thought of as cone point.  The link is a circle bundle over the genus $g$ surface, with Euler number $-d$. 
The canonical contact structure is the Boothby--Wang structure, which has an open book decomposition as described above: 
the page is a genus $g$ surface with $d$ boundary components, and the monodromy is the boundary multi-twist.

As explained above, for $d\leq 2 g-4$, we have an infinite collection of Stein fillings,
produced by factorizations of the multi-twist. Interestingly, this method no longer applies when $d>4g+4$: in that range, the 
boundary multi-twist admits no non-trivial positive factorizations, again by \cite[Theorem A]{BMVHM}.
On the other hand, for $d>4 g+4$, the singularity is realized by an affine cone over a projective curve and is known to be non-smoothable~\cite{Tendian}. In fact, it is also known that the resolution gives the unique Stein filling in this case \cite[Proposition 8.2]{OhtaOno1}.

Similarly, for cones over elliptic curves, i.e. $g=1$,  the singularity is non-smoothable for 
$d>9$ \cite{Pinkham}, and the only Stein filling is indeed given by the resolution, while for $d \leq 9$, all
Stein fillings are given by smoothings and resolutions, \cite{OhtaOno2}. 
  \end{remark}

\section{Every symplectic filling comes from a symplectic deformation of curvettas} \label{find-curvettas-for-filling}

\subsection{Braided wiring diagrams}
	A braided wiring diagram is a generalization of a braid in $ \R \times \C$ (where the braid condition means that 
	the curves should be transverse to each $\{t\}\times \C$). In a wiring diagram, instead of only looking at smooth 
	braids, we allow the strands to intersect. Let $\pi_\R: \R\times \C\to \R$ denote the projection to the first coordinate. {We will also use the natural projection from $\C$ to $\R$ sending a complex number its real part.}
	

	\begin{definition}\label{def:wiring}
		A \emph{braided wiring diagram} is a union of curves $\gamma_j: \R \to \R\times \C$, $j=1,\dots, n$, each of 
		which is a section of the projection $\pi_\R: \R\times \C\to \R$,  i.e. each ``wire'' is given by
		$\gamma_j(t)=(t, h_j(t)+iw_j(t))$. Different wires $\gamma_j$ may intersect; 
		in this article we will assume that they are not tangent at intersections. 
		
		We say a braided wiring diagram is in \emph{standard form} if there are disjoint 
		intervals $I_1,\dots, I_N\subset \R$ such that $I_\ell \times \C$ contains a unique intersection point
		of some subcollection of the curves $\gamma_j$, and in $I_\ell \times \C$, the wires are given as $\gamma_j(t)=(t, k_jt+a_j+ib_j)$.
		If $\gamma_j$ does not pass through the intersection point, we require $k_j=0$.
	\end{definition}
	
	Note that any braided wiring diagram can be isotoped through braided wiring diagrams to be in standard form.
	
	{We can encode a braided wiring diagram by projecting the union of the images of the $\gamma_j$ to $\R\times \R$ and denoting the crossings of the projection as in a knot diagram.}
	
	A braided wiring diagram can be encoded by sequence $(\beta_0, J_1, \beta_1, J_2, \dots, \beta_{m-1}, J_m, \beta_m)$, where each $\beta_i$ is a 
	braid and $J_i=\{k_i, k_i+1, \dots, k_i+\ell_i\}$ is a consecutive sequence of integers indicating the local indices 
	of the strands involved in the $i^{th}$ intersection point. For brevity, we will say that $J_i$ is a consecutive set.
	
	\textbf{Conventions:} Strands in a wiring diagram are numbered from bottom to top. The convention in \cite{CS} is to draw 
	this sequence of braids and intersections from right to left. If one thinks of composing words
	in the braid group using group notation (left to right) instead of functional notation (right to left), then one will need to read off the braid words from left to right--this is the convention used in \cite{CS}. However, in our case since we are always thinking of braids as diffeomorphisms of the punctured plane, we will use functional notation to compose braid words, and thus read everything--the intersections and the braid words--from right to left.
	
	\begin{example}
		The braided wiring diagram shown in Figure \ref{fig:wiringexample} corresponds to the sequence 
		$$(id, \{2,3\}, id, \{3,4\}, \sigma_1^{-1}\circ \sigma_2^{-1}, \{3,4\}).$$
		
		\begin{figure}[h]
			\centering
			\includegraphics[scale=1]{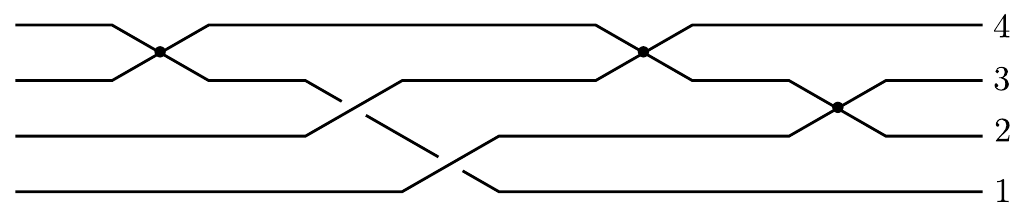}
			\caption{Braided wiring diagram.}
			\label{fig:wiringexample}
		\end{figure}
	\end{example}
	
	Braided wiring diagrams were introduced {in \cite{CS} (inspired by foundational work of \cite{MoTeI} and generalized from diagrams of \cite{Goodman})} to study configurations of complex curves, particularly line arrangements, 
	and the fundamental groups of their complements. The definition works just as well to study configurations of smooth graphical disks 
          in $\C^2$. As in Section~\ref{s:homotopy}, let $(x,y)$ be complex coordinates on $\C^2$, and let $\pi_x$ be the projection to the first coordinate. 
          Let $\Gamma_1,\dots, \Gamma_m$ be smooth disks in $\C^2$ which are graphical with respect to 
	the projection to $x$, $\Gamma_i=\{y=f_i(x)\}$. Assume that all the intersections
	between the $\Gamma_i$'s are transverse and positive (with respect to the natural orientation on the graphical disks projecting to $\C$).
                    
	\begin{definition} 
	A braided wiring diagram of a graphical configuration $\Gamma=\{\Gamma_1, \Gamma_2, \dots, \Gamma_n\}$ of smooth disks in $\C^2$ is obtained as follows.  
	Choose a (real) embedded curve $\eta:[0,1] \to \C$ which passes once through the projection of each singular point of the configuration, {such that the real part $\Re \eta$ is non-increasing.}
	The preimage of the curve $\eta$ under $\pi_x$ in $\C^2$ is diffeomorphic to $[0,1]\times \C$, and the intersection of this copy of $[0,1]\times \C$ with the
	configuration $\Gamma$ is the braided wiring diagram.
	\end{definition}
		
	The transversality of each smooth disk $\Gamma_j$ to the projection $\pi_x$ ensures that the wiring diagram curves are transverse to the projection 
	$\pi_\R: \R\times \C\to \R$. Note different choices of $\eta$ may result in different braided wiring diagrams which are related by certain generalized Markov moves. See for example \cite{CS} for more details. 
	We will show in Section~\ref{s:sympl-constr} that one can always construct a configuration $\Gamma$ with a given braided wiring diagram; moreover,
	the components $\Gamma_j$ of $\Gamma$ can be chosen to be symplectic.
	
\subsection{Braided wiring diagrams to vanishing cycles} \label{s:vanishing}
	Given a configuration $\Gamma=\{\Gamma_1, \Gamma_2, \dots, \Gamma_m\}$  in~$\C^2$ as above, Lemma~\ref{construct-fibration} produces
	an associated Lefschetz fibration. Recall that a Lefschetz 
	fibration is completely determined by its fiber and an ordered list of vanishing cycles. {(Critical points are assumed to have distinct critical values.)} The fiber in this situation is planar with
	$m$ boundary components, where $m$ is the number of curves in the configuration. If we are given a braided wiring diagram of $\Gamma$, 
	we can explicitly determine the vanishing cycles, as follows.
	
	To describe the vanishing cycles of a Lefschetz fibration $L: M\to \C$, we first need to fix certain data.
	Choose a regular fiber $F_0:=L^{-1}(p_0)$ as the reference fiber. Let $p_1,\dots, p_n$ denote the critical values of $L$. 
	Choose paths $\eta_j$ connecting $p_0$ to $p_j$ in the complement of the $p_j$'s, such that the paths
	$\eta_j$ are ordered counterclockwise from $1$ to $n$ locally around $p_0$.
	Then the $j^{th}$ vanishing cycle $V_j$ is the simple closed curve in $F_0$ which collapses to a point under parallel transport along the path $\eta_j$.
	
	When given a braided wiring diagram, we can construct the paths $\eta_j$ in a systematic manner
	and compute the vanishing cycles $V_j$ in terms of the braided wiring data.
	The wiring diagram lies over a curve $\eta:[0,1]\to \C$ such that the real part of $\eta$ is always decreasing. 
	The Lefschetz fibration from Lemma~\ref{construct-fibration} comes from the composition $L:=\pi\circ \alpha$ of the blow-down
	map $\alpha: \C^2\#_n \cptwobar \to \C^2$  with the projection map $\pi_x: \C^2 \to \C$. 
	One then takes the complement of the sections given by proper transforms of the curves $\Gamma_1, \Gamma_2, \dots, \Gamma_m$ in  $ \C^2\#_n \cptwobar $, so that 
	each $\Gamma_j$ corresponds to a hole in the planar fiber. Thus the $j^{th}$ hole corresponds to the wire $\gamma_j$ in the diagram, and in the standard form the holes 
	are arranged vertically in the fiber, {labeled $1,\dots, m$ consecutively}. Each consecutive set $J_i$ corresponds to a subcollection of holes contained in a convex subset of $\C$.  
	The Lefschetz critical points occur in $ \C^2\#_n \cptwobar $ above the intersection points of the braided wiring diagram. 
	Let $0<t_1<\dots< t_n<1$ denote the times such that the  $j^{th}$  intersection point of the wiring diagram lies over $\eta(t_j)$. 
	We will choose our reference fiber to lie over the right end-point $p_0=\eta(0)$ of the curve $\eta$ in $\C$. Strictly speaking, we need a compact version 
	of this construction, which is obtained by working in a closed Milnor ball and taking complements of tubular neighborhoods of $\overline{\Gamma}_i$'s, but for simplicity we omit the Milnor ball from the notation.

	We will choose paths $\eta_j:[0,t_j]\to \C$ given by $\eta_j(t)= \eta(t)-\varepsilon_j \rho_j(t) i$, 
	where $\rho_j:[0,t_j]\to [0,1]$ is a bump function which is $0$ near $t=0$ and $t=t_j$, and $1$ 
	outside a small neighborhood of $0$ and $t_j$, and $0<\varepsilon_1<\varepsilon_2<\dots < \varepsilon_n<\varepsilon$. See Figure~\ref{fig:paths}.
	\begin{figure}
		\centering
		\includegraphics[scale=.75]{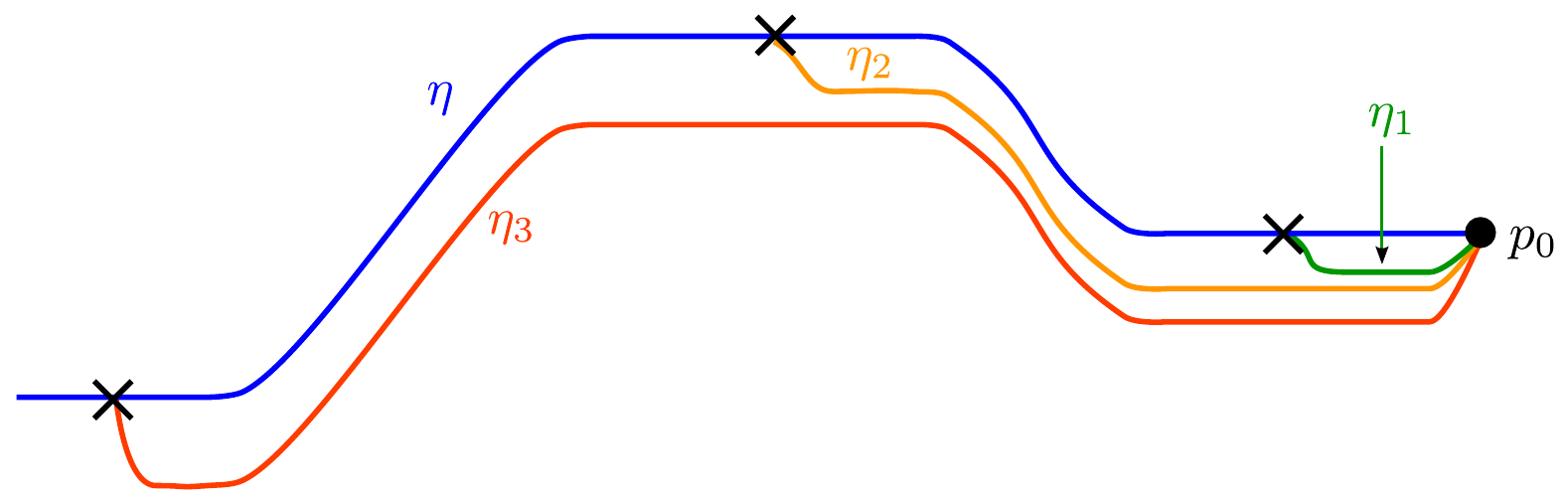}
		\caption{The vanishing paths $\eta_j$ chosen to identify the vanishing cycles in the fiber over $p_0$ relative to the wiring diagram path $\eta$.}
		\label{fig:paths}
	\end{figure}
	
	Our local model for the Lefschetz fibration in Lemma~\ref{construct-fibration} shows that the curve which collapses to a point in the fiber
	$L^{-1}(\eta_j(t_j-\delta))$ (for small $\delta>0$) is a convex curve 
	enclosing the holes in the set $J_j$. To determine the vanishing cycle in our reference fiber $F_0=L^{-1}(p_0)$, we need to 
	track the monodromy over the path $\eta_j$ for $t\in [0,t_j-\delta]$.
	This is the monodromy of the braid given by the intersection of the configuration with the slice of $\C^2$ which projects to $\eta_j$. 
	(Note this intersection is indeed a braid over the interior of $\eta_j$, 
	because each curve $\eta_j$ is disjoint from the critical points away from its endpoints.) 
	By assuming $\varepsilon$ to be sufficiently small, we see that this braid agrees with corresponding portion of the braided wiring diagram, 
	except when passing near an intersection point. 
	When $\eta_j$ passes an interval near $t_k$ for $k<j$, the braid resolves the intersection by separating the strands.
	The strands are ordered from bottom to top in decreasing order by slope in the projection $\R \times \C \to \R \times \R$ 
	(the most positive slope is the lowest strand in the crossing).
	This can be verified by checking the local model for the complexification of real lines because all of our intersections are positive and transverse (see~\cite{MoishezonTeicherI}). 
	After resolving an intersection of the strands in the set $J_k=\{i_k,i_k+1,\dots, i_k+l_k\}$, the element of the mapping class 
	group which corresponds to this portion of the braid from right to left is $\Delta^{-1}$, where $\Delta$ is 
	positive half-twist of the strands $i_k,i_k+1,\dots, i_k+l_k$. {(In terms of the standard generators of the braid group, $\Delta_{J_k}=(\sigma_{i_k}\cdots \sigma_{i_k+l_k-1})(\sigma_{i_k}\cdots \sigma_{i_k+l_k-2})(\sigma_{i_k}\sigma_{i_k+1})(\sigma_{i_k})$.)}
	Therefore, the braid lying above $\eta_j$ is given by
	$$\phi_j = \beta_{j-1}\circ \Delta_{j-1}^{-1}\circ \dots \circ \beta_1\circ \Delta_{1}^{-1}\circ \beta_0,$$
	where $\Delta_{k}$ denotes the positive half-twist of the strands in the set $J_k$.
	Namely, $\Delta_{k}$ is the diffeomorphism supported in a neighborhood of the disk convexly enclosing the holes in the set $J_k$, 
	which acts by rotating the disk by $\pi$ counterclockwise. 
	The $j^{th}$ vanishing cycle is the curve which is taken to the convex curve $A_j$ enclosing the holes in the set $J_j$ under the braid lying above $\eta_j$.
	Therefore, $V_j= \phi_j^{-1}(A_j)$.
	
	\begin{figure}
		\centering
		\includegraphics[scale=1.25]{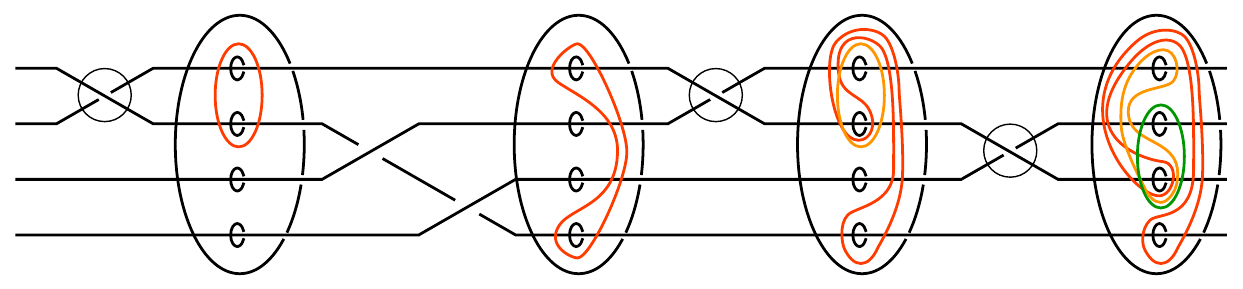}
		\caption{{Vanishing cycles corresponding to the braided wiring diagram of Figure~\ref{fig:wiringexample}. The circled crossings correspond to intersections in the wiring diagram. Uncircled crossings come from braiding between intersections.}}
		\label{fig:vanishingcycles}
	\end{figure}
	
	\begin{remark} \label{rmk:freepts} We can encode blow-ups at ``free'' points (as is allowed by Lemma \ref{construct-fibration}) by 
	adding marked points in our braided wiring diagram indicating ``intersection points'' {that involve}
	only a single strand (so the corresponding $J$ will have $|J|=1$).  \end{remark}
	
	Note that the total monodromy of the curve configuration around a circle enclosing all of the critical points can now be calculated in two different ways:
	\begin{enumerate} 
		\item using the total monodromy of the curve configuration encoded by the braided wiring diagram
		\item taking the product of positive Dehn twists about the induced vanishing cycles.
	\end{enumerate}
	
	To reassure the reader that our formulas and conventions are consistent, we verify that these two different ways of calculating the monodromy agree.
	
	The total monodromy encircling a braided wiring diagram $(\beta_0,J_1,\beta_1,\dots, \beta_{n-1},J_n,\beta_n)$ is given by 
	following the diffeomorphisms induced by a counterclockwise rotation around the wiring interval.
	Such a counterclockwise circle is obtained by connecting an upward push-off of the wire interval oriented right to left 
	with a downward push-off oriented left to right as in Figure \ref{fig:wiremonodromy}. The intersections between the strands of $J_j$ 
	are resolved as the positive half-twist $\Delta_j$ in the upward push-off (right to left). 
	In the downward push-off the intersection is resolved as the negative half-twist $\Delta_j^{-1}$ right to left,
	but since we pass throught the downward push-off from left to right, each such segment contributes $\Delta_j$ to the monodromy.
	The braids contribute $\beta_j$ when traversed right to left and $\beta_j^{-1}$ when traversed left to right. See Figure \ref{fig:wiremonodromy}.
	The total monodromy is therefore
	$$\beta_0^{-1}\circ \Delta_1 \circ \beta_1^{-1} \circ \Delta_2 \circ \beta_2^{-1} 
	\dots \beta_{n-2}^{-1}\circ \Delta_{n-1}\circ \beta_{n-1}^{-1}\circ \Delta_n^2 \circ \beta_{n-1}\circ \Delta_{n-1} 
	\circ \beta_{n-2}\circ \dots \circ \beta_2 \circ \Delta_2 \circ \beta_1 \circ \Delta_1 \circ \beta_0.$$
	
	\begin{figure}
		\centering
		\includegraphics[scale=.8]{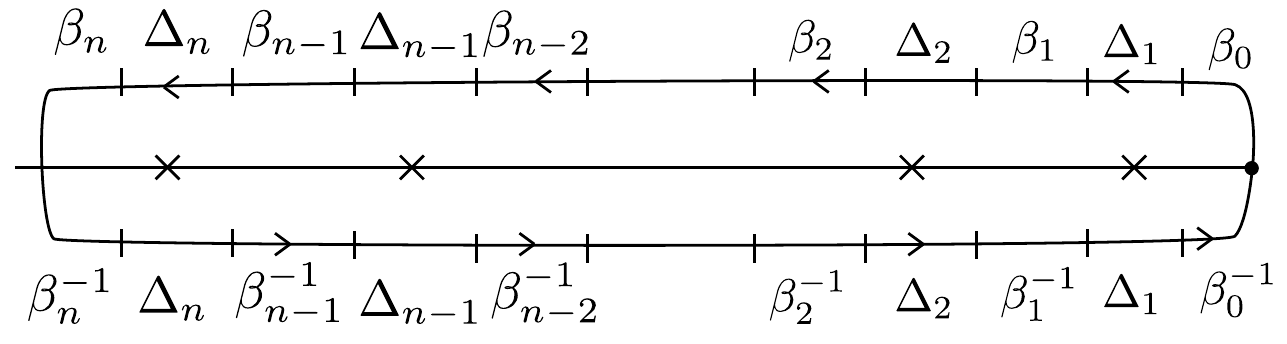}
		\caption{The total monodromy about a braided wiring diagram.}
		\label{fig:wiremonodromy}
	\end{figure}
	
	On the other hand, each vanishing cycle is given as 
	$$V_j = \phi_j^{-1}(A_j)=(\beta_{j-1}\circ \Delta_{j-1}^{-1}\circ \dots \circ \beta_1 \circ \Delta_1^{-1}\circ \beta_0)^{-1}(A_j).$$
	Therefore a Dehn twist $\tau_{V_j}$ about $V_j$  is equal to
	$$\tau_{V_j} = \phi_j^{-1}\circ \Delta_j^2 \circ \phi_j$$
	because $\tau_{A_j} = \Delta_j^2$ and in general $\tau_{\phi(C)} = \phi\circ \tau_C \circ \phi^{-1}$. Thus, the total monodromy of the Lefschetz fibration
	given by the product of positive Dehn twists about the vanishing cycles is 
	$$\phi_n^{-1} \circ \Delta_n^2 \circ \phi_n \circ \phi_{n-1}^{-1}
	\circ \Delta_{n-1}^2 \circ \phi_{n-1} \circ \dots \circ \phi_1^{-1}\circ \Delta_1^2 \circ \phi_1.$$
	We can simplify $\phi_j\circ \phi_{j-1}^{-1}$ as
	$$(\beta_{j-1}\circ \Delta_{j-1}^{-1}\circ \dots \circ \beta_1 \circ \Delta_1^{-1}\circ \beta_0) \circ (\beta_0^{-1}\circ \Delta_1 
	\circ \beta_1^{-1}\circ \dots \circ \Delta_{j-2}\circ \beta_{j-2}^{-1})=\beta_{j-1}\circ \Delta_{j-1}^{-1}.$$
	Therefore $\tau_{V_n}\circ \dots \circ \tau_{V_1}$ is equal to
	$$\phi_n^{-1} \circ \Delta_n^2 \circ (\beta_{n-1} \circ \Delta_{n-1}^{-1}) \circ \Delta_{n-1}^2 \circ \dots \circ
	(\beta_1 \circ \Delta_1^{-1}) \circ \Delta_1^2 \circ \beta_0,$$
	which equals
	$$\beta_0^{-1}\circ \Delta_1\circ \beta_1^{-1} \circ \dots \circ \Delta_{n-1}\circ \beta_{n-1}^{-1} \circ \Delta_n^2 \circ \beta_{n-1} \circ \Delta_{n-1}
	\circ \dots \circ \beta_1 \circ \Delta_1 \circ \beta_0.$$
	This coincides with the total monodromy of the braided wiring diagram given above, as required.
	
\subsection{Wiring diagrams to symplectic configurations} \label{s:sympl-constr}

Given any braided wiring diagram, we interpret it as a collection of intersecting curves in $\R\times \C$. 
We will extend each of these curves to a symplectic surface in $\C\times \C$.

\begin{prop}\label{symp-config} 
	Given a braided wiring diagram $\cup_j \gamma_j\subset \R\times \C$ in standard form, there exists a configuration of 
	symplectic surfaces $\cup_j \Gamma_j$ in $\C\times \C$ such that $\Gamma_j$ extends $\gamma_j$, that is,
	$$
	\left(\cup_j \Gamma_j\right) \cap (\R \times \{0\} \times \C) = \cup_j \gamma_j,
	$$ and all intersections $\Gamma_j\cap \Gamma_k$ 
	lie in the original wiring diagram in $(\R\times \{0\} \times \C)$ and are transverse and positive.
\end{prop}

\begin{proof}
	Let $t_1=\pi_\R(p_1),\dots, t_n=\pi_\R(p_n)$ denote the $\R$ coordinates of the intersection points $p_1,\dots, p_n$
	in the wiring diagram. Braid crossings in the wiring diagram can be viewed as additional intersections that appear in the image of the diagram 
	under the projection $\R\times \C \to \R \times \R$. Choose $\delta>0$ sufficiently small so that there are no crossings in the braided {wiring} diagram 
	in $\pi_\R^{-1}([t_i-4\delta, t_i+4\delta])$ (except the intersection at $p_n$). 
	Let $\rho_i: \R\to [0,1]$ be a smooth bump function such that
	$$\rho_i(t)= \begin{cases} 1 & t\in [t_i-\delta, t_i+\delta]\\ 0 & t\notin (t_i-2\delta, t_i+2\delta)  \end{cases}$$
	Let $\rho = \sum_{i=1}^n \rho_i$. (Figure \ref{fig:rho}.)
	
	\begin{figure}
		\centering
		\includegraphics[scale=.4]{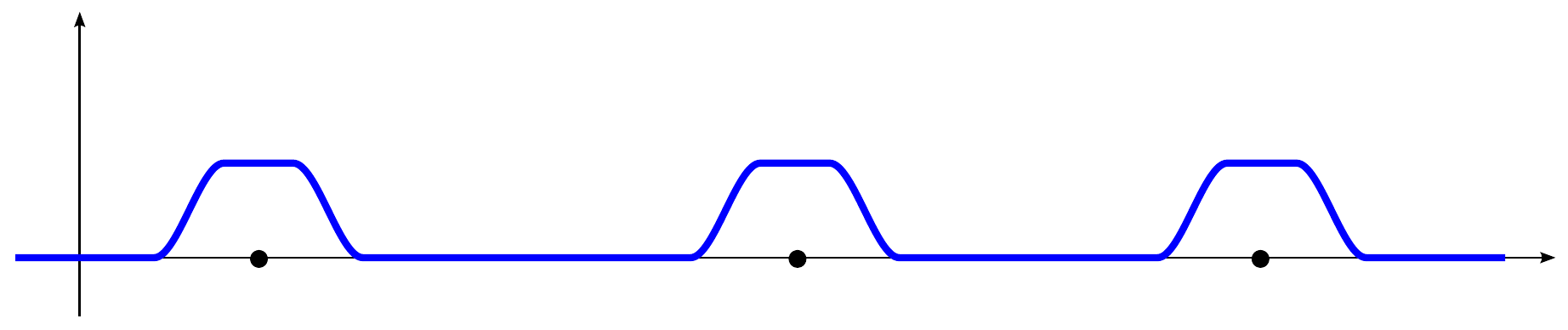}
		\caption{The graph of $\rho: \R\to [0,1]$. The marked points on $\R$ are the $t_i$.}
		\label{fig:rho}
	\end{figure}
	
	Let $\eta>0$. Let $\chi: \R\to [-\eta, \eta]$ (Figure \ref{fig:chi}) be a smooth function such that
	$$\begin{cases} 
	\chi(s) = -\eta & x\leq -2\eta \\
	\chi(s) = s & -\frac{\eta}{2}\leq s \leq \frac{\eta}{2}\\
	\chi(s) = \eta & x\geq 2\eta \\
	\chi'(s)\geq 0 & \forall s\in \R \end{cases}$$

	\begin{figure}
		\centering
		\includegraphics[scale=.3]{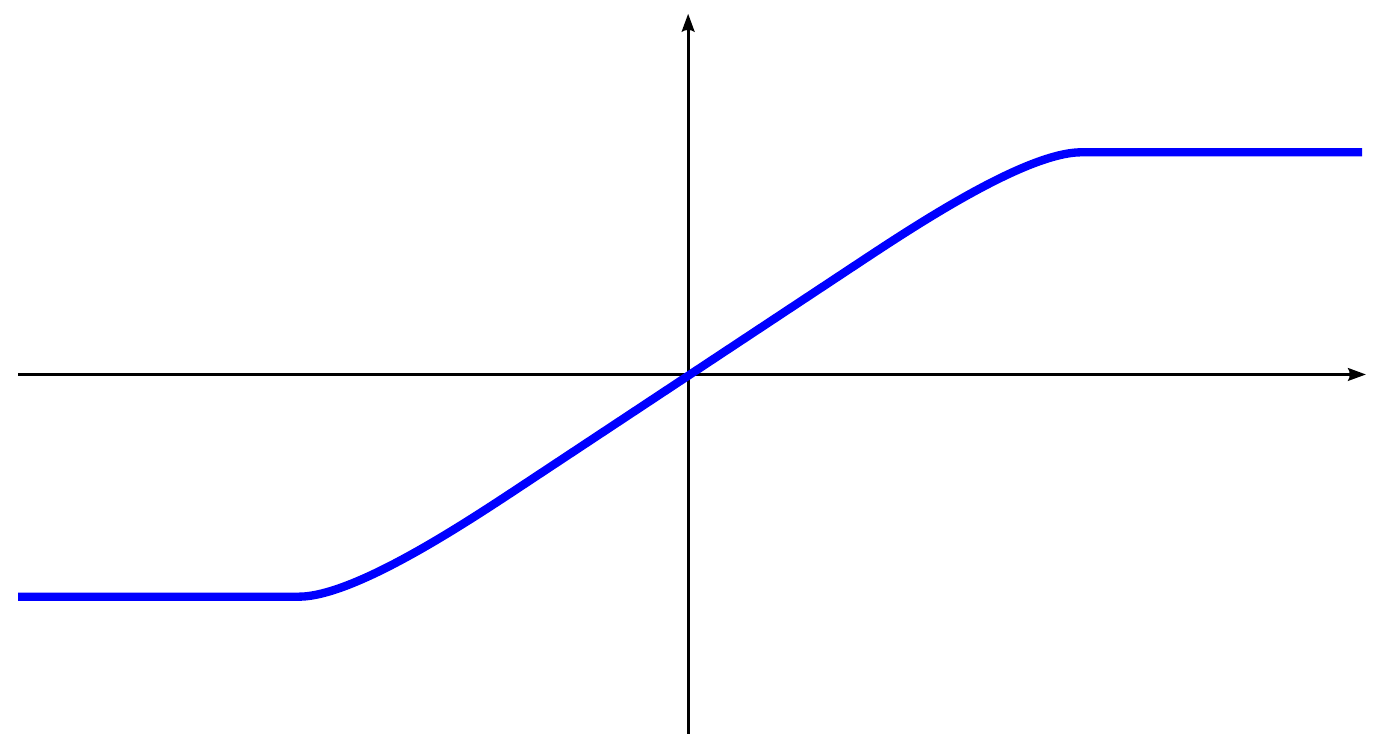}
		\caption{The graph of $\chi: \R\to [-\eta,\eta]$.}
		\label{fig:chi}
	\end{figure}
	
	For each wire, we will define its extension to a symplectic surface. Suppose the wire is parametrized as 
	$$\gamma_j(t) = (t, h_j(t) +iw_j(t))\in \R\times \C.$$
	Define $\Gamma_j(t,s): \R^2 \to \C^2$ by
	$$\Gamma_j(t,s) = (t+is, h_j(t)+i(w_j(t)+\rho(t)\chi(s)h_j'(t))).$$
	The tangent space of the image of $\Gamma_j$ is spanned by $\frac{\partial \Gamma_j}{\partial t}=d\Gamma_j(\frac{\partial}{\partial t})$ and 
	$\frac{\partial \Gamma_j}{\partial s}=d\Gamma_j(\frac{\partial}{\partial s})$. The previous formulas use complex coordinates $(x,y)$ on 
	$\C^2$; now we pass to real coordinates $(x_1,x_2, y_1,y_2)$, so that $x=x_1+i x_2$, $y=y_1+iy_2$.  In these coordinates, the standard symplectic form is 
	given by $\omega = dx_1\wedge dx_2+dy_1\wedge dy_2$.  We have
	
	$$\frac{\partial \Gamma}{\partial t} = \frac{\partial}{\partial x_1} +  h_j'(t)\frac{\partial}{\partial y_1}+(w'(t)+\rho'(t)\chi(s)h_j'(t)
	+\rho(t)\chi(s)h_j''(t))\frac{\partial}{\partial y_2} $$
	$$\frac{\partial \Gamma_j}{\partial s} = \frac{\partial}{\partial x_2} + \rho(t)\chi'(s)h_j'(t) \frac{\partial}{\partial y_2}$$
	Evaluating the symplectic form gives
	$$\omega\left( \frac{\partial \Gamma_j}{\partial t}, \frac{\partial \Gamma_j}{\partial s} \right) = 1+\rho(t)\chi'(s)(h_j'(t))^2>0,$$
	so the image of $\Gamma_j$ is a symplectic surface.
	
	To verify that these extensions do not intersect outside of the original intersections of the wiring diagram, 
	we observe that any intersection between $\Gamma_j$ and $\Gamma_k$ would occur at the same parameters $(t_0,s_0)$
	and must have $h_j(t_0)=h_k(t_0)$ and $w_j(t_0)+\rho(t_0)\chi(s_0)h_j'(t_0)=w_k(t_0)+\rho(t_0)\chi(s_0)h_k'(t_0)$. If $h_j(t_0)=h_k(t_0)$, 
	this means that the wires $\gamma_j$ and $\gamma_k$ project to the same point under the projection $\R\times \C \to \R\times \R$. 
	This means there is either a crossing or an intersection between wires $\gamma_j$ and $\gamma_k$ at $t_0$.
	
	If $t_0$ is an intersection point of the wires, $w_j(t_0)=w_k(t_0)$. Additionally, at $t_0$, the projections of the wires
	have different slopes, so $h_j'(t_0)\neq h_k'(t_0)$. We also have $\rho(t)\equiv 1$ near $t_0$. Using this, the intersection assumption that 
	$$w_j(t_0)+\rho(t_0)\chi(s_0)h_j'(t_0)=w_k(t_0)+\rho(t_0)\chi(s_0)h_k'(t_0)$$ 
	implies that
	$$\chi(s_0)(h_k'(t_0)-h_j'(t_0)) = w_j(t_0)-w_k(t_0)=0.$$
	Therefore, $\chi(s_0)=0$, so $s_0=0$ by definition of $\chi$.
	
	If $t_0$ is a crossing between wires, $w_j(t_0)\neq w_k(t_0)$. Because $\rho$ is supported only near the intersection times, and
	we assume the crossings occur outside of these intervals, $\rho\equiv 0$. Therefore,
	the assumption that $w_j(t_0)+\rho(t_0)\chi(s_0)h_j'(t_0)=w_k(t_0)+\rho(t_0)\chi(s_0)h_k'(t_0)$ gives a contradiction.
	
	Finally, we check that $\Gamma_j$ and $\Gamma_k$ intersect positively. 
	If we assume that the wiring diagram is in standard form near the intersection points $h_j(t)=k_j t+a_j$ 
	and constant coordinate $w_j(t)\equiv b_j$, then near $(t_i,0)$ where $\rho(t)\equiv 1$ and $\chi(s)\equiv s$, we have that 
	$$\Gamma_j(t,s) = (t+is, k_jt+a_j+i(b_j+k_js))$$
	so the image of $\Gamma_j$ agrees with the complex line $y=k_j x+a_j+b_ji$, 
	so the intersection of  $\Gamma_j$ and $\Gamma_k$ locally agrees with an intersection of complex lines.
\end{proof}

\subsection{Stein fillings correspond to symplectic configurations}

Given a contact structure supported by a planar open book, a theorem of Wendl \cite{We} says
that every Stein filling is symplectic deformation equivalent to a Lefschetz fibration with the same planar fiber;
Niederkr\"uger--Wendl \cite{NiWe} extend this result to minimal weak symplectic fillings.
Thus, Stein fillings are essentially in one-to-one correspondence with positive factorizations of the monodromy of the given
planar open book (and the same is true even for weak symplectic fillings, up to blow-up). The following statement is equivalent to Theorem~\ref{thm:intro-symp}.

\begin{prop}\label{filling-to-config} Let $(Y, \xi)$ be the link of a rational
 singularity $(X, 0)$ with reduced fundamental cycle. Fix a decorated germ $(\mathcal{C},w)$ for $(X, 0)$, 
with smooth branches $C_1, C_2, \dots, C_m$. 

Then every Stein filling of $(Y,\xi)$ is supported by a Lefschetz fibration built from a configuration of $m$ symplectic disks $\{\Gamma_1, \Gamma_2, \dots, \Gamma_m\}$
in $\C^2$ with marked points, via Lemma \ref{construct-fibration}. 
\end{prop}

\begin{proof} 
	Because the contact manifold is planar, any Stein filling is supported by a planar Lefschetz fibration with the same fiber. We will 
	reverse-engineer the required configuration of symplectic disks.
	Let $F_0$ be a fixed identification of the planar fiber, where the holes are lined up vertically and labeled by numbers $1, 2, \dots, m$. 
	Let $V_1,\dots, V_n$ be the ordered list
	of vanishing cycles for the Lefschetz fibration.
	We begin by  producing a collection $(\psi_0,\dots, \psi_{n-1})$ of diffeomorphisms $\psi_i: F_0\to F_0$ and $(J_1,\dots, J_n)$ of
	consecutive subsets of $\{1,\dots, m\}$. Here, ``consecutive'' means that  $J_j=\{i, i+1, \dots, i+k\}$ for some $i,k$.
	
	Choose a diffeomorphism $\beta_0: F_0\to F_0$ such that $\beta_0(V_1)$ is isotopic to a curve convexly enclosing a consecutive collection 
	of holes; let $J_1$ be the corresponding consecutive subset. Let $\Delta_1$ be the counter-clockwise half-twist
	of the convex disk that contains precisely the holes indexed by $J_1$. Recursively, choose a diffeomorphism 
	$\beta_{j}: F_0\to F_0$ such that $\beta_{j}\circ \Delta_j^{-1}\circ \dots \circ \beta_1 \circ \Delta_1^{-1}\circ \beta_0(V_{j+1})$ is isotopic
	to a curve convexly enclosing a consecutive collection of holes that corresponds to the set $J_{j+1}$, and let $\Delta_{j+1}$ denote the corresponding half-twist.
	
	Consider the braided wiring diagram determined by $(\beta_0,J_1,\beta_1,J_2,\dots, \beta_{n-1}, J_n)$. 
	By Proposition~\ref{symp-config}, we can construct a configuration of symplectic surfaces $\Gamma_1, \dots, \Gamma_m$  in $\C^2$ extending this diagram. 
	Using Lemma~\ref{construct-fibration}, we obtain a planar Lefschetz fibration. We need to use the compact version of the construction to get a fibration 
	whose general fiber is a disk with $m$ holes; for this, we start with a Milnor ball of the form $B=D_x \times D_y$,
	such that $D_x$ is a neighborhood of $\eta$, and $D_y$ is a disk of sufficiently large radius to include the wires above $D_x$.
		
	As explained in Subsection~\ref{s:vanishing}, the vanishing cycles of this Lefschetz fibration will be given by
	$$V_j' = (\beta_{j-1}\circ \Delta_{j-1}^{-1}\circ \dots \circ \beta_1\circ \Delta_1^{-1}\circ \beta_0)^{-1}(A_j)$$
	for $j=1,\dots, n$, where $A_j$ is a convex curve enclosing the consecutive holes in the set $J_j$. 
	The choice of the $\beta_j$ ensures that these vanishing cycles are identical to our original ones: $V_j'=V_j$.

	Along with the symplectic disk configuration $\{\Gamma_1, \dots, \Gamma_m\}$, we also obtain a collection of marked points on these disks. The marked points 
include all the intersections as well as additional free marked points, as in Remark~\ref{rmk:freepts}.  Each free marked point can be chosen anywhere on the corresponding disk, as long as 
all marked points are distinct.
As in Lemma~\ref{same-monodromy},  counting multiplicities of pairwise Dehn twists in the monodromy shows that the number of marked points 
on each disk $\Gamma_j$ is the same as the weight $w(C_j)$ of the corresponding curvetta 
$C_j$ of  the defining decorated  germ $(C, w)$ of the singularity.
	\end{proof} 
	
	\begin{remark} Note that the diffeomorphisms $\beta_j$ are not unique. Any choice will suffice to produce an appropriate braided wiring diagram and corresponding symplectic configuration. \end{remark}

To show that every Stein  filling is generated by a symplectic analog of de Jong--van Straten's theorem, it remains to prove that different symplectic configurations 
with the same monodromy are related by deformations. The role of de Jong--van Straten's picture deformations is played by graphical homotopies.   

\begin{prop} \label{p:deform} Let $(X, 0)$ be a rational singularity with reduced fundamental cycle, 
	and $(C,w)$  its decorated plane curve germ with smooth branches $C_1, \dots, C_m$. Let  $(Y,\xi)$ be the contact link of $(X,0)$.
	Suppose that $\Gamma=\{\Gamma_1, \Gamma_2, \dots, \Gamma_m\}$ is a configuration of symplectic disks with marked points $p_1, \dots, p_n$, 
	constructed for a given Stein filling of  $(Y, \xi)$ as in Proposition \ref{filling-to-config}.  Then $(\Gamma, \{p_j\})$ can be connected to $(C, w)$ by a smooth graphical
	homotopy.
\end{prop}

\begin{lemma}\label{l:graphicalconvex}
	Suppose $C_1^0,\dots, C_m^0$ and $C_1^1,\dots, C_m^1$ are two configurations of graphical disks in a Milnor ball $B=D_x \times D_y$,
	such that $\partial C_j^0=\partial C_j^1$  for $j=1,\dots, m$. Then there is a family of graphical disks
	$C_1^t,\dots, C_m^t$ (potentially with negative intersections) interpolating between these two configurations with fixed boundary link 
	$\partial C_1^t \cup \dots \cup \partial C_m^t \subset \d B$. Here, $\d C_j^t = C_j^t \cap \d B = C_j^t \cap (\d D_x \times D_y)$.  
\end{lemma}

\begin{proof}
	Because we are not limiting the behavior of the intersections of the components, it suffices to check that 
	there is a family $C_j^t$ interpolating between $C_j^0$ and $C_j^1$ for one component. 
	For simplicity of notation we will drop the $j$. For this, because both $C^0$ and $C^1$ are graphical, we
	can write them as $C^s = \{ (x, f^s(x)) \}$ for $s=0,1$. Then since $\partial C^0 = \partial C^1$, we have that $f^0(x)=f^1(x)$ 
	for $x\in \d D_x$. Let $C^t = \{(x,tf^1(z)+(1-t)f^0(x) \}$. Then $C^t$ interpolates smoothly between $C^0$ and $C^1$, and its boundary is fixed.
\end{proof}

\begin{lemma}\label{l:samebraid}
	Suppose $C_1\cup \dots \cup C_m$ is a configuration of graphical disks, so its boundary $\partial C_1 \cup \cdots \cup \partial C_m$ is a braid. 
	Let $L_1, \dots, L_m$ be the components of a braid $L_1\cup \cdots \cup L_m$ which is braid isotopic (with corresponding indices) to
	$\partial C_1 \cup \cdots \cup \partial C_m$. Then there is a homotopy of graphical disks $C_1^t, \dots, C_m^t$ such that $C_j^0 = C_j$ and $\partial C_j^1 = L_j$.
\end{lemma}

\begin{proof}
	If $C_1,\dots, C_m$ are graphical over a disk $D_x$, choose a larger disk $D'_x$ containing $D_x$. Then we can extend $C_1,\dots, C_m$ 
	to graphical disks $C_1',\dots, C_m'$ over $D'_x$ such that $\partial C_1',\dots, \partial C_m'$ is the braid $L_j$, 
	by realizing the trace of the braid isotopy over the annulus $D'_x\setminus D_x$. 
	Next, we can shrink $D'_x$ to $D_x$ continuously via a family of embeddings $\phi_t: D_x'\to D'_x$ where $\phi_0=id$,
	$\phi_1(D'_x) =D_x$, and $\phi_1$ identifies points in $\partial D'_x$ with points in $\partial D_x$ according to the same identification used to realize the trace.
	Then if $C_j' = \{(x,f_j(x))\}$ for $x\in D'_x$, we can let 
	$$C_j^t = \{(\phi_t(x), {f_j(x)} \mid x\in D'_x \} \cap (D_x\times \C).$$
	Then $C_j^0=C_j$ and $\partial C_j^1 = L_j$ as required.
\end{proof}

\begin{proof}[Proof of Proposition~\ref{p:deform}]
        
	When we fix   the germ $(C,w)$ and apply the method of Proposition~\ref{filling-to-config} to a given  Stein filling for $(Y,\xi)$, we first consider
	the open book on $(Y, \xi)$ induced by the decorated germ as in Proposition~\ref{artin-ob}. The Stein filling then carries a Lefschetz fibration that 
	induces the same open book on the boundary, and the arrangement  $(\Gamma, \{p_j\})$ is constructed from the monodromy of this Lefschetz fibration.  
	The smooth disks $\Gamma_1, \dots, \Gamma_m$ are contained in the Milnor ball $B$ for $\mathcal{C}$ and are transverse to its boundary $S^3$, so that 
	$S^3\cap (\Gamma_1\cup\cdots \cup \Gamma_m)$ is a braid. 
	By Lemma~\ref{lem:braid-monodromy}, the monodromy of this braid is the image of the monodromy 
	of the open book under the projection  $MCG(P_m)\to MCG(\C_m)$ of the mapping class group of the compact disk with holes to the mapping class group of the punctured plane, so the two braids are braid-isotopic.
	Therefore, we can apply Lemma~\ref{l:samebraid} to perform a graphical homotopy to $\Gamma_1,\dots, \Gamma_m$
	so that its boundary agrees with that of $C_1,\dots, C_m$. Next, apply Lemma~\ref{l:graphicalconvex} 
	to continue the graphical homotopy from $C_1,\dots, C_m$ to $\Gamma_1,\dots, \Gamma_m$.
\end{proof}

\begin{remark}
	For our construction of a Lefschetz fibration, it is not important that the $C_i^t$ are symplectic disks, we only care that they are graphical.
	However, by performing a rescaling in the $y$ direction, we can ensure that all of 
	the graphical disks are symplectic if the partial derivatives of the function $f$ are sufficiently small. More specifically, if $C = \{ (x,f(x)) \}$
	where $x=x_1+ix_2$ and 
	$$\left| \frac{\partial f}{\partial x_1} \right|, \left| \frac{\partial f}{\partial x_2} \right| < \frac{\sqrt{2}}{2}$$
	then $C$ will be symplectic. This bound is sufficient although not necessary; it can be achieved by rescaling $f$ which itself is a graphical homotopy.
	Moreover, if $f^0$ and $f^1$ both satisfy these bounds, then 
	their convex combination $tf^0+(1-t)f^1$  also satisfies the bound {for all $t \in[0,1]$}, 
	so the interpolation between the two {disks} will also be symplectic.
\end{remark}

\section{Incidence matrix and topology of fillings} \label{topology}

\subsection{Basic topological invariants} 
It is shown in \cite{dJvS} that the basic topological invariants of the Milnor fibers obtained from the picture deformations 
can be easily computed from the deformed curvetta arrangement.
Moreover, the incidence matrix of the arrangement can be reconstructed from the Milnor fiber, \cite{NPP}.
We now review these facts briefly and adapt and generalize them in our context: the goal is to show that exactly the same results hold for more general Stein fillings, constructed from smooth disk arrangements as in Section~\ref{find-curvettas-for-filling}.

As we have shown in Section~\ref{find-curvettas-for-filling}, every Stein filling $W$ can be described by an arrangement $\Gamma=\{\Gamma_i\}$ 
of  symplectic  curvettas with marked points $\{p_j\}_{j=1}^n$, related to the plane curve germ $\mathcal{C}=\{C_1,\dots, C_m\}$ by a smooth graphical homotopy.
We always assume that curvettas intersect positively. 
The set of marked points $\{p_j\}_{j=1}^n$ contains all intersection points between the $\Gamma_i$'s and possibly a number of free points. The {\em incidence matrix}
$\mathcal{I}(\Gamma, \{ p_j\})$ has $m$ rows and $n$ columns, 
defined so that its entry $a_{ij}$ at the intersection of $i$-th row and $j$-th column equals 1 if $p_j \in \Gamma_i$, and 0 otherwise. Note that there is no 
canonical labeling of the points $p_j$, so the incidence matrix is defined only up to permutation of columns. We will say that two arrangements 
$(\Gamma, \{ p_j\})$ and $(\Gamma', \{ p'_j\})$ are {\em combinatorially equivalent} if their incidence matrices coincide (up to permutation of columns, i.e. up to relabeling of 
the marked points). 

Let $\mathcal{L}$ be the Lefschetz fibration constructed for the arrangement $(\Gamma, \{ p_j\})$ as in  Lemma~\ref{construct-fibration}. Its general fiber is a disk 
with $m$ holes that correspond to the curvettas $\Gamma_1, \dots, \Gamma_m$ of $\Gamma$; in particular, the number of holes equals the number of rows in the matrix 
$\mathcal{I}(\Gamma, \{ p_j\})$. The vanishing cycles of $\mathcal{L}$ correspond to the marked points $\{p_j\}_{j=1}^n$ and enclose sets of holes that correspond to curvettas 
passing through that point: if $\Gamma_{i_1}, \dots, \Gamma_{i_k}$ are all curvettas that intersect at $p_j$,  
the vanishing cycle $V_j$ encloses the holes $h_{i_1}, \dots, h_{i_k}$.  It follows that {\em homology classes} of the vanishing cycles of $\mathcal{L}$ can be determined from 
the incidence matrix $\mathcal{I}(\Gamma, \{ p_j\})$, and we have

\begin{prop} \label{column-holes} 
Let $\mathcal{L}$ be the Lefschetz fibration for the arrangement $(\Gamma, \{ p_j\})$ with incidence matrix $\mathcal{I}(\Gamma, \{ p_j\})$. 
If the $j^{th}$ column of $\mathcal{I}(\Gamma, \{ p_j\})$ has $1$'s in rows $i_1, i_2, \dots, i_k$, the corresponding vanishing cycle $V_j$ of $\mathcal{L}$ encloses the holes
$h_{i_1}, \dots, h_{i_k}$ in the fiber.
\end{prop}

\begin{cor} \label{homologous}
Let $(\Gamma, \{ p_j\})$ and $(\Gamma', \{ p'_j\})$ be two combinatorially equivalent arrangements, and $\mathcal{L}$ and 
$\mathcal{L}'$ the corresponding Lefschetz fibrations. Then the vanishing cycles of $\mathcal{L}$ and~$\mathcal{L}'$ are in one-to-one correspondence, so that the two vanishing 
cycles that correspond to one another are given by homologous curves in the fiber.
\end{cor}

Because smooth graphical homotopies do not allow intersections to escape through the boundary,  the number of pairwise intersections of $\Gamma_i$ and $\Gamma_j$ is  
 given by 
 $tang(C_i, C_j)=\rho( v_i, v_j; v_0)$, see Remark~\ref{rem:toptype}. The weight of $\Gamma_i$ (the total number of intersection points and the free marked 
 points on $\Gamma_i$) is given by $w(C_i)=1+l(v_0, v_i)$. The intersections between $\Gamma_i$ and $\Gamma_j$ correspond to the points among  $p_1, p_2, \dots, p_n$ contained in both lines, and each such point
gives a ``1'' in the same column for the $i$-th row and the $j$-th row of the incidence matrix. Therefore we have

\begin{lemma} \label{where1} Let $(\mathcal{C}, w)$ be a decorated germ corresponding to $(X, 0)$, with branches $C_1, C_2, \dots, C_m$.
Consider any arrangement $\{ \Gamma_i\}_{i=1}^m$ of smooth curvettas
encoding a Stein filling of the link of $(X, 0)$. 
The incidence matrix $\mathcal{I}(\Gamma, \{ p_j\})$ has the following properties:

\noindent (i)  the number of $1$'s in the $i$-th row of $\mathcal{I}(\Gamma, \{ p_j\})$  is $w(C_i)=1+l(v_0, v_i)$; 

\noindent (ii) the number of $1$'s which appear in the same columns for the $i$-th row and the $j$-th row is $tang(C_i, C_j)=\rho( v_i, v_j; v_0)$.

Here,  $l(v_0, v_i)$ and $\rho( v_i, v_j; v_0)$ are the length and overlap functions on the resolution graph $G$, defined in Remark~\ref{rem:toptype}, and $v_0$ is 
the choice of root.
\end{lemma}

We now describe how the incidence matrix $\mathcal{I}(\Gamma, \{ p_j\})$ determines  basic algebraic topology of the filling $W$, 
namely $H_1(W)$, $H_2(W)$, the intersection form of $W$, and the first Chern class $c_1(J)$ of the Stein structure. (Homology is taken with $\Z$ coefficients thoughout.) The statements about the homology and the intersection form of 
$W$ are proved in \cite[Section 5]{dJvS} for the algebraic case, but the proofs are entirely topological and apply in the more general settings as well. Alternatively, the same invariants 
can be computed from the vanishing cycles of the Lefschetz fibration \cite[Lemma 16]{Aur}. For Lefschetz fibrations with planar fiber, detailed proofs for the intersection form and $c_1(J)$
calculations are given in \cite{GGP}. We write $\Z\langle \{p_j\}\rangle$ for the free abelian group generated by $\{p_j\}_{j=1}^n$;
$\Z\langle \{\Gamma_i\}\rangle$ is defined similarly. The incidence matrix    $\mathcal{I}(\Gamma, \{ p_j\})$ defines a map between the corresponding lattices.
\begin{prop}\label{homology} There is a short exact sequence 
$$
0 \longrightarrow H_2(W) \longrightarrow \Z\langle \{p_j\}\rangle \stackrel{\mathcal{I}}\longrightarrow \Z\langle \{\Gamma_i\}\rangle  \longrightarrow H_1(W) \longrightarrow 0.
$$
\end{prop}
\begin{proof}
Let $W$ be the total space of a Lefschetz fibration over a disk $D$, with planar fiber $P$. (We always assume that $W$, $P$, and $D$ are compatibly oriented.)
If $D' \subset D$ is a small disk that contains no critical points, then $W$ is obtained from $P \times D'$ by attaching 2-handles to copies of the vanishing cycles 
contained in the vertical boundary $P\times \d D'$, so that distinct handles are attached along knots contained in distinct fibers. We use 
the exact sequence of the pair $(W, P \times D')$; since $P \times D'$ retracts onto $P$, we can replace the former with the latter. Notice also that  $H_1(W,P) = 0$, so we get
\[
0 \longrightarrow H_2(W) \stackrel{j_*}\longrightarrow H_2(W, P) \stackrel{\d_*}\longrightarrow  H_1(P) \longrightarrow H_1(W) \longrightarrow 0.
\]
The group $H_2(W, P)$ is freely generated by the cores of the attached
2-handles;
we can identify these generators with the vanishing cycles. By construction of the Lefschetz fibration, 
each vanishing cycle corresponds to a blow-up at some marked point, so we can identify the vanishing cycles with the set $\{p_j\}$.  The free abelian group $H_2(W, P)$ is then 
identified with the lattice $\Z \langle\{p_j\}\rangle$. The generators for the free abelian group $H_1(P)$ can be given by loops around the holes in the planar fiber. The holes correspond to 
the branches of $\mathcal{C}$, thus $H_1(P)$ can be identified with the lattice $\Z\langle\{\Gamma_i\}\rangle$. The map $\d_*$ is evaluated as follows: to compute $\d_*(p_j)$, we take 
the boundary of the core of the corresponding 2-handle, given by the vanishing cycle associated with $p_j$, and express this vanishing cycle in terms of the generators of 
$H_1(P) = \Z\langle \{\Gamma_i\}\rangle$. Since the vanishing cycle is a simple closed curve on the planar page, its first homology class equals the sum of the boundaries of the holes it encloses, 
which in turn correspond to the branches $\Gamma_i$ passing through $p_j$. Therefore, $\d_*(p_j)$ is given precisely by the $j$-th column of the incidence matrix  
$\mathcal{I}(\Gamma, \{ p_j\})$, as required. 
\end{proof}

\begin{remark} \label{b1=0} Since the link $Y$ of a rational singularity $(X,0)$ is always a rational holomology 3-sphere, a standard argument shows that $b_1(W)=0$ for any Stein filling $W$ 
of $Y$. Indeed, $W$ has no 3-handles, so  $H^3(W; \Q)=0$; then for the pair $(W, Y)=(W, \partial W)$ we have 
$$
0=H_1(\partial W; \Q) \rightarrow H_1(W; \Q)  \rightarrow H_1(W,\d W; \Q)\cong H^3(W; \Q)=0.
$$
It follows that the matrix $\mathcal{I}(\Gamma, \{ p_j\})$ always has full rank.
\end{remark}

Note that $H_2(W)$ is isomorphic to $\Im j_*$, which in turn equals $\Ker \d_*$. So $H_2(W)$ can be identified with null-homologous 
linear combinations of vanishing cycles (thought of as 1-chains in $P$).  One can explicitly describe an oriented embedded surface in $W$ representing 
a given second homology class, as follows \cite[Section 2]{GGP}. First, one constructs an oriented embedded surface in $P \times D'$  whose boundary is the given 
null-homologous linear combination of the vanishing cycles, and then the vanishing cycles are capped off in $W$. A similar construction is given in \cite{dJvS} 
without Lefschetz 
fibrations, for Milnor fibers obtained by blowing up the 4-ball at the marked points and taking the complement of the proper transforms of curvettas; exactly the same argument 
works for a smooth curvetta arrangement $(\Gamma, \{ p_j\})$. After blowing up the 4-ball $B$ at the points $p_1, p_2, \dots, p_n$, we have the 4-manifold 
$\ti{B}$, the blow-up of $B$, with generators of $H_2(\ti{B})$  given by the fundamental classes $E_{p_i}$ of the exceptional divisors. We identify 
$H_2(\ti{B})= \Z\langle \{p_j\} \rangle$.  The intersection form of $\ti{B}$ is  standard negative definite in the given basis, as $E_p \cdot E_p=-1$. The manifold $W$ is obtained from 
$\ti{B}$ by removing the tubular neighborhoods $T_i$ of the proper transforms $\ti{\Gamma}_i$ of the curvettas $\Gamma_i$. The inclusion induces a map $H_2(W) \to H_2(\ti{B})$, which is in fact 
the same map as $j_*$ above, under obvious identifications. Every homology class in  $H_2(W)$ is represented by an embedded oriented surface which can be constructed
by taking the collection of the corresponding exceptional spheres $E_{p_i}$, punctured at their intersections with $\ti{\Gamma}_j$, and connected by tubes running inside the 
cylinders $T_i$. The intersection of two such surfaces can be computed by taking the intersections of the corresponding collections of exceptional spheres, as the tubes can be 
arranged to be disjoint. For the Stein structure $J$ on $W$ associated to the given Lefschetz fibration, we can compute $c_1(J)$ using the same inclusion
$H_2(W) \to H_2(\ti{B})$. Indeed,  $J$ is homotopic to the restriction of complex structure $j$ on $\ti{B}$ and $c_1(j)[E_{p_i}]= 1$ for every $E_{p_i}$.
Therefore we have  
\begin{prop}\label{intform} The intersection form on $H_2(W) \subset \Z\langle \{p_j\} \rangle$ is the restriction of the standard negative definite
form given by $p_i \cdot p_j = -\delta_{ij}$, $i, j=1, \dots, n$. The first Chern class $c_1(J)$ of the Stein structure is the restriction of the linear form on 
$\Z\langle \{p_j\}\rangle$ given 
by $c_1[p_i]=1$, $i=1, \dots, n$.
\end{prop}
See also \cite[Proposition 2.1, Proposition 2.4]{GGP} for a detailed calculation (in terms of the vanishing cycles) of the intersection form and $c_1(J)$ for an arbitrary 
Lefschetz fibration $(W, J)$  with planar fiber. 

\subsection{Uniqueness of the Artin filling and proof of Theorem~\ref{kollar-fill}}
In general, the topology of the filling might not be fully determined by the incidence  matrix of the corresponding curvettas arrangement; Proposition~\ref{column-holes}
gives the
homology classes of the vanishing cycles but not their isotopy classes.     However,  it turns out that the incidence 
matrix completely determines the smoothing for  picture deformations that are combinatorially equivalent to the Scott deformation, so that one gets the Artin smoothing 
component, \cite[Cases 4.13]{dJvS}. We prove that an analogous result holds for Stein fillings as well.
Note that the argument in~\cite{dJvS} uses simultaneous resolutions and only works in the algebraic setting, while we work with mapping class groups instead. Our 
argument works because the Artin filling has a Lefschetz fibration  with {\em disjoint} vanishing cycles in the fiber.

\begin{prop} \label{artin-fill}
Let $(X,0)$ be a rational surface singularity with reduced fundamental cycle, with contact link $(Y, \xi)$ and decorated germ $(\mathcal{C}, w)$.
Let $\Gamma$ 
be an arrangement of smooth graphical curves with positive intersections and marked points $\{p_j\}$, related to 
the germ $(\mathcal{C}, w)$ by a smooth graphical homotopy, so that  $(\Gamma, \{p_j\})$ gives rise to a Stein filling $W$ of $(Y, \xi)$.

Suppose that $(\Gamma, \{p_j\})$ is  
combinatorially equivalent to the Scott deformation $(\mathcal{C}^s, w^s)$ of $(\mathcal{C}, w)$. 
Then the Stein filling given by $(\Gamma, \{p_j\})$ is Stein deformation equivalent to the Artin filling of $(Y, \xi)$.
\end{prop}

\begin{proof}  Let $\mathcal{L}$ the Lefschetz fibration  for $(\Gamma, \{p_j\})$, constructed as in Lemma~\ref{construct-fibration}, and let 
$\mathcal{L}_A$ be the Lefschetz fibration for the Artin smoothing, given in Proposition~\ref{gay-mark-fibration}.
We know that $\mathcal{L}_A$ is given by the monodromy factorization as in Proposition~\ref{artin-ob}; let $\phi$ denote the monodromy 
of the open book as in the lemma.  

Both fibrations $\mathcal{L}$ and  $\mathcal{L}_A$ 
have the same fiber $S$, and the fibration $\mathcal{L}$ 
corresponds to some factorization of the same monodromy $\phi$.    By Corollary~\ref{homologous}, the vanishing cycles $\{V_j\}$ and $\{V_j^A\}$ of the two fibrations 
are in one-to-one correspondence, so that the curves $V_j$ and $V_j^A$ are homologous in the fiber. We need to show that $V_j$ and $V_j^A$ are isotopic.

There are two types of the vanishing cycles in the fibration $\mathcal{L}_A$: 1) boundary-parallel curves that enclose a single hole each and 2) the curves that go around the necks 
connecting the spheres, as shown at the top of Figure~\ref{fig:outer}. The isotopy class of a boundary-parallel curve in the fiber is uniquely determined by its homology class, so if $V_j^A$ is boundary-parallel, 
then $V_j=V_j^A$. Now, because the total monodromy of $\mathcal{L}$ and $\mathcal{L}_A$ is the same, and 
the Dehn twists around the boundary-parallel curves are in the center of the mapping class group of the fiber, we see that the products 
of the Dehn twists around the vanishing cycles homologous to necks are the same for both  $\mathcal{L}$ and $\mathcal{L}_A$. 
In other words, if $N$ denotes the set the vanishing 
cycles homologous to necks,  we have 
\begin{equation} \label{eq:non-bdry}
\prod_{V_j \in N} \tau_{V_j}  = \prod_{V^A_j \in N} \tau_{V^A_j}.
\end{equation}
Let $\psi$ denote the diffeomorphism of the fiber given by the product~(\ref{eq:non-bdry}).

To prove that each vanishing cycle $V_j$ is indeed isotopic to the vanishing class $V_j^A$ homologous to $V_j$, we proceed by induction on the number of necks in the fiber $S$
(this is the same as the number of edges in the dual resolution graph $G$). Equivalently, we can induct on the number of vertices, since $G$ is a tree. 
When $G$ has only one vertex, there are no necks, so all the vanishing cycles are boundary-parallel, 
and $V_j=V_j^A$ for all pairs of vanishing cycles. Assume that the claim is established for all graphs  with $k$ vertices or fewer. Consider a graph
$G$ with $k+1$ vertices and pick  a leaf vertex $v$ of $G$. We will be able to remove $v$ to reduce the question to a graph $G'$ with $k$ vertices.  

In the Lefschetz fibration of Lemma~\ref{gay-mark-fibration}, the leaf $v$ corresponds to the sphere $S_v$  with holes, connected to the rest of the fiber $S$ by a single neck. 
The fibration $\mathcal{L}_A$ has a vanishing cycle $V^A$  that goes around this neck, and $\mathcal{L}$ has a vanishing cycle $V$ in the same homology class. 
Since $v$ is a leaf, $S_v$ is separated from its complement $S \setminus S_v$ by the curve $V^A$. Observe that all the other non-boundary parallel vanishing cycles of $\mathcal{L}_A$ 
lie outside $S_v$. A priori, non-boundary parallel vanishing cycles of $\mathcal{L}$ may belong to different isotopy classes and intersect $S_v$;
we want to show that they can be isotoped to lie outside $S_v$.

If the self-intersection $v \cdot v=-2$, then in fact $V^A$  encloses only one hole, so it is boundary-parallel, and we can immediately conclude that $V$ and $V^A$ are isotopic, 
and $S_v$ is a boundary-parallel annulus disjoint from all the other vanishing cycles. 

Suppose now that $v \cdot v\leq -3$, so that $V^A$ encloses $r=-1-v \cdot v>1$ holes. Connect these holes by $r-1$ disjoint arcs $\alpha_1, \dots, \alpha_{r-1}$ 
in the sphere $S_v$, so that if the fiber $S$  is cut along these arcs, the $r$ holes will become a single hole, see Figure~\ref{fig:artincut}.

	\begin{figure}[htb]
		\centering
		\includegraphics[scale=.9]{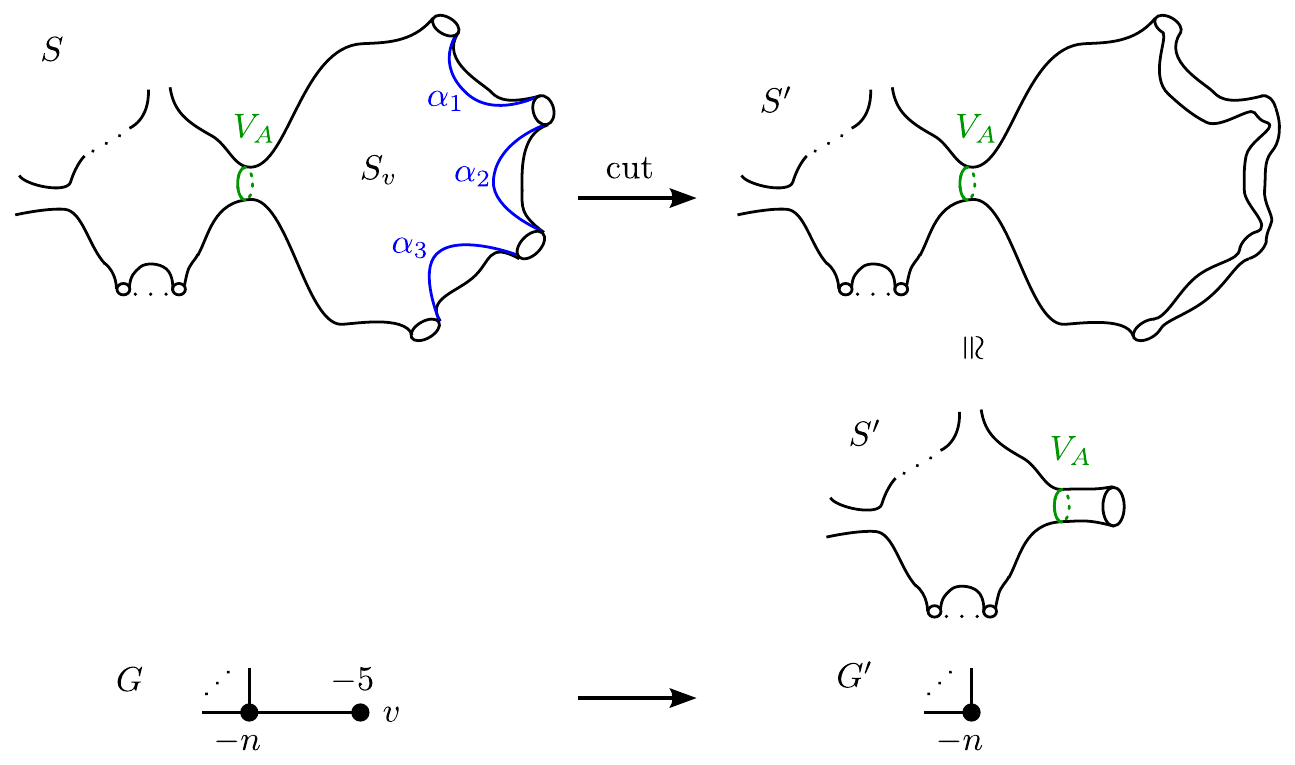}
		\caption{After cutting the fiber $S$, the vanishing cycle $V^A$ becomes boundary-parallel in $S'$.}
		\label{fig:artincut}
	\end{figure}

By construction, the arcs $\alpha_1, \dots, \alpha_{r-1}$ are disjoint from all non-boundary parallel vanishing cycles $V_j^A$ of $\mathcal{L}_A$. It follows that 
each $\alpha_i$ is fixed by the diffeomorphism $\psi$. As in \cite[Proposition~3]{BMVHM} and  \cite[Section 2]{Foss}, we now make the following key observation:  after an isotopy removing non-essential intersections, 
all arcs $\alpha_1, \dots, \alpha_{r-1}$ must be also disjoint from all non-boundary-parallel vanishing cycles $V_j$ of $\mathcal{L}$. To see this, we recall that each 
right-handed Dehn twist is a right-veering diffeomorphism of the oriented surface $S$, \cite{HKM}. If $\alpha$ and $\beta$ are two arcs with the same endpoint $x \in \partial S$, we say that $\beta$
lies to the right of $\alpha$ if the pair of tangent vectors $(\dot{\beta},\dot{\alpha})$ at $x$ gives the orientation of $S$.
The right-veering property of a boundary-fixing map $\tau: S\to S$ means that for every simple 
arc $\alpha$ with endpoints on $\partial S$, the image $\tau(\alpha)$ is either isotopic to $\alpha$ or lies to the right of $\alpha$ at both endpoints,
once all non-essential intersections between $\alpha$ and  $\tau(\alpha)$
are removed.  Now, suppose that $\mathcal{L}$ has  a vanishing cycle $V_j \in N$  that 
essentially intersects one of the arcs, say $\alpha_1$. Then the curve  $\tau_{V_j} (\alpha_1)$ is not isotopic to $\alpha$ (see e.g. \cite[Proposition 3.2]{FM}), 
so $\tau_{V_j} (\alpha_1)$ lies  to the right of $\alpha_1$. Since the composition of right-veering maps is right-veering, 
we can only get curves that lie further to the right of $\alpha$  after composing with the 
other non-boundary parallel vanishing cycles of $\mathcal{L}$. 
However, the composition $\psi=\prod_{V_j \in N} \tau_{V_j}$ fixes $\alpha_1$, a contradiction.

Once we know that no vanishing cycles of $\mathcal{L}$ or $\mathcal{L}_A$ intersect any of the arcs $\alpha_1, \dots, \alpha_{r-1}$, we can cut the fiber $S$ along these arcs, 
and consider the image of the relation~\eqref{eq:non-bdry} in the resulting cut-up surface $S'$. In $S'$, 
$V^A$ becomes a boundary-parallel curve, and since $V$ lies in the same 
homology class, we see that $V$ and $V^A$ are isotopic in $S'$ (and therefore in $S$).   We then have 
$$
\prod_{V_j \in N, V_j \neq V} \tau_{V_j}  = \prod_{V^A_j \in N, V^A_j \neq V^A} \tau_{V^A_j}.
$$
Now observe that cutting up $S$ along the arcs as above has the same effect as removing the sphere $S_v$ with its neck from the set of subsurfaces forming the fiber $S$ in
Lemma~\ref{gay-mark-fibration}.  Then the cut-up fiber $S'$ with its non-boundary parallel vanishing cycles $\{V_j\}$ and $\{V_j^A\}$ corresponds to the fibrations 
for the graph $G'$ obtained by deleting the leaf $v$ and its outgoing edge from the graph $G$. 
By the induction hypothesis, we can conclude that all pairs of homologous vanishing cycles $V_j, V_j^A$  are isotopic in $S'$, and thus in $S$. It follows that the Lefschetz fibrations 
$\mathcal{L}$ and $\mathcal{L}_A$ are equivalent, and therefore the Stein filling given by $\mathcal{L}$  is Stein deformation equivalent to the Artin filling.  
\end{proof}

The above results have the following interesting application, related to conjectures of Koll\'ar on deformations of rational surface singularities.
Suppose that a rational singularity $(X, 0)$ has a   
 dual resolution graph $G$ such that  $v\cdot v \leq -5$ for every vertex $v \in G$. In this case, Koll\'ar's  conjecture asserts that
the base space of a semi-universal deformation of $X$ has just one component, the
Artin component; in particular, there is a unique smoothing, up to diffeomorphism. In the special case of reduced fundamental cycle,
this conjecture was proved by de Jong and van Straten
via their picture deformations method. We establish the symplectic version of this result, proving Theorem~\ref{kollar-fill}.

\begin{proof}[Proof of Theorem~\ref{kollar-fill}]
We can focus on Stein fillings: by \cite{We} and~\cite{NiWe}, 
every weak symplectic filling of a planar contact manifold is a blow-up of a Stein filling, up to symplectic 
deformation. By Section~\ref{find-curvettas-for-filling}, Stein fillings are given by arrangements of symplectic curvettas.  
The argument in \cite[Theorem 6.23]{dJvS} shows that under the given hypotheses on the resolution of $(X, 0)$,
there is a unique {\em combinatorial} solution to the smoothing problem, namely, any arrangement of curvettas must have the same incidence matrix 
as the Artin incidence matrix given by the Scott deformation. De Jong--van Straten's argument  is somewhat involved, so we will not summarize it here, but
we emphasize that the proof of this fact is completely combinatorial and does not use the algebraic nature of arrangements. The same claim holds for 
an arbitrary smooth arrangement subject to the same hypotheses. 
The only input used in \cite{dJvS} is the properties of the incidence matrix determined by the 
resolution graph as in Lemma~\ref{where1}, together with the the following observation:  if all vertices of the 
resolution graph $G$ have self-intersection $-5$ or lower, each end vertex of $G$ (except the root) gets at least three  $(-1)$ vertices  attached in the 
augmented graph $G'$, so that there are at least three corresponding curvettas. An important step in the inductive proof is that the matrix must
have a column where all entries are~1, i.e.
all $\Gamma_i$'s must have a common point. 

Once we know that all arrangements corresponding to possible Stein fillings are combinatorially equivalent to the arrangement given by the Scott deformation, 
Theorem~\ref{kollar-fill} follows from Proposition~\ref{artin-fill}. 
\end{proof}

In the case where additionally the graph $G$ is star-shaped with three legs, uniqueness of minimal symplectic filling (up to symplectomorphism and symplectic deformation) 
was proved by Bhupal--Stipsicz \cite{BhSt}. (They give a detailed proof under the hypothesis that the self-intersection of the central vertex is at most $-10$ 
but mention that one can go up to $-5$ with similar techniques.) Their method relies on McDuff's theorem \cite{McD} and was previously used by Lisca \cite{Li}: one finds a concave symplectic cap which is a plumbing of spheres that 
completes an arbitrary filling to a rational surface, which must be a blow-up of $\cptwo$, analyzes possible configurations of $(-1)$ curves, 
and then verifies that the configurations in the image of the cap plumbing under the blow-down is a pencil of symplectic lines which has a unique symplectic isotopy class.
 {To our knowledge, this strategy 
has not been applied to non-star-shaped graphs in existing literature.} {The difficulty in the non-star-shaped case is that there is not an obvious concave symplectic plumbing which can serve as a cap.}
Our proof works for completely arbitrary trees.

\subsection{Distinguishing Stein fillings}
We now turn to constructions that will be needed in the next section, and explain how to use  incidence matrices to distinguish Stein fillings, at least relative to certain boundary data. 
Indeed, as shown by N\'emethi and 
Popescu-Pampu \cite{NPP}, the 
 incidence matrix is ``remembered'' by the Milnor fiber of the corresponding smoothing, which 
allows us to show that certain Milnor fibers are not diffeomorphic (in the strong sense, i.e. relative to a boundary marking). 
The argument in \cite{NPP} is purely topological, so we can generalize it to arbitrary Stein fillings.
While \cite{NPP} applies more generally to sandwiched singularities, we only consider the case of reduced fundamental cycle.

Instead of the boundary marking used in \cite{NPP}, we will keep track of the boundary data via a choice of a compatible open book for $(Y, \xi)$. 
As in Section~\ref{s:picdef}, we fix a choice of extension $G'$ of the dual resolution graph $G$ of a singularity with link $(Y, \xi)$, to fix the topological type of the associated decorated germ $(\mathcal{C}, w)$ with labeled branches $C_1, \dots, C_m$. Each branch $C_j$ corresponds to a hole $h_j$ of the open book as explained in 
Section~\ref{s:artin}; {fixing the open book, up to isotopy, is equivalent to fixing the topological type of the decorated germ.} In fact, this open book decomposition provides the data of the ``markings'' of \cite{NPP}, where each of the solid tori components of the binding correspond to ``pieces'' of the marking data which allow one to fix the gluing of the smooth cap of \cite{NPP} to the filling using the open book instead of the markings.

By Wendl's theorem \cite{We}, all Stein fillings of a planar contact 3-manifold are given, up to symplectic deformation, by Lefschetz fibrations with same fiber, so that these 
fibrations are encoded by  monodromy factorizations of the fixed open book as above. 
{Suppose that Stein fillings $W$ and $W'$ arise from symplectic curvetta arrangements $(\Gamma, \{ p_j\})$ and $(\Gamma',\{p_j'\})$ as in 
Propositions~\ref{filling-to-config} and~\ref{p:deform}.
On the boundaries $\d W$ and $\d W'$, these arrangements induce open books which are isomorphic, because both are isomorphic to the open book induced by the germ $(\mathcal{C}, w)$. Fix these two open books, {$\mathcal{OB}$ on $W$ and $\mathcal{OB}'$ on $W'$, defined up to isotopy; as part of the open book data, we also label the binding components (with the exception of the outer boundary of the disk, the boundary components of the page correspond to the branches of the decorated germ)}.

 We will say that $W$ and $W'$ are {\em strongly diffeomorphic} if there is an orientation-preserving diffeomorphism 
 $W \to W'$ whose  restriction to $\d W$ maps the open book  on $\d W$ to an open book on $\d W'$ which is isotopic to the given one. If the open book on $\d W'$ is isotopic to the image of the open book on $\d W$,  we can compose the diffeomorphism $W \to W'$ with a 
 self-diffeomorphism of $W'$ which extends the isotopy of $\d W'$ to obtain a diffeomorphism matching the open books. Therefore, we can equivalently say that $W$ and $W'$ are {\em strongly diffeomorphic} if there is an orientation-preserving diffeomorphism $W \to W'$ that identifies the open books $\mathcal{OB}$ on $\d W$ and 
 $\mathcal{OB}'$ on $\d W'$. This identification is required to preserve the labeling of the binding components.
}

Rephrasing the theorem of \cite{NPP} in our context, we have:

\begin{prop}(\cite[Theorem 4.3.3]{NPP})\label{incidence-matrix} Let $(Y, \xi)$ be the contact link of a rational singularity
with reduced fundamental cycle, and fix {the isotopy class of an embedded open book} as above.  Let two {strongly diffeomorphic} Stein fillings $W$ and $W'$ arise from
arrangements $(\Gamma, \{ p_j\})$ and $(\Gamma',\{p_j'\})$ of symplectic curvettas with marked points, as in Section~\ref{find-curvettas-for-filling}.  
Then the incidence matrices $\mathcal{I}(\Gamma, \{ p_j\})$ and $\mathcal{I}(\Gamma', \{ p_j'\})$ are equal, up to permutation of columns.
\end{prop}

\begin{proof}  We outline the proof briefly, referring the reader to \cite{NPP} for details, as we use exactly the same topological argument in a slightly different 
(in fact, simpler) context.

Let $(\mathcal{C}, w)$ be the decorated germ with labeled smooth branches $C_1, \dots, C_m$, determined up to topological equivalence by the open book data for 
$(Y, \xi)$. Unlike \cite{NPP}, we  only work with the case of smooth components of $\mathcal{C}$; therefore, 
all $\delta$-invariants of the branches $C_i$ are 0, and the formulas of \cite{NPP} become 
simpler.

As in \cite{NPP}, we construct a cap $U$, {which is a smooth manifold with boundary}  that can be attached to any Stein filling $W$ of $(Y, \xi)$, so   
that $W \cup U$ is a blow-up of a 4-sphere. To construct $U$, let $B \subset \C^2$ be a closed Milnor ball as in Section~\ref{s:homotopy}, 
so that $B$ contains both the branches of the
germ $\mathcal{C}$ and the arrangement $\Gamma$ together with all intersection points between curvettas $\Gamma_i$. 
Let $(B', \mathcal{C'})$ be another copy of this ball with 
the  germ $\mathcal{C}$ inside, {with reversed orientation.}
After an isotopy of the boundaries of the curvettas $\Gamma_i$ to match $\partial C_i$, 
we can glue 
$(B, \Gamma)$ and $({B}', {\mathcal{C}}')$ so that the boundary of $\Gamma_i$ is glued to the boundary of the corresponding germ branch ${C}'_i$. Note that each disk $\Gamma_i$ is 
 {oriented as a graph over $\C$}, so the result of gluing 
is a smooth 4-sphere $B \cup {B}'$ containing the embedded smooth 2-spheres $\Sigma_i=\Gamma_i \cup {C}_i'$.   Blowing up at the points $p_1, \dots, p_n$, we get
$\#_{i=1}^n \cptwobar$, represented as the blow-up $\tilde{B}$ of the ball $B$
glued to ${B}'$. Let $T_i$ be a thin tubular neighborhood of the proper transform of $\Gamma_i$ in $\tilde{B}$.  Recall that 
by Lemmas~\ref{construct-fibration} and~\ref{same-monodromy}, we have $W= \tilde B \setminus \cup_{i=1}^m T_i$. Set $U = {B}' \cup_{i=1}^m T_i$, so that we have 
$U \cup W = \#_{i=1}^n \cptwobar$. As in \cite[Lemma 4.2.4]{NPP}, the cap $U$ is independent on $W$ and is determined 
by the boundary data. Indeed, to form $U$, we attach 2-handles to the 
4-ball ${B}'$. The attaching circles are given by the boundaries of $\Gamma_i$'s, and the link $\cup_i \partial \Gamma_i$ is isotopic to the link given by the boundaries of 
the branches of the original decorated germ. The framing for $\partial \Gamma_i$ is $-w_i$, the negative weight on the branch $C_i$ of the decorated germ. The proof of
Lemma~\ref{same-monodromy} shows that the weight $w_i$ is given by the number of Dehn twists enclosing the $i$-th hole in (any decomposition of) the monodromy of the open book. 
Thus, the cap $U$ and the way it is glued to $W$ is determined by the decorated germ defining the singularity, together with the fixed open book data of 
$(Y, \xi)$. Finally, as in \cite{NPP}, we see that there is a unique basis $\{e_j\}_{j=1}^n$ for $H_2(\#_{i=1}^n \cptwobar)$ {of classes of square $-1$} such that the intersection numbers 
$\Sigma_i \cdot e_j$ are all positive. It follows that these numbers depend only on $W$ and the open book data. On the other hand, 
the numbers $\Sigma_i \cdot e_j$  form the incidence matrix  $\mathcal{I}(\Gamma, \{ p_j\})$, as $\Sigma_i \cdot e_j=1$ if $p_j \in \Gamma_i$, and 0 otherwise. {It follows that the incidence matrices $\mathcal{I}(\Gamma, \{ p_j\})$ and $\mathcal{I}(\Gamma', \{ p'_j\})$ are the same, up to relabeling the marked points, which amounts to permutation of columns.}
\end{proof}

\begin{remark} Our definition of a strong diffeomorphism and the above proof assumes that the binding components of the open book are labeled, and that the diffeomorphism preserves this labeling. In other words, we think of the page of the open book(s) as a disk with holes, where each hole $h_i$ corresponds to the $i$-th branch of the fixed decorated germ; the diffeomorphism matches the $i$-th hole of the page for $\partial W$ to the $i$-th hole for $\partial W'$.  It is in fact possible to  consider a less restrictive definition of strong diffeomorphisms, by allowing permutations of binding components, and to prove a sightly stronger version of Proposition~\ref{incidence-matrix} and Theorem~\ref{thm:nonMilnor}.  More precisely, the proposition still holds if there is a diffeomorphism $f: W \to W'$ that sends the chosen open book $\mathcal{OB}$ on $\d W$ to an open book on 
$\partial W'$ which is isotopic to $\mathcal{OB}'$, in the sense of isotoping the binding and the pages, but the isotopy matches the binding components in a wrong order. Moreover, it is plausible that the proposition still holds if we only have a diffeomorphism $W \to W'$ whose restriction to $\d W$ takes the binding of the open book $\mathcal{OB}$ to an oriented link which is isotopic to the binding of $\mathcal{OB}'$ on $\d W'$ 
(because $\d W=Y$ is a link of rational singularity, and thus a rational homology sphere, it seems possible to use \cite{Thurston} to construct an isotopy of pages of the open books if their bindings are isotopic, perhaps under some mild additional hypotheses). We leave most of the details to the motivated reader, only indicating below why the proposition should hold if the identification of the open books permutes the binding components. 
It should be emphasized that these arguments would yield only a mild generalization of Proposition~\ref{incidence-matrix}: fixing appropriate boundary data is crucial for our proof. Note that by Wendl's theorem, all Stein fillings of a planar contact manifold fill {\em the same} open book; so in this sense, it is reasonable to think of the boundary open book as fixed.

To consider the case where the diffeomorphism between the fillings permutes the binding components of the 
open book, assume that there is  an orientation-preserving self-diffeomorphism $\sigma$ of the page of the open book that commutes with the monodromy. We do not assume that $\sigma$ fixes the boundary of the page; in particular, we are interested in the case where
$\sigma$ permutes the boundary components.  It can be shown that if $\sigma$ acts non-trivially on the set of boundary components, then the decorated germ and/or the resolution graph of the singularity has the corresponding symmetry. For example, if $\sigma$ exchanges holes $h_1$ and $h_2$, these holes must be enclosed by the same number of Dehn twists (in any positive factorizations of the open book); this implies, in partucular, that equality of weights for the corresponding curvetta branches, $w_1=w(C_1)=w(C_2)=w_2$. Additionally, for any other hole $h_i$, the number of Dehn twists enclosing the pair $h_1, h_i$ must be the same as the number of Dehn twists enclosing the pair $h_2, h_i$. Because the Artin factorization is determined by combinatorial data (see Proposition~\ref{artin-fill}), it follows that the Artin factorization admits a symmetry interchanging holes $h_1$ and $h_2$. Then, we can argue as in Proposition~\ref{Artin-factor} to reconstruct the resolution graph of the singularity, and to see that the graph must have a symmetry, and the corresponding curvetta arrangement must admit a symmetry interchanging curvettas $C_1$ and $C_2$ (up to a topological equivalence). A similar reasoning would work for a more general self-diffeomorphism $\sigma$; we do not give the complete argument to avoid setting up complicated notation.  If $\sigma$ exchanges the boundary of a hole with the outer boundary of the page (thought of a disk with holes), there must be a symmetry of the resolution graphs and the corresponding extended graphs (see Section~\ref{s:picdef}). 

Since the self-diffeomorphism $\sigma$ of the page commutes with the monodromy, it induces  a self-diffeomophism of the supporting 3-manifold $Y$, which is not necessarily isotopic to the identity. We will use the same notation for this self-diffeomorhism of $Y$, $\sigma:Y\to Y$.  

Now, suppose that fillings $W$ and $W'$ are as in Proposition~\ref{incidence-matrix}, and that 
there is an orientation-preserving diffeomorphism $f:W \to W'$ that maps the open book $\mathcal{OB}$ on $W$
to the open book $f(\mathcal{OB})$ on $W'$ that is isotopic to 
$\sigma(\mathcal{OB}')$ rather than to $\mathcal{OB}'$. As explained above, the decorated germ admits a symmetry induced by $\sigma$; in turn, it follows that the cap $U$ admits a self-diffeomorphism that restricts to the map   
$\sigma: Y\to Y$ on the boundary, after an orientation reversal. Using this self-diffeomorphism to glue the cap to $W'$, and comparing $W \cup_{id} U$ and $W' \cup_{\sigma} U$, we can argue as in   Proposition~\ref{incidence-matrix} to conclude that the incidence matrices $\mathcal{I}(\Gamma, \{ p_j\})$ and $\mathcal{I}(\Gamma', \{ p'_j\})$ are the same.

\end{remark}

\section{Milnor fibers and unexpected Stein fillings: examples} \label{s:examples}

We now construct examples where the link of a rational singularity with reduced fundamental cycle has Stein fillings that are not realized by  Milnor fibers of any smoothing. 

Our examples build on results of the previous sections: by \cite{dJvS},  Milnor fibers of smoothings correspond to (algebraic) picture deformations
of the decorated germ, while  Stein fillings of the link can be constructed from arbitrary smooth graphical homotopies of the curvettas. 
During the picture deformation, the decorated germ $\mathcal{C}$ is {\em immediately} deformed into an arrangement 
of curvettas yielding a Milnor fiber, so that the arrangement appears as the deformation $\mathcal{C}^s$
for small $s$ (and for a given deformation, all values of $s$ close to 0 produce diffeomorphic  Milnor 
fibers and equivalent Lefschetz fibrations). Indeed, for an algebro-geometric  1-parameter deformation of the germ $\mathcal{C}$, 
the general fibers of the deformation all ``look the same'' (up to diffeomorphism). By contrast, during the course of a
smooth graphical homotopy, we are allowed to change the topology of the arrangement of curvettas, and thus will produce Stein fillings whose topology varies during the homotopy. 
We emphasize that {\em immediate deformation} vs {\em long-term homotopy} of the branches of $\mathcal{C}$ makes the key 
difference between Milnor fillings and Stein fillings of links of rational singularities with reduced fundamental cycle.
In Section~\ref{s:further},  we  explain why this is the key aspect and  compare  picture deformations and smooth graphical homotopies in more detail.
In this section, we exploit the difference between immediate deformations and long-term homotopies to produce examples of Stein fillings
that are not diffeomorphic (rel boundary) to any Milnor fibers.

\subsection{Arrangements of symplectic lines and pseudolines}
To construct links of singularities that admit unexpected Stein fillings, we first consider decorated germs given by pencils of lines (with weights) and focus on 
their associated singularities. 
{In this section, we will use the following terminology: several points are {\em collinear} if they all lie on the same line, and several lines are {\em concurrent} if they all pass through the same point. Concurrent lines form a {\em pencil}; we will refer to an arrangement of concurrent lines as a {\em pencil of lines}.  We will also talk about concurrent pseudolines or concurrent smooth disks, with the same meaning.}

Note that any two pencils of complex lines in $\C^2$ are isotopic through pencils, therefore the corresponding 
singularities are topologically equivalent and have contactomorphic links. Let $\mathcal{C}=\{C_1, C_2, \dots, C_m\}$ be a pencil of $m$ complex lines, with each 
line $C_k$  decorated by a weight $w_k=w(C_k)$.  Consider the  surface singularity  that corresponds to the decorated germ $(\mathcal{C}, w)$, 
and let $Y(m, w)=Y(m; w_1, \dots, w_m)$ denote its link with the canonical contact structure $\xi$. Note that $Y(m, w)$ is a Seifert fibered space over $S^2$ with 
at most $m$ singular fibers.  
Indeed, consider the dual resolution graph of the singularity; the graph gives a surgery diagram for the link. This graph has $m$ legs emanating from the central vertex. 
Legs correspond to the lines of the pencil, so that $k$-th leg has $w_k-1$ 
vertices (including the central vertex).

Note that legs of length $1$ consist only of the central vertex and thus will appear invisible. However, in the examples we focus on, every leg will have length greater than $1$.
The central vertex has self-intersection $-m-1$, all the other vertices have self-intersection $-2$. See Figure~\ref{pencil} for an example.
The decorated pencil $\mathcal{C}$ can be recovered from the graph as in 
Section~\ref{s:picdef}: we add $(-1)$ vertices at the end of each leg, take the corresponding collection of curvettas, and blow down the augmented graph.

To construct Stein fillings of $Y(m,w)$, we will use curvetta homotopies taking the pencil of complex lines to a symplectic line arrangement in $\C^2$. 
We define these arrangements as follows.
\begin{definition}
	A \emph{symplectic line arrangement} in $\C^2$ is a collection of $m$ symplectic graphical disks $\Gamma_1,\dots, \Gamma_m$ in $\C^2$ with respect to a projection $\pi:\C^2\to \C$ such that
	\begin{enumerate}
		\item \label{i:int} for every pair $i,j\in\{1,\dots, m\}$ with $i\neq j$, $\Gamma_i$ intersects $\Gamma_j$ positively transversally exactly once, and
		\item \label{i:monodromy} for $R$ sufficiently large, $(\Gamma_1\cup\cdots \cup \Gamma_m)\cap \pi^{-1}(S_R)$ is isotopic to the braid given by one full twist on $m$ strands in the solid torus $\pi^{-1}(S_R)$ (where $S_R\subset \C$ is the circle of radius $R$).
	\end{enumerate}
\end{definition}
Equivalently, we can view the symplectic line arrangement in a Milnor ball $B=D_x \times D_y \subset \C^2$ containing all intersections. The intersection 
of the arrangement with $\partial B$ is then the braid of one full twist in $\partial D_x \times D_y$. 
A symplectic line arrangement in the closed ball $B$ can 
always extended to an arrangement in $\C^2$, so we will give all statements about symplectic line arrangements in $\C^2$.

\begin{example}
	A pencil of complex lines intersecting at the origin in $\C^2$ is a symplectic line arrangement. 
	Clearly every pair of lines intersects at a single point (the origin) transversally (and positively because they are complex). 
	That the monodromy in $\pi^{-1}(S_R)$ is one full twist on $m$ strands can be computed directly from a model as in \cite{MoTeI}.
	
	More generally, any complex line arrangement of $m$ lines in $\C^2$ such that no intersections between lines
	occur at infinity (i.e. every complex line has a different complex slope) gives a symplectic line arrangement. 
	This can be seen by compactifying the line arrangement in $\cptwo$ and looking at the intersection of the lines 
	with the boundary of a regular neighborhood of the $\cpone$ at infinity. These intersections form an $m$ 
	component link with one component for each line, such that the link components are isotopic to disjoint fibers of 
	the $\varepsilon$-neighborhood (which can be identified with a subset of the normal bundle) of the $\cpone$ at infinity. 
	After changing coordinates from the perspective of the $\cpone$ at infinity to the perspective of the complementary ball, 
	the components of the link obtain one full twist. From the Kirby calculus perspective, the boundary of the $\varepsilon$-neighborhood 
	of $\cpone$ is presented as $(+1)$ surgery on the unknot, and the link is $m$ parallel meridians of this surgery curve. 
	After reversing orientation to get the boundary of the complementary ball, the surgery coefficient on the unknot becomes a $(-1)$ surgery, 
	and blowing down this surgery curve induces one full twist in the $m$ unknotted meridians.
\end{example}


Since any symplectic line arrangement has the same monodromy as the pencil of complex lines, Lemmas~\ref{l:graphicalconvex} and~\ref{l:samebraid} imply they are related to the pencil by a smooth graphical homotopy.

Our primary source of examples of non-complex symplectic line arrangements is given by pseudoline arrangements as described below. 
However, symplectic line arrangements are more general and can include braiding in the associated wiring diagram.

\begin{example} \label{discuss-plines}
	A {\em pseudoline arrangement} is a collection $\ell_1,\dots, \ell_m$ of smooth graphical curves
	in $\R^2$ where for every pair $i,j$, $\ell_i$ and $\ell_j$ intersect transversally
	at exactly one point. Such a pseudoline arrangement can be considered a braided wiring 
	diagram as in Definition~\ref{def:wiring}, but in the particular case where there is no braiding. 
	In particular, we can apply Proposition~\ref{symp-config} to extend the pseudoline
	arrangement to an arrangement of symplectic graphical disks $\Gamma_1,\dots, \Gamma_m$; the extension produces a symplectic line arrangement. Indeed, 
	condition (\ref{i:int}) in the definition of a symplectic line arrangement is satisfied because any two pseudolines intersect transversally at one point, and their 
	extensions intersect positively by construction. Condition (\ref{i:monodromy}) follows
	from the calculation of the total monodromy as in Section~\ref{s:vanishing} and a classical theorem of Matsumoto and Tits~\cite{Matsumoto}
	about uniqueness of reduced factorizations in the braid group. 
	
	Alternatively, we can refer to the results of \cite[Section 6]{RuSt}, where  pseudoline arrangements in $\rptwo$ are  extended to symplectic
	line arrangements in $\cptwo$ (extensions in $\cptwo$ are strictly harder to construct than extensions in $\C^2$.) Additionally, using the same theorem of 
	 Matsumoto and Tits, \cite[Proposition 6.4]{RuSt} provides a homotopy of pseudoline arrangements connecting the given 
	arrangement to the pencil. After applying Proposition~\ref{symp-config}, we get  a homotopy of the corresponding symplectic line arrangements. 
	{Note that by construction,  this}  homotopy of symplectic line arrangements keeps 
	all intersections positive at all times, whereas the smooth graphical  homotopy given by 
	Lemmas~\ref{l:graphicalconvex} and~\ref{l:samebraid} may introduce negative intersections.
\end{example}	

We use symplectic line and pseudoline arrangements to construct Stein fillings of Seifert fibered spaces $(Y(m;w), \xi)$ via
Lemma~\ref{construct-fibration} and Lemma~\ref{same-monodromy}.

\begin{prop} \label{p:symp-lines-fill} Let $(\mathcal{C}, w)$ be a decorated pencil of $m$ lines. Suppose that 
$\Gamma =\{\Gamma_1,\dots, \Gamma_m \}$ is a  symplectic line arrangement such that each disk $\Gamma_i$ has at most $w_i$ 
distinct intersection points with the other disks of the arrangement. Then, $(\Gamma, \{ p_j\})$ yields a Stein filling of 
$(Y(m;w_1, w_2, \dots, w_m), \xi)$.  

In particular, a pseudoline arrangement $\Lambda =\{\ell_1,\dots, \ell_m \}$ gives a Stein filling of  $(Y(m;w_1, w_2, \dots, w_m), \xi)$ via an extension to a symplectic line arrangement, provided that $\ell_i$ has at most $w_i$ distinct intersection points with the other pseudolines.
\end{prop}


\subsection{Unexpected line arrangements yield unexpected fillings}
Now we will show that some of the Stein fillings as above do not arise as Milnor fibers. In the next lemma, 
we consider analytic deformations of reducible plane curve germs, associated to a singularity by the De Jong--Van Straten theory, and establish a property that will play a key role in our construction of unexpected arrangements.

{The term {\em $\delta$-constant} deformation in the next lemma refers to an algebro-geometric property: the deformation is required to preserve the $\delta$-invariant of a singular plane curve. We  keep this terminology since it is used in \cite{dJvS, NPP}; however, under the hypothesis that the germ has smooth branches, 
the $\delta$-constant condition simply means that the deformation changes the germ component-wise, without merging different components. Intuitively, the $\delta$-invariant counts the number of double points ``concentrated'' in each singular point \cite[Section 10]{Milnor}; for example, an ordinary $d$-tuple point (where $d$ smooth components meet transversely) contributes 
$\delta=\frac12 d(d-1)$, since it can be perturbed to $\frac12 d(d-1)$ double points. Thus, we can deform a triple point to three double points by a $\delta$-constant deformation, but we are not allowed to deform two transversely intersecting lines into a smooth conic (such a deformation would kill a double point).}      


\begin{lemma} \label{get-lines} Consider the germ of a reducible plane curve $\mathcal{C}$ 
	in $\C^2$ with $m$ smooth graphical branches $C_1, C_2, \dots, C_m$ passing through $0$, and let  
	$\mathcal{C}^s = \cup_{k=1}^m C^s_k $ be a $\delta$-constant deformation of $\mathcal{C}$. 
	{(Here, {\em $\delta$-constant} means that each branch of the germ is deformed individually, i.e. the deformation is not allowed to merge different branches.)}
	Suppose that all the branches $C_1, \dots, C_m$ have distinct tangent lines at $0$, and that not all deformed branches $C_1^s, \dots, C_m^s$ are concurrent for $s \neq 0$.

	Then there exists a complex line arrangement $\mathcal{A}=\{L_1, \dots, L_m\}$ in $\C^2$ such that not all lines in $\mathcal{A}$ are concurrent, no two lines are equal, and $\mathcal{A}$ satisfies all the incidence relations of $\mathcal{C}^s$. Namely, for any  collection of the deformed branches $C_{i_1}^s$, $C_{i_2}^s$, \dots, $C_{i_k}^s$ that intersect at one point,  
	the corresponding lines $L_{i_1}, L_{i_2}, \dots, L_{i_k}$ also intersect:
	\begin{equation} \label{eq:incidence}
	C_{i_1}^s \cap C_{i_2}^s\cap \dots \cap C_{i_k}^s \neq \emptyset \Longrightarrow L_{i_1} \cap L_{i_2}\cap \dots \cap L_{i_k} \neq \emptyset.
	\end{equation}
\end{lemma}

{Note that the incidence pattern for branches of $\mathcal{C}^s$ is the same for all $s\neq 0$, because the definition of a 1-parameter deformation implies that all nearby fibers ``look the same''. 
It's important to keep in mind}  that the complex line arrangement $\mathcal{A}$ may satisfy additional incidences, so that certain intersection points coincide in $\mathcal{A}$
but are distinct for the arrangement $\{ C_1^s, C_2^s, \dots, C_m^s \}$. In particular, a pencil of lines would satisfy incidence relations 
of any other arrangement, but we postulate that $\mathcal{A}$ cannot be a pencil (the lines in $\mathcal{A}$ are not all concurrent).

\begin{proof}[Proof of Lemma~\ref{get-lines}]
Since any two curvettas intersect positively
in the original germ $\mathcal{C}$, any two deformed branches $C_i^s$ , $C_j^s$ intersect for $s\neq 0$.
We can make an $s$-dependent translation to ensure that the first two branches always intersect
at the origin, $C_1^s \cap C_2^s = \{0\}$; strictly speaking, this means passing to a slightly different deformation of
the germ $\mathcal{C}$.

All components of the reducible curve $\mathcal{C}$ pass through $0$ and are graphical analytic disks with respect to the projection to the $x$-coordinate. 
Thus we can define the germ of 
$\mathcal{C}$ near $0$ by an equation of the form  
$$
\prod_{i=1}^m (a_i  x + c_i(x)-y)=0,
$$
where $c_i(x)=\sum_{k\geq 2} c_{i,k} x^k$  are analytic functions in $x$ with $ord_{x} c_i >1$ at $0$. 
We can also assume that $a_i \neq 0$ for all $i=1, \dots, m$.

The 1-parameter deformation $\mathcal{C}^s$ is then given by, for $s$ close to $0$, by an equation of the form 

$$
\prod_{i=1}^m (a_i(s) x + b_i(s) +c_i(x,s)-y)=0. 
$$
Here  $a_i$, $b_i$ are analytic functions in $s$, and {at the origin $(0,0)$} $ord_s a_i=0$, $ord_s b_i>0$; additionally,    $c_i(x, s)$ is analytic in $x,s$, and $ord_{x} c_i>1$. 
The $i$-th component $C_i^s$ of the deformed curve at time $s$ is given by the equation $a_i(s) x + b_i(s) +c_i(x,s)-y=0.$ Because the branches $C_1^s$ 
and $C_2^s$ pass through 0 for all $s$, we have $b_1\equiv b_2\equiv 0$. At $s=0$ all components pass through the origin so $b_i(0)=0$ for all $i$.

Let $r = \min_i (ord_s b_i)$ {(where the order is always taken at the origin)}. Because $b_i(0)=0$ for all $i$, we have $r>0$, and $r=ord_s b_{i_0}$ for some $3 \leq i_0 \leq m$.  Notice also that $r<+\infty$, since otherwise all the components $C_i^s$ would pass through $0$ for all $s\neq 0$.  
We  write $b_i(s) = s^r \bar{b}_i(s)$; then $\bar{b}_{i_0}(0)\neq 0$. 

Now make a change of variables {for $s\neq 0$}
$$
x= s^r x', \qquad y = s^r y'. 
$$
Since $ord_{x} c_i(x,s) \geq 2$, we have $c_i(x,s)= s^{2r} \bar{c}_i(x', s)$ for some analytic function $\bar{c}_i$. 
Thus, the equation for the deformation becomes
$$
\prod_{i=0}^m  (a_i(s) s^r x' +  s^r \bar{b}_i(s)  + s^{2r} \bar{c}_i(x',s)  - s^r y' )=0.
$$
Equivalently, for  $s\neq 0$ and $i=1, \dots, m$, the deformed components $C_i^s$ are given by the equations 
$$
a_i(s)  x' +   \bar{b}_i(s)  + s^{r} \bar{c}_i(x',s)  - y' =0.
$$
When we pass to the limit as $s \to 0$, 
the equations become 
$$
a_i(0)  x' +   \bar{b}_i(0)  - y' =0,
$$
so in the limit we obtain an arrangement of of straight lines in $\C^2$.  Not all of these lines are concurrent, since  
$\bar{b}_{i_0}(0)\neq 0$ while $\bar{b}_1(0)=\bar{b}_2(0) = 0$. 

The curves $C_i^s$ satisfy the same incidence relations for all $s \neq 0$. Since intersection points between curves 
vary continuously with $s$,  the incidence relations must be preserved in the limit, so~(\ref{eq:incidence}) holds. \end{proof}

Our examples of unexpected Stein fillings are given by pseudoline arrangements with the following special property.
\begin{definition}  Let  $\Lambda=\{\Gamma_1, \dots \Gamma_m\} \subset \R^2$ be a symplectic line arrangement where not all lines are concurrent. We say that $\Lambda$ is {\em unexpected} if the only 
	complex line arrangements that satisfy all the incidence relations of $\Lambda$ are pencils of lines. Namely, whenever
	a complex line arrangement  $\mathcal{A}=\{L_1, L_2, \dots, L_m\} \subset \C^2$ has the property
	$$
	\Gamma_{i_1} \cap \Gamma_{i_2}\cap \dots \cap \Gamma_{i_k} \neq \emptyset \Longrightarrow L_{i_1} \cap L_{i_2}\cap \dots \cap L_{i_k} \neq \emptyset,
	$$
	all the lines $L_1, L_2, \dots, L_m$ of $\mathcal{A}$ must be concurrent.
	
	If an unexpected symplectic line arrangement comes from a pseudoline arrangement, we will say that the pseudoline arrangement is unexpected.
\end{definition}

\begin{remark} It is important to note that unexpected symplectic line arrangements are not the same as symplectic line arrangements not realizable by complex lines.
        Being an unexpected arrangement 
	is a stronger condition: we want to rule out not only complex line arrangements with the same incidence relations as those of $\Lambda$, 
	but also complex line arrangements that satisfy all the incidence relations of $\Lambda$ and possibly additional incidence relations (without being a pencil). For instance, 
	the pseudo-Pappus arrangement (Example~\ref{pseudopapp} in the next section)
	is not realizable by complex lines but it is not unexpected, 
	because the classical Pappus arrangement has all of the same incidences and a additional 
	one. 	
\end{remark}

\begin{theorem} \label{thm:nonMilnor}  Suppose that $\Gamma=\{\Gamma_1, \dots, \Gamma_m \}$ is an arrangement of smooth graphical disks
with marked points $\{p_j\}$, related by a smooth graphical homotopy to a decorated germ $(\mathcal{C}, w)$. Let $(Y, \xi)$ be the link of the surface singularity that corresponds to $(\mathcal{C}, w)$.
Suppose that a subcollection of disks $\{\Gamma_1, \Gamma_2, \dots, \Gamma_r\}$ of $\Gamma$ forms an unexpected symplectic line arrangement. 

Then the Stein filling $W$ given by $(\Gamma, \{p_j \})$ is not strongly diffeomorphic to any 
Milnor filling of $(Y, \xi)$. If the weights on $\mathcal{C}$ are large enough,  $W$ is simply-connected. 
\end{theorem}

By Proposition~\ref{p:symp-lines-fill}, unexpected line arrangements yield unexpected fillings of Seifert fibered spaces of the form $Y(m,w)$.

\begin{cor} \label{cor:nonMilnor} Let $\Gamma= \{\Gamma_1, \dots, \Gamma_m\}$ be an unexpected symplectic line arrangement, and for $k=1, \dots, m$ 
let $w(\Gamma_k)$ denote the number of intersection points of $\Gamma_k$ with the disks $\Gamma_i$, $i\neq k$.
Then  for every weight $w=(w_1, w_2, \dots, w_m)$ with $w_k \geq w(\Gamma_k)$, $k=1, \dots, m$,  the Seifert fibered space $(Y(m,w), \xi)$ has a Stein filling not
strongly  diffeomorphic to any Milnor filling.  
This Stein filling is given by a Lefschetz fibration constructed from 
the arrangement $\Gamma$ with the appropriate choice of marked points. When strict inequalities
$w_k> w(\Gamma_k)$ hold for all $k$, we get a simply-connected unexpected Stein filling.
\end{cor}


\begin{proof}[Proof of Theorem~\ref{thm:nonMilnor}] Observe that when the number of intersection points on each $\Gamma_i$ 
is smaller than the weight of the corresponding  branch of the decorated germ,
each $\Gamma_i$ has a free marked point. Then the  Lefschetz fibration constructed from $(\Gamma, \{p_j \})$
	has a boundary-parallel vanishing cycle around every hole in the disk fiber, so that the corresponding thimbles kill all generators of $\pi_1(\text{fiber})$,
	and therefore, in this case $\pi_1(W)=0$.

	Let $W_M$  be a Milnor filling  that arises from a smoothing of some surface singularity with the link~$Y$. 
	By Theorem~\ref{thm:intro-djvs}, $W_M$ corresponds to a picture deformation $\mathcal{C}'^s$ of  a decorated germ
	$\mathcal{C}' = \cup_{i=1}^m C'_i$ with weight $w$, topologically equivalent to $(\mathcal{C}, w)$.

	Although the germs $\mathcal{C}$ and $\mathcal{C}'$ may differ analytically, they are topologically equivalent
	and thus have isotopic boundary braids. Therefore by Lemma~\ref{same-monodromy} the open book decomposition naturally induced by the Lefschetz fibration in Lemma~\ref{construct-fibration} for $W$ agrees with that for $W_M$, so comparing them via strong diffeomorphism makes sense.
	   
	By Proposition~\ref{incidence-matrix}, if $W$ is strongly diffeomorphic to $W_M$, 
	the incidence matrix of the deformed curvetta arrangement $\{C'^{s}_1, \dots, C'^{s}_{m}\}$, $s\neq 0$, with its marked points
	must be the same as the incidence matrix for the  arrangement $(\Gamma, \{p_j \})$, up to permutation of columns. In particular, 
	we see that the subarrangement 
	$\{\Gamma_1, \dots, \Gamma_r\}$ of symplectic lines satisfies the same incidence relations as the subarrangement $\{C'^{s}_1, \dots, C'^{s}_{r}\}$ of the
	deformed curvettas of $\mathcal{C}'$. By assumption, in each of these arrangements not all curvettas are concurrent.
	Because pairs of curves $\Gamma_1,\dots, \Gamma_r$ intersect algebraically positively once, $C'_1, \dots, C'_{r}$ have distinct tangent lines.
	Now by Lemma~\ref{get-lines}, there exists a complex line 
	arrangement $\mathcal{A}$ that satisfies all the incidence relations of $\{C'^{s}_1, \dots, C'^{s}_{r}\}$, and thus all the incidence relations of $\Gamma$. 
	This is a contradiction because $\Gamma$ is an unexpected arrangement.	
	\end{proof}

\subsection{Constructing unexpected pseudoline arrangements}
We now give examples of unexpected pseudoline arrangements; these will yield concrete examples of unexpected Stein fillings.
We start with classical projective geometry constructions. 

\begin{example} \label{ex:pappus} Recall that the classical Pappus arrangement in $\R^2$ is constructed as follows. Take two lines, $\ell_1$ and $\ell_2$, and mark three distinct points $a, b, c$ on $L_1$ 
	and three distinct points 
	$A, B, C$ on $\ell_2$, avoiding the intersection $\ell_1 \cap \ell_2$. Consider the following lines through pairs of marked points: 
	$$
	\ell_3 = aB, \quad \ell_4= aC, \quad \ell_5=bA, \quad \ell_6=bC, \quad \ell_7=cA, \quad \ell_8=cB. 
	$$
	The Pappus theorem asserts that the three intersection points $\ell_3 \cap \ell_5$, $\ell_4 \cap \ell_7$, and $\ell_6\cap \ell_8$ are collinear; the classical Pappus arrangement consists of 
	the lines $\ell_1, \dots, \ell_8$, together with the line through these three points. We modify this last line to make an unexpected pseudoline arrangement, as follows. Let $\ell_9$ be a line through
	$C$, distinct from $\ell_4$ and $\ell_6$. Consider the intersection point $\ell_8 \cap \ell_9$ and let $\ell_{10}$ be a pseudoline passing through points  $\ell_3 \cap \ell_5$, $\ell_4 \cap \ell_7$, and 
	$\ell_8 \cap \ell_9$ as shown in Figure~\ref{fig:pappus}. Let $\mathcal{P}=\{\ell_1, \ell_2, \dots, \ell_{10} \}$. 
	
	Notice that in this case, it is clear that the pseudoline $\ell_{10}$ can be homotoped to the classical Pappus line 
	through the points $\ell_3 \cap \ell_5$, $\ell_4 \cap \ell_7$, and $\ell_6\cap \ell_8$. The resulting arrangement of straight lines in $\R^2$  
	can be homotoped to a pencil by linear homotopy. (We already know from discussion in Example~\ref{discuss-plines} that $\mathcal{P}$ is homotopic 
	to the pencil, but here we have a very simple explicit homotopy.) 
\end{example}

\begin{figure}[htb]
	\centering
	\bigskip
	\labellist
	\small\hair 2pt
	\pinlabel $\ell_1$ at 4 107
	\pinlabel $\ell_2$ at 4 35
	\pinlabel $\ell_4$ at 23 127
	\pinlabel $\ell_3$ at 37 127
	\pinlabel $\ell_6$ at 88 131
	\pinlabel $\ell_5$ at 107 131
	\pinlabel $\ell_8$ at 157 138
	\pinlabel $\ell_7$ at 171 139
	\pinlabel $\ell_9$ at 129 134
	\pinlabel $\ell_{10}$ at 45 51
	\pinlabel $a$ at 38 111
	\pinlabel $b$ at 91 118
	\pinlabel $c$ at 161 127
	\pinlabel $A$ at 56 26
	\pinlabel $B$ at 95 21
	\pinlabel $C$ at 162 11
	
	\endlabellist
	\includegraphics[scale=1.3]{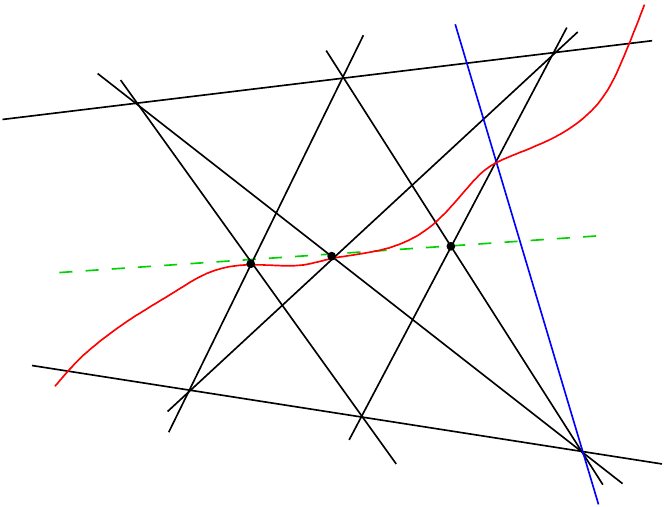}
	\caption{The pseudoline arrangement $\mathcal{P}=\{\ell_1, \ell_2, \dots, \ell_{10}\}$ is given by the black lines, the blue line, and the red line in the figure.
		The dotted line in the middle is not included. The dotted line and the eight black lines give the classical Pappus arrangement. Note that the intersection points 
		$\ell_1 \cap \ell_2$, $\ell_3\cap \ell_6$, $\ell_3\cap \ell_9$, $\ell_5\cap \ell_8$, and $\ell_5\cap \ell_9$ are not shown in the figure.}
	\label{fig:pappus}
\end{figure}
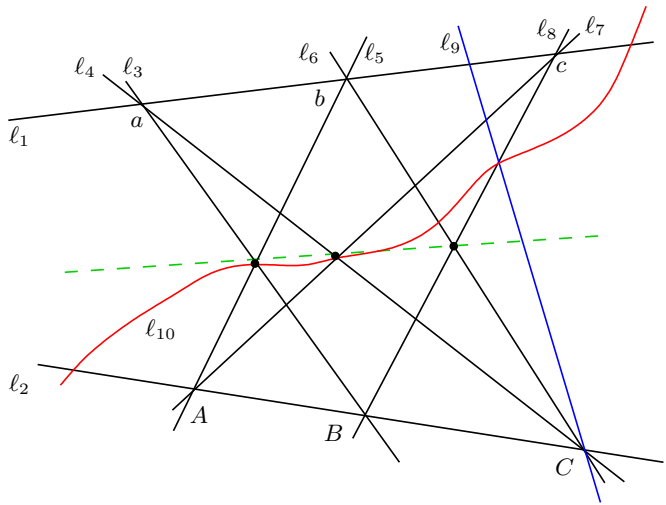

\begin{prop}	\label{pappus-exotic}
	The arrangement $\mathcal{P}$ is unexpected.
\end{prop}

\begin{proof} As already stated, the classical Pappus theorem asserts that for the given arrangement, the intersection points 
	$\ell_3 \cap \ell_5$, $\ell_4 \cap \ell_7$, and $\ell_6\cap \ell_8$ are collinear. Collinearity holds both in the real and in the complex projective geometry settings, so that if 
	$L_1, L_2, \dots, L_7, L_8 \subset \C^2$ are complex lines with given incidences, then  $L_3 \cap L_5$, $L_4 \cap L_7$, and $L_6\cap L_8$ are collinear. 
	From this, we can immediately 
	see that the arrangement $\mathcal{P}$ is not realizable by complex lines $\{L_1, L_2, \dots, L_{10}\}$: since  $L_6\cap L_8$ and $L_8 \cap L_9$ are distinct points on $L_8$,  the points 
	$L_3 \cap L_5$, $L_4 \cap L_7$ and $L_8 \cap L_9$ cannot be collinear. 
	
	To show that  $\mathcal{P}$ is unexpected, we need to prove that no complex line arrangement satisfies all the incidence relations of $\mathcal{P}$ even if some (but not all) 
	of the intersection points coincide. Indeed, we show that if a complex line arrangement $\mathcal{A}=\{L_1, L_2, \dots, L_{10}\}$ satisfies the incidence relations of $\mathcal{P}$ and two of the intersection points 
	coincide, then $\mathcal{A}$ must be a pencil. Remember that we always assume that all the lines in the arrangement are distinct.
	
	The following trivial fact, applied systematically, greatly simplifies the analysis of cases:
	\begin{observation} \label{3points} Let $L_1, L_2, L_3, L_4$ be four lines in $\C^2$, which are not necessarily distinct. Suppose that two of the pairwise intersection points coincide: 
		$L_1\cap L_2=  L_3\cap L_4$. Then $L_1$, $L_2$, $L_3$, and $L_4$ are concurrent, so that 
		they all intersect at the point $L_1\cap L_2=L_1\cap L_3=L_1\cap L_4=L_2\cap L_3 = L_2\cap L_4 = L_3\cap L_4$.
		
		In the case of three lines, if $L_1\cap L_2 = L_3\cap L_1$, then $L_2\cap L_3 = L_1\cap L_2 = L_3\cap L_1$. Visually, if two vertices
		of a triangle coincide, the third vertex of the triangle  coincides with the first two.
		
	\end{observation}
	
	Assuming that some of the intersection  points in Figure~\ref{fig:pappus} coincide, we mark these points by ``O'', and 
	then use Observation~\ref{3points} to chase vertices that coincide: starting with two marked vertices, we  look for additional vertices that coincide with the first two,
	further mark these by ``O'', and continue. When every line contains a marked intersection point, we know that all lines in the arrangement are concurrent: 
	they form a pencil though O.
	
	We begin this process. First, assume that the intersection points $L_3 \cap L_5\cap L_{10}$ and $L_4 \cap L_7\cap L_{10}$ are distinct. By the Pappus theorem, 
	the complex line arrangement $\mathcal{A}=\{L_1, L_2, \dots, L_{10}\}$ can satisfy all the incidence relations of $\mathcal{P}$ only if $L_6\cap L_8=L_8 \cap L_9\cap L_{10}$. Setting
	$\text{O}=L_6\cap L_8=L_8 \cap L_9\cap L_{10}$, by Observation~\ref{3points} we have $\text{O}=C=L_4 \cap L_6\cap L_9 \cap L_2$, then  
	$\text{O}= B= L_8 \cap L_2 \cap L_3$, then $\text{O}= a= L_3 \cap L_4 \cap L_1$, then $\text{O}= b= L_5\cap L_6\cap L_1$ and $\text{O}=c = L_7 \cap L_8 \cap L_1$. Now, O appears on every line at least once, so the arrangement degenerates to a pencil. 
	
	(This can be seen quickly if in the above diagram, you highlight the lines passing through intersection points marked by O, in order. 
	You can mark a new intersection by O if it contains at least two highlighted lines, and then highlight all the lines through that point O. 
	When all the lines are highlighted, you have a pencil.)
	
	For the second case, assume that the intersection points $L_3 \cap L_5\cap L_{10}$ and $L_4 \cap L_7\cap L_{10}$ coincide. Set $\text{O}=L_3 \cap L_5\cap L_{10}= L_4 \cap L_7\cap L_{10}$. Then 
	$\text{O}=a=L_3\cap L_4 \cap L_1$ and
	$\text{O}=A=L_5\cap L_7 \cap L_2$. Then $\text{O}=c=L_7\cap L_8\cap L_1$ and $\text{O}=C=L_4\cap L_6\cap L_2\cap L_9$. Again, every line contains a point marked O, so the arrangement degenerates to a pencil.
\end{proof}

\begin{cor}    \label{Ypappus}
	Let $Y= Y(10; w)$ be a Seifert fibered space given by a star-shaped plumbing graph with 10 legs as in Figure~\ref{pencil},
	such that eight of the legs of the graph have at least 5 vertices each, including the central vertex, and  two remaining legs have at least 4 vertices each. 
	(Equivalently, two components of $w$ are 5 or greater, and the rest are 6 or greater.)  
	Observe that $Y$ is the link of a rational singularity, and let $\xi$ be the Milnor fillable contact structure on $Y$. 
	Then $(Y, \xi)$ admits a Stein filling which is not strongly diffeomorphic to any Milnor filling.
\end{cor}
\begin{figure}[htb]
	\centering
	\bigskip
	\labellist
	\small\hair 2pt
	\pinlabel $w_1\geq 6$ at 237 175
	\pinlabel $w_2\geq 5$ at 237 160
	\pinlabel $w_3\geq 6$ at 237 138
	\pinlabel $w_4\geq 5$ at 237 122
	\pinlabel $w_5\geq 6$ at 237 107
	\pinlabel $w_6\geq 6$ at 237 91
	\pinlabel $w_7\geq 6$ at 237 74
	\pinlabel $w_8\geq 6$ at 237 58
	\pinlabel $w_9\geq 6$ at 237 40
	\pinlabel $w_{10}\geq 6$ at 237 7
	\pinlabel $-11$ at 325 98
	\endlabellist
	\includegraphics[scale=1]{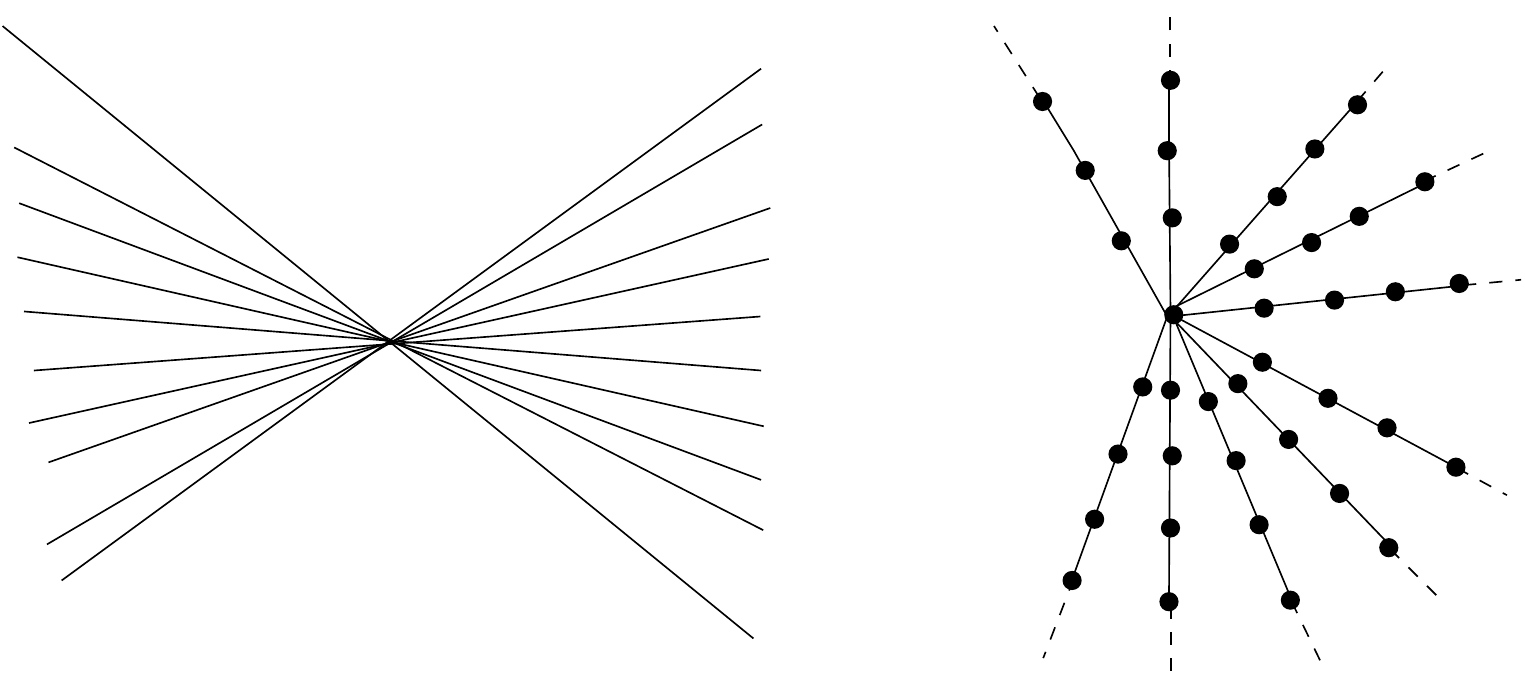}
	\caption{Left: a pencil of 10 lines decorated with weights. Right: the plumbing graph for $Y$, the central vertex has self-intersection $-11$, all the rest 
		self-intersection~$-2$. Eight of the legs have at least 5 vertices 
		each (including the central vertex), and  two remaining legs have at least 4 vertices each.}
	\label{pencil}
\end{figure}
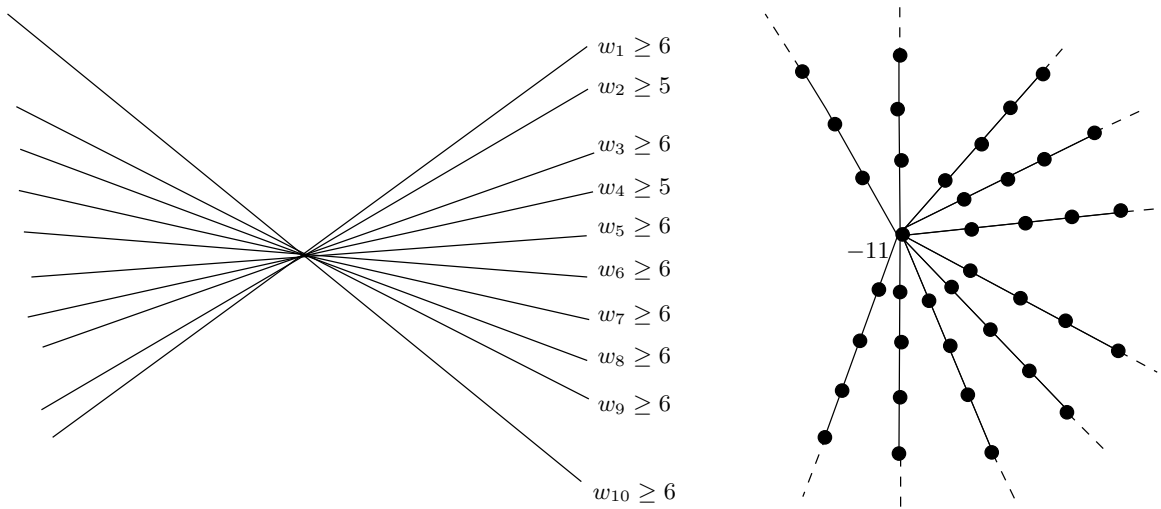
\begin{proof} We count the intersection points on each line in the arrangement $\mathcal{P}$: $w(\ell_2)=w(\ell_4)=5$, $w(\ell_k)=6$ for $k\neq 2, 4$. Then for any collection of 
	integers $w_1, w_2, \dots, w_{10}$ such that $w_2 \geq 5$, $w_4\geq 5$, and $w_k \geq 6$ for $k\neq 2, 4$, we can mark the lines of the arrangement $\mathcal{P}$ as required 
	in Corollary~\ref{cor:nonMilnor}. The corresponding singularity has the dual resolution graph as shown in Figure~\ref{pencil}, with one leg of length $w_k-1$ for each line $L_k$ in the arrangement, 
	so the link is the Seifert fibered space $Y(10,w)$. The result now follows from Corollary~\ref{cor:nonMilnor} and Proposition~\ref{pappus-exotic}.
\end{proof}


A different example comes from a version of the Desargues theorem; we use complete quadrangles and harmonic conjugates. The example in Figure~\ref{fig:orev-lines} was pointed out to us by Stepan Orevkov. He suggested an approach to prove that this arrangement cannot appear as an algebraic deformation of a pencil. We are grateful for his input which inspired us to define unexpected line arrangements and prove Lemma~\ref{get-lines}.

\begin{figure}[htb]
	\centering
	\bigskip
	\labellist
	\small\hair 2pt
	\pinlabel $V$ at 57 157
	\pinlabel $H$ at 150 60
	\pinlabel $P$ at 103 127
	\pinlabel $\ell_1$ at 21 110
	\pinlabel $\ell_2$ at 38 110
	\pinlabel $\ell_3$ at 56 110
	\pinlabel $\ell_4$ at 73 110
	\pinlabel $\ell_0$ at 21 155
	\pinlabel $\ell_6$ at 5 60
	\pinlabel $\ell_5$ at 5 82
	\pinlabel $\ell_7$ at 5 39
	\pinlabel $\ell_8$ at 5 20
	\pinlabel $\ell_9$ at 10 4
	\pinlabel $\ell_{10}$ at 108 10
	\pinlabel $a$ at 40 78
	\pinlabel $b$ at 57 67
	\pinlabel $c$ at 73 32
	\pinlabel $a'$ at 22 79
	\endlabellist
	\includegraphics[scale=1.3]{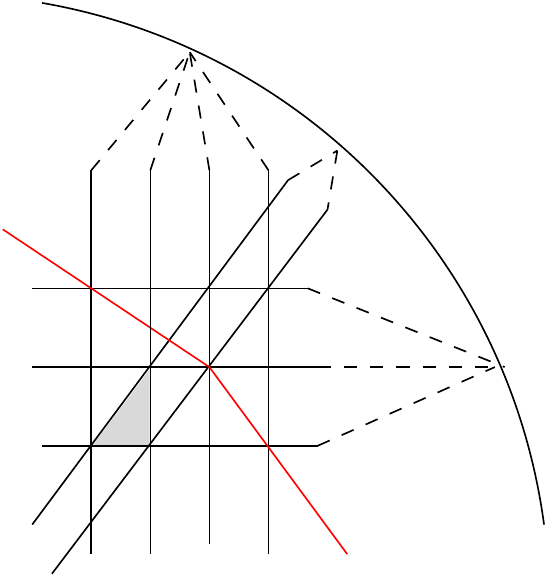}
	\caption{An arrangement of real pseudolines. The intersection of $\ell_0$ and $\ell_{10}$ is not shown.}
	\label{fig:orev-lines}
\end{figure}
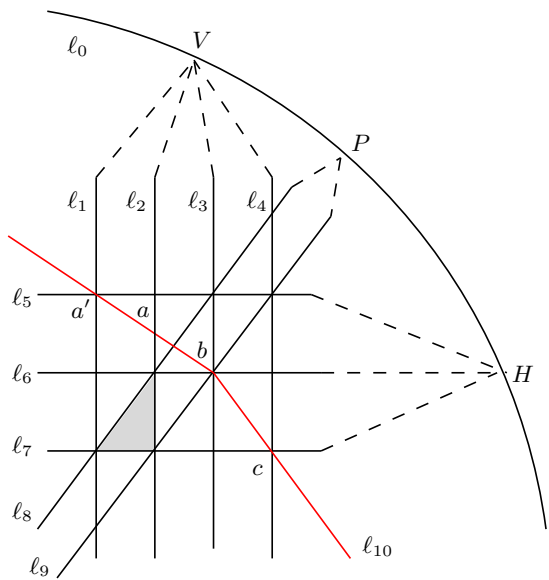

\begin{example} \label{example-orevkov}In the standard $\R^2 \subset \rptwo$, we take  four vertical lines  $\ell_1, \ell_2, \ell_3, \ell_4$,  three horizontal 
	lines $\ell_5, \ell_6, \ell_7$, the two parallel diagonal lines $\ell_8, \ell_9$, and a ``bent'' pseudoline 
	$\ell_{10}$ as shown in Figure~\ref{fig:orev-lines}.  Let $\ell_0$ be the line at infinity. Note that 
	because $\ell_1, \ell_2, \ell_3, \ell_4$ are all parallel in $\R^2$, they intersect at a point $V$ on $\ell_0$. Similarly, the lines $\ell_5, \ell_6, \ell_7$ have a common intersection with $\ell_0$ at a point $H$, and the lines $\ell_8$ and $\ell_9$ intersect on $\ell_0$ at a point $P$. Removing from $\rptwo$ a line which is different from all $\ell_i$'s and intersects them generically,
	we can consider $\mathcal{Q}=\{\ell_i\}_{i=0}^{10}$ as a pseudoline arrangement in $\R^2$. (See Figure~\ref{fig:orev-proj} for a version where $\ell_0$ is no longer the line at infinity.)
\end{example}

\begin{prop} \label{orevkov-exotic} The pseudoline arrangement $\mathcal{Q}$ is unexpected.
\end{prop}
\begin{proof} Suppose that a complex line arrangement $\mathcal{A}=L_0, L_1, \dots, L_{10}$ satisfies all the incidence relations of $\mathcal{Q}$. This means that 
	for all intersections between the pseudolines in Figure~\ref{fig:orev-lines}, the corresponding lines of $\mathcal{A}$ intersect. We claim that unless 
	$\mathcal{A}$ is a pencil, all of these intersection points must be distinct (That is, no two distinct intersection points in Figure~\ref{fig:orev-lines} can coincide for the 
	arrangement $\mathcal{A}$.) To see this, we use Observation~\ref{3points} repeatedly, as in Proposition~\ref{pappus-exotic}.  Recall that $V=L_1\cap L_2\cap L_3\cap L_4 \cap L_0$ and $H=L_5\cap L_6\cap L_7\cap L_0$.

	If $H=V=\text{O}$, then we have $L_i \cap L_j=\text{O}$ for all $1 \leq i \leq 4$, {$5 \leq j \leq 7$}, and so $\mathcal{A}$ is a pencil. 
	
	If one of the intersection 
	points $L_i \cap L_j$, $1 \leq i \leq 4$, {$5 \leq j \leq 7$}, coincides with $V$ or $H$, then we have two intersection points marked with $\text{O}$ on a vertical or horizontal 
	line  in Figure~\ref{fig:orev-lines}; then $\text{O}=V=H$, and all lines are concurrent.
	
	If any two intersection points $L_i \cap L_j$,
	$1 \leq i \leq 4$, {$5 \leq j \leq 7$}, coincide, Observation~\ref{3points} implies that they will both coincide with at least one of $V$ or $H$, so we revert to the previous case. 
	
	Finally, if all the points $V$, $H$,  $L_i \cap L_j$,
	$1 \leq i \leq 4$, {$5 \leq j \leq 7$} are distinct, all remaining intersection points which do not coincide with one of these are necessarily generic double points
	(otherwise we would have a pair of lines intersecting more than once).
	
	Once we know that all the distinct intersections for $\mathcal{Q}$  are distinct for $\mathcal{A}$, it remains to show that $\mathcal{Q}$ cannot be realized as a complex line 
	arrangement $\mathcal{A}=\{L_i\}_{i=0}^{10}$. Suppose that it is, for the sake of contradiction.

	\begin{figure}[htb]
		\centering
		\includegraphics[scale=.75]{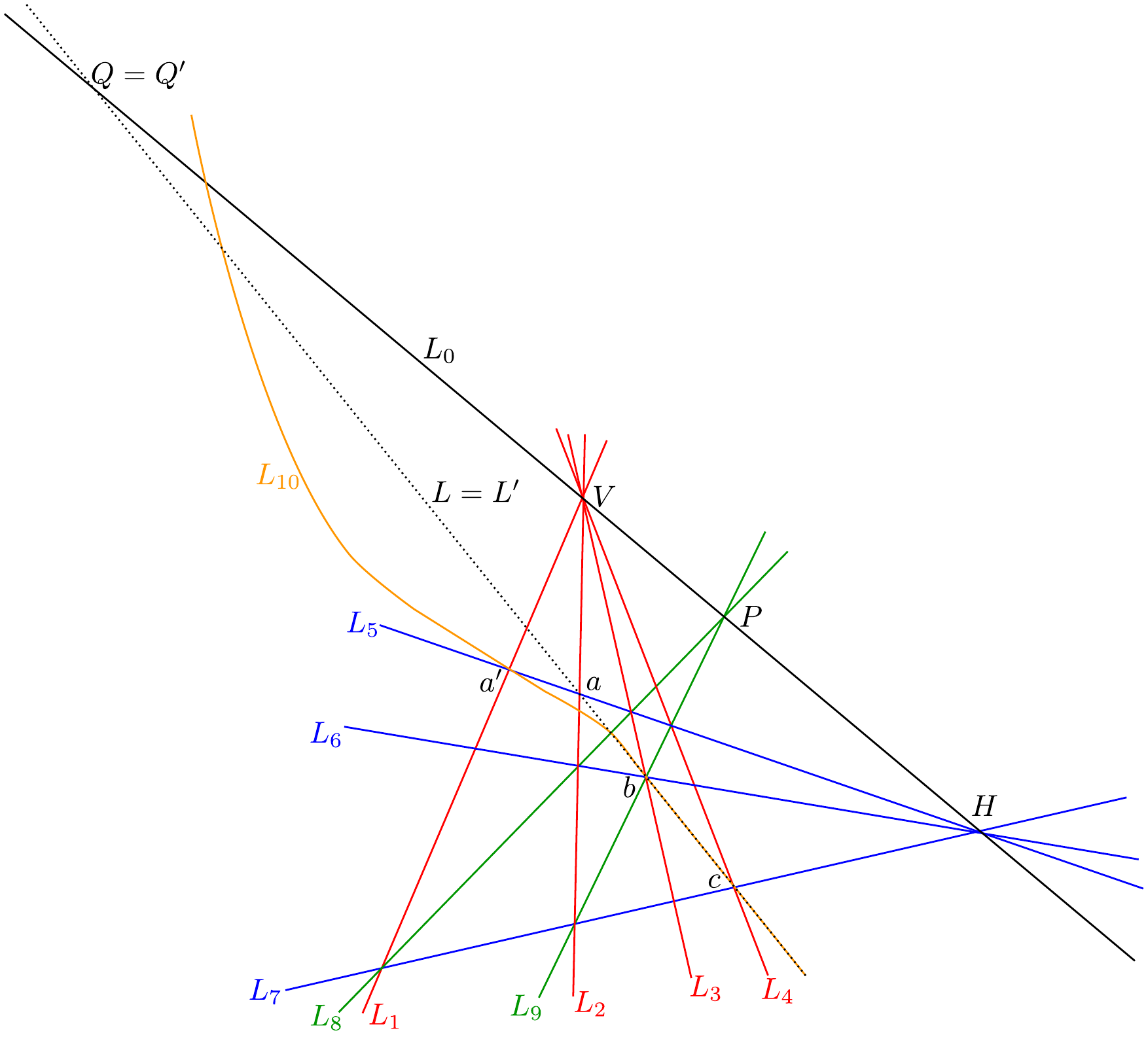}
		\caption{An arrangement of lines $L_0,L_1,\dots, L_9$ and pseudoline $L_{10}$ with incidences as in Figure~\ref{fig:orev-lines}. We show that the line $L$ through $a$ and $b$ and the line $L'$ through $a$ and $c$ coincide (with the dotted line shown), so the points $a$, $b$, and $c$ are collinear. Therefore $a'$, $b$, and $c$ cannot be collinear.}
		\label{fig:orev-proj}
	\end{figure}
	
	We will show that the intersection points $a=L_2 \cap L_5$, $b=L_3\cap L_6$, and $c=L_4 \cap L_7$ are collinear. (See Figure~\ref{fig:orev-proj}) Then we can conclude that the points
	$a'=L_1 \cap L_5$, $b$, and $c$ cannot be collinear. Indeed, $a \neq a'$, since 
	all intersection points in the diagram are distinct. If  all four points $a$, $a'$, $b$, $c$ were collinear, then the line $L_5$ through $a$ and $a'$ would 
	coincide with the line $L_{10}$ through $a'$, $b$, and $c$, but we assume that $L_5$ and $L_{10}$ are distinct. 
	
	
	
	To see that the points $a$, $b$, and $c$ are collinear, we will use some notions of classical projective geometry, namely complete quadrangles and harmonic conjugates. (In Remark~\ref{eucl} below, we also indicate an alternative proof, in the more familiar Euclidean terms.)   
	Observe that the lines $L_5$, $L_6$, $L_2$, $L_3$, $L_8$, and the line $L$ through $a,b$ form the four sides and the two
	diagonals of a complete quadrangle. Then the point $Q=L\cap L_0$ is the harmonic conjugate of the point $P=L_8 \cap L_0$
	with respect to the points $V=L_2 \cap L_3$ and $H=L_5 \cap L_6$.  Now, consider the lines 
	$L_2$, $L_4$, $L_5$, $L_7$, $L_9$, and the line $L'$ through $a$ and $c$. Again these form a complete quadrangle, so that 
	the point  $Q'=L' \cap L_0$  is the harmonic conjugate of the point $P=L_9 \cap L_0$  with respect to  $V=L_2 \cap L_4$ and $H=L_5 \cap L_7$.
	Since the harmonic conjugate of $P$ with respect to $V, H$ is unique, it follows that $Q=Q'$. Since the lines $L$ and $L'$
	both pass through $Q=Q'$ and $a$, we must have $L=L'$, and so all three points $a$, $b$, $c$ lie on this line.  
\end{proof}
\begin{remark}\label{eucl} The above statement also has an easy Euclidean geometry proof, after some projective transformations. Indeed, we can find an automorphism of $\cptwo$ such that 
	$$
	L_1 \cap L_5 \mapsto (0:0:1), \quad L_1 \cap L_6  \mapsto (1:0:1), \quad L_2 \cap L_5 \mapsto (0:1:1), \quad  L_2 \cap L_6 \mapsto (1:1:1).  
	$$
	Then  $H \mapsto (1:0:0)$,   $V\mapsto (0:1:0)$,
	and it is not hard to see that all the lines in the figure must be complexifications of real lines. The line $L_{0}$ is the line at infinity; the remaining lines 
	are (complexifications of) the corresponding real lines in $\R^2$. We use the same notation for the real lines.  Now we see that   
	$L_1, L_2, L_3, L_4$ are parallel vertical lines, $L_5, L_6, L_7$ are parallel horizontal lines, etc. So that the arrangement 
	looks like Figure~\ref{fig:orev-lines}. The lines in the  figure form  a number of triangles
	that are similar to the shaded triangle; it then follows that the points $a, b, c$ are collinear, so $a', b, c$ are not. 
	
	Note, however, that the above proof is somewhat incomplete: Figure~\ref{fig:orev-lines} assumes a particular position of the lines $L_3, L_4, L_7$ relative to $L_1, L_2, L_5, L_6$. For a complete proof, an additional analysis of cases is required, with slightly different figures for other possible relative positions of the lines. Our projective argument with harmonic conjugates allows to avoid this analysis, and also to emphasize the projective nature of the statement and the proof. 	
\end{remark}

\begin{cor}   \label{Yorevkov}
	Let $Y=Y(11;w)$ be the Seifert fibered space given by a star-shaped plumbing graph with 11 legs, 
	such that two legs have at least 5 vertices each, two legs have at least 3 vertices, and the remaining 
	7 legs have at least 4 vertices each (including the central vertex). In other words, two components of the multi-weight $w$
	are 4 or greater, two are 6 or greater, and the remaining seven are 5 or greater. 
	Let $\xi$ be the Milnor fillable contact structure on $Y$.	
	Then $(Y, \xi)$ admits a Stein filling which is not strongly diffeomorphic to any Milnor filling.
\end{cor}

\begin{proof} 
Exactly as in Corollary~\ref{Ypappus}, this follows from Corollary~\ref{cor:nonMilnor} and Proposition~\ref{orevkov-exotic}. The picture is similar to Figure~\ref{pencil},
with the obvious minor changes. Indeed, the pseudoline arrangement
of Proposition~\ref{orevkov-exotic} has two lines $\ell_0, \ell_3$ with weight 4, two lines $\ell_9, \ell_{10}$ with weight 6, and seven remaining lines with weight 5.
Note that a permutation of the components of $w$  does not change the contact manifold, so we avoided labeling the components of $w$ in the statement of the corollary. 
\end{proof}


It is easy to generalize the above examples to star-shaped graphs with higher negative self-intersection values of the central vertex. Indeed, 
by Theorem~\ref{thm:nonMilnor}, we can construct unexpected Stein fillings from an arbitrary  arrangement of smooth graphical disks that contains an unexpected 
symplectic line arrangement. We turn to the general case later in this section; for now, we create more unexpected pseudoline arrangements simply by adding extra lines.

\begin{lemma} \label{add-line} Suppose that $\Lambda$ is an unexpected symplectic line arrangement. Let $\ell$ be  a symplectic line that 
passes through at least one  intersection point of two or more lines in $\Lambda$. 
Then the pseudoline arrangement $\Lambda \cup \{\ell \}$ is also unexpected.
\end{lemma}
\begin{proof} If there exists a complex line arrangement $\mathcal{A} \cup \{L\}$ that satisfies all the incidence relations of $\Lambda \cup \{\ell \}$, and $L$ corresponds to 
	$\ell$, then $\mathcal{A}$ satisfies all incidences of $\Lambda$, and so $\mathcal{A}$ is a pencil. The line $L$ must pass through the intersection of two or more 
	lines of $\mathcal{A}$, so $\mathcal{A}\cup \{L\}$ is also a pencil. 
\end{proof}

\begin{theorem}  For any $m \geq 10$, consider the Seifert fibered space $Y_m=Y(m,w)$ with $m\geq 10$, with weights $w=(w_1, \dots, w_m)$ such that $w_i \geq m-1$ for all $i=1, \dots, m$.
The space $Y_m$ is given by a star-shaped graph with $m\geq 10$ legs, such that the length of each leg is at least $m-1$. The central vertex has self-intersection $-m-1$, 
and all other vertices have self-intersection $-2$.
Let $\xi$ be the Milnor fillable contact structure on $Y$. Then $(Y, \xi)$ admits a simply-connected Stein filling not strongly diffeomorphic to any Milnor fiber.
\end{theorem}
\begin{proof} We can add lines to  the arrangement $\mathcal{P}$ to form unexpected arrangements of $m \geq 10$ pseudolines. Since any pseudolines intersect at most once, 
each pseudoline has at most $m-1$ intersections with other lines. By Corollary~\ref{cor:nonMilnor}, $Y=Y(n,w)$ is an unexpected Stein filling if $w_i \geq m-1$ for 
all $i=1, \dots, m$, which is simply-connected if all inequalities are strict.
\end{proof}

Varying the positions of the additional lines and/or applying a similar procedure to different arrangements such as $\mathcal{P}$ and $\mathcal{Q}$, it is possible 
to construct a variety of pairwise non-homeomorphic Stein fillings of the same link, so that none of the Stein fillings is strongly diffeomorphic to a Milnor filling.  
We give one such construction below to prove the first part of Theorem~\ref{thm:intro-examples}. The second part of Theorem~\ref{thm:intro-examples} follows from the discussion
at the end of this section, where we extend star-shaped graphs that correspond to unexpected arrangements to a much wider collection of graphs of rational singularities with reduced 
fundamental cycle.

\begin{theorem} For every $N>0$ there exists a rational singularity with reduced fundamental cycle whose link $(Y, \xi)$ admits 
at least $N$ pairwise non-homeomorphic simply-connected Stein fillings, none of which is strongly diffeomorphic to any Milnor fiber.
The link $Y$ is given by a Seifert fibered space $Y=Y(2N+5, w)$ with sufficiently large 
weights $w$. 
\end{theorem}
\begin{figure}[htb]
	\centering
	\bigskip
	\labellist
	\small\hair 2pt
	\pinlabel $V$ at 82 200
	\pinlabel $H$ at 177 97
	\pinlabel $P$ at 135 162
	\pinlabel $V$ at 308 200
	\pinlabel $H$ at 403 97
	\pinlabel $P$ at 360 162
	\pinlabel $\lambda_0$ at 8 4
	\pinlabel $\lambda_3$ at 232 4
	\endlabellist
	\includegraphics[scale=1.0]{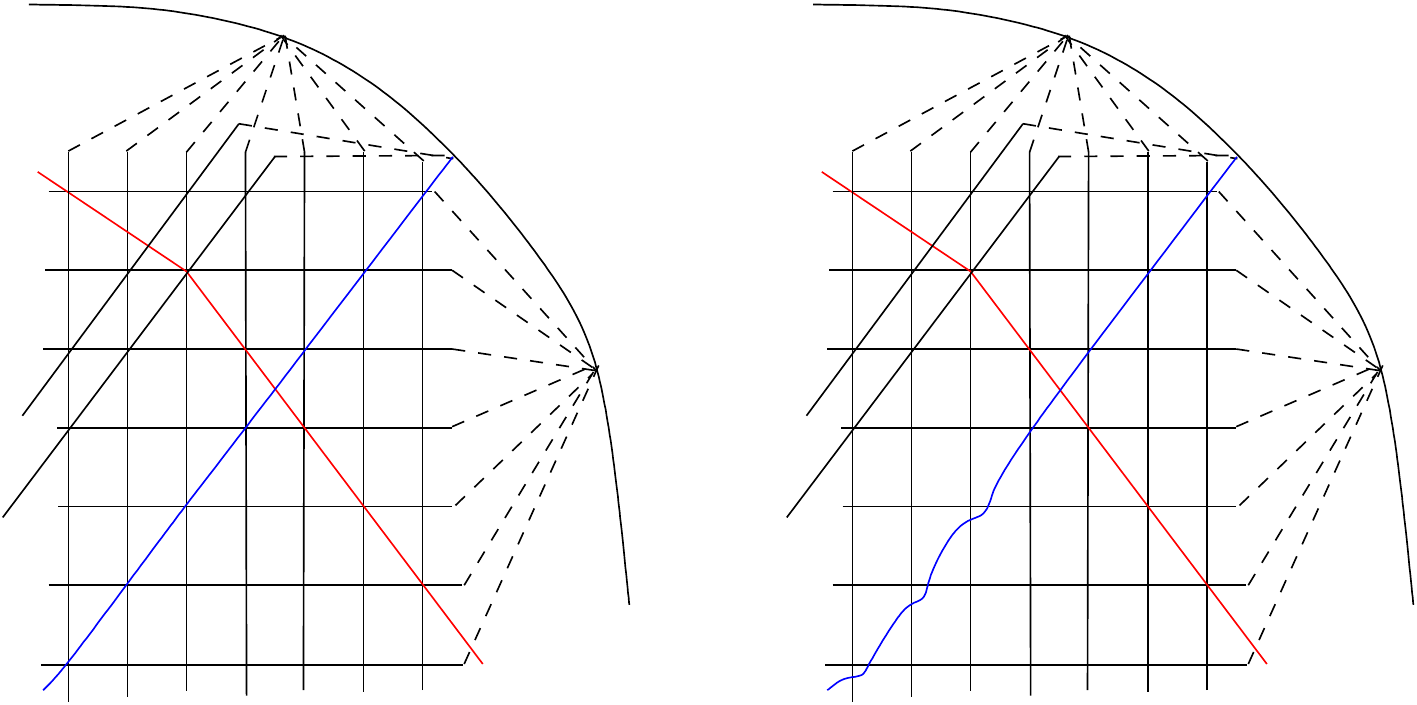}
	\caption{Pseudoline arrangements and fillings with different topology.}
	\label{fig:manylines}
\end{figure}
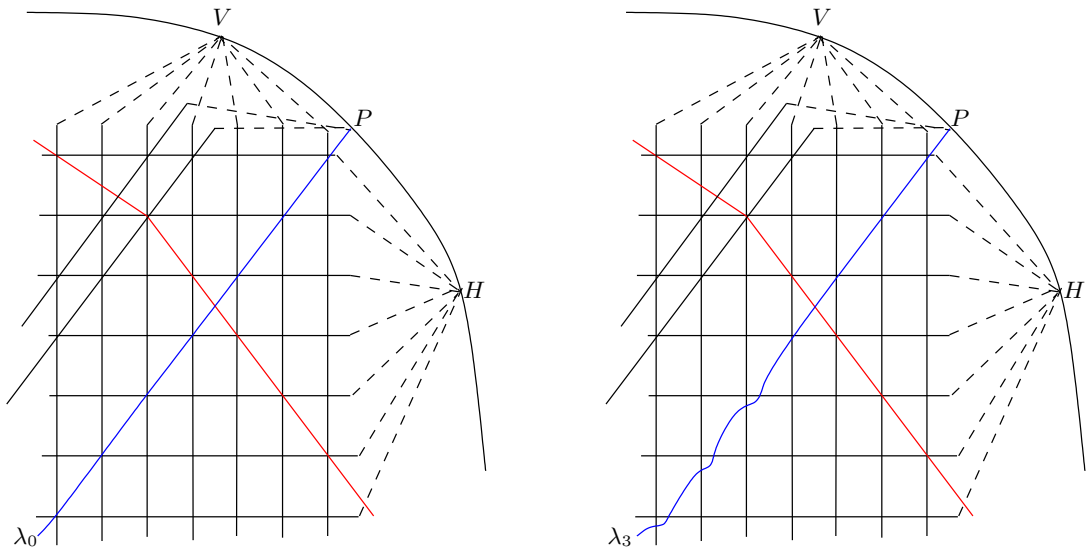
\begin{proof} We will start with the arrangement $\mathcal{Q}$ of Figure~\ref{fig:orev-lines} and augment it to other unexpected arrangements, using Lemma~\ref{add-line}. First, 
	we add more ``vertical'' and ``horizontal'' lines to the arrangement, so that it has $N$ vertical and $N$ horizontal lines, creating a grid as shown in Figure~\ref{fig:manylines}. (We assume $N\geq 4$ as the $N=4$ case fulfills the statement for lower values of $N$.)
	All ``vertical'' lines intersect at the point $V$, all horizontal lines intersect at the point $H$. The two diagonal lines, $\ell_8$ and $\ell_9$ intersecting at $P$,  
	the bent pseudoline $\ell_{10}$, and the line at infinity $\ell_0$ are present as in the arrangement $\mathcal{Q}$. Let $\mathcal{Q'}$ denote this arrangement. We will 
	now produce $N+1$ unexpected arrangements $\mathcal{Q}'_k=\mathcal{Q'}\cup{\lambda_k}$, $k=0, 1, \dots, N$, 
	by adding  to $\mathcal{Q'}$ different additional 
	``diagonal'' pseudolines $\lambda_0, \lambda_1, \dots, \lambda_N$ passing through $P$, see Figure~\ref{fig:manylines}. Each arrangement $\mathcal{Q}'_k$ consists of 
	$2N+5$ pseudolines. The pseudoline $\lambda_0$ is taken to be the main diagonal of the grid 
	formed by the vertical and horizontal lines; it  is a straight line in $\rptwo$ passing through  the point $P$. The pseudoline $\lambda_1$ differs from $\lambda_0$ in a small
	neighborhood of a single grid intersection: while $\lambda_0$ passes through the chosen intersection point of a vertical and a horizontal line, $\lambda_1$ intersects these 
	two lines at distinct points. Similarly, $\lambda_k$ differs from $\lambda_0$ in neighborhoods of $k$ grid intersections and meets the corresponding vertical and horizontal 
	lines at distinct points. Figure~\ref{fig:manylines} shows the arrangements  $\mathcal{Q}'_0=\mathcal{Q'}\cup{\lambda_0}$ and  $\mathcal{Q}'_3=\mathcal{Q'}\cup{\lambda_3}$. 
	
	Now, consider the decorated germ given by a pencil of $2N+5$ lines, each with a weight greater than $2N+4$. We choose the weights
	to be greater than the number of intersection points on each line in 
   any of the arrangements $\mathcal{Q}'_k$; obviously, taking weights greater than $2N+4$ suffices because each line intersects the other $2N+4$ lines once (in fact $w\geq 2N+2$ suffices for this arrangement). Let $(Y_N, \xi)$ be the contact link of the corresponding singularity. Similarly to the previous 
	examples, $Y_N$ is the Seifert fibered space given by a star-shaped plumbing graph with $2N+5$ sufficiently long legs,
	with the central vertex having the self-intersection
	$-2N-6$ and all the other vertices self-intersection $-2$. By Corollary~\ref{cor:nonMilnor}, each arrangement $\mathcal{Q}'_k$ 
	yields a Stein filling $W_k$ of $(Y_N, \xi)$ which is not strongly 
	diffeomorphic to any Milnor filling. 
	
	Finally, we argue that all fillings $W_0, W_1, \dots, W_N$ have different Euler characteristic. Each $W_k$ has the structure of 
	a Lefschetz fibration  with the same planar fiber (a disk with $2N+5$ holes), but these Lefschetz fibrations have different number of vanishing cycles. Every time we replace 
	a triple intersection of pseudolines in the arrangement by three double points (and arrange the marked points on the lines accordingly),
	the number of vanishing cycles 
	decreases by 1. Indeed, three double points correspond 
	to three vanishing cycles in the Lefschetz fibration (each enclosing two holes), while a triple intersection together with an additional
	free marked point on each of three lines corresponds to four vanishing cycles (one vanishing cycle enclosing three holes, 
	the remaining three enclosing a single hole each). Thus, replacing a triple point by three double points
	corresponds to a lantern relation monodromy substitution, which in turn corresponds to a rational blow-down of a $(-4)$ sphere.
	Therefore, $\chi(W_0)>\chi(W_1)>\dots >\chi(W_N)$, as required. 
\end{proof}

\subsection{Generalizations}
All our previous examples were given by singularities with star-shaped graphs where most vertices have self-intersection $-2$. It is not hard 
	to obtain examples with much more general graphs, using the full power of Theorem~\ref{thm:nonMilnor}:
	we add more smooth disks to an unexpected symplectic line arrangement.		

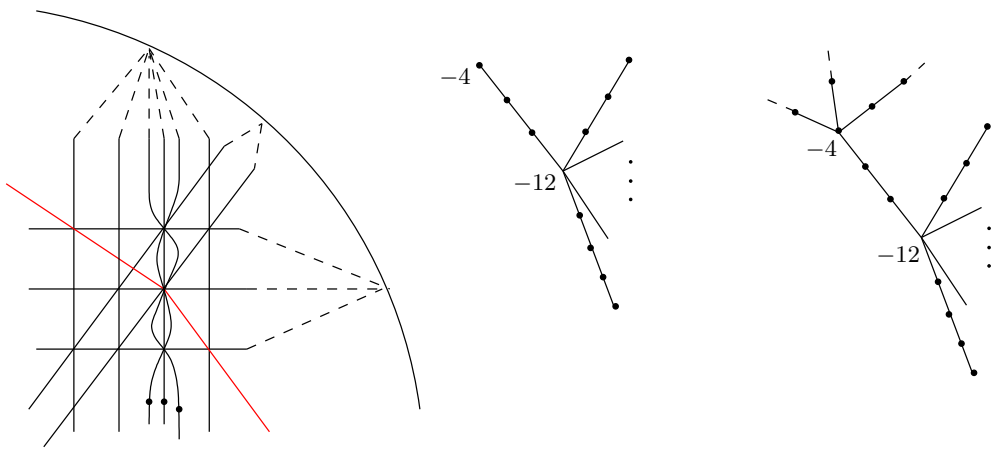
\begin{figure}[htb]
	\centering
	\bigskip
	\labellist
	\small\hair 2pt	
	\pinlabel $-12$ at 200 100
	\pinlabel $-4$ at 170 140
	\pinlabel $-12$ at 337 73
	\pinlabel $-4$ at 308 113
	\endlabellist
	\includegraphics[scale=1.0]{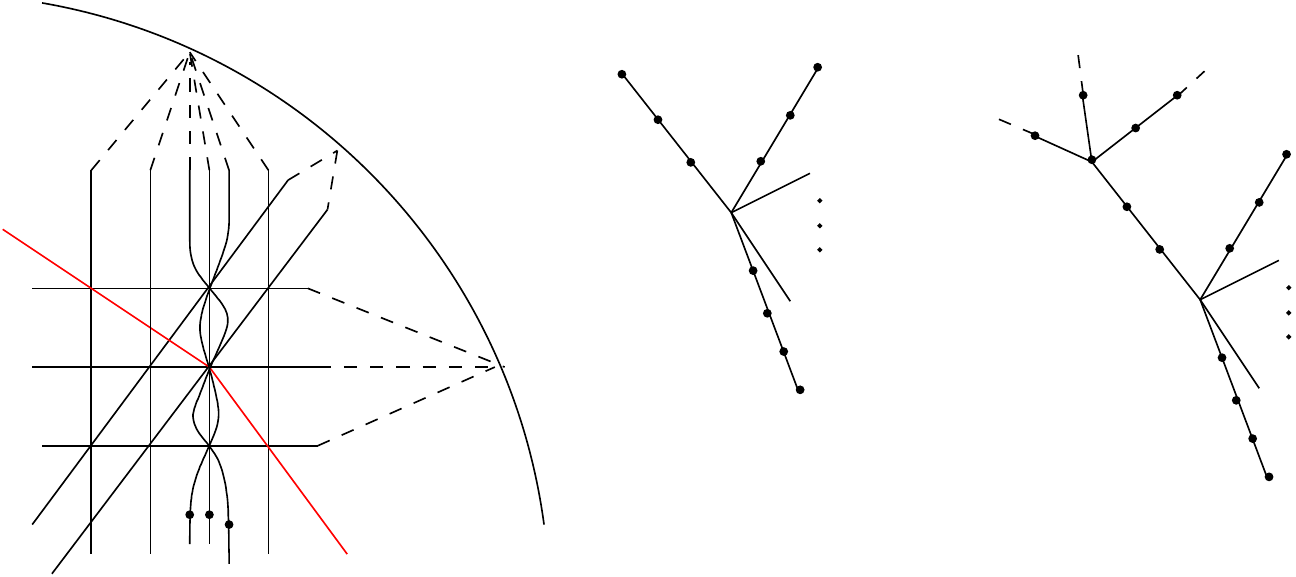}
	\caption{The pseudoline arrangement $\mathcal{Q}$ of  Figure \ref{fig:orev-lines} is modified: the pseudoline $\ell_3$ is replaced by three smooth curves
	with 4 intersections, as shown. There are 3 free marked points, one on each of the new curves; the rest of the marked points are the intersections in the diagram. 
	The germ of the corresponding singularity has three curvettas tangent to order $4$, each of weight 5, replacing one of the lines.  The resolution graph of the corresponding singularity is shown in the middle of the 
	figure. If the weights of the three tangent curvettas are taken to be higher, the  graph will have additional branching as shown on the right. All unlabeled vertices have self-intersection $-2$.}
	\label{fig:orev-lines-more}
\end{figure}

\begin{example} \label{ex:orev-more} In the arrangement $\mathcal{Q}$ of Figure~\ref{fig:orev-lines}, replace the line $\ell_3$ by several pseudolines that 
all pass through the same four intersection  points. Note that because of multiple intersections,  the result is no longer a pseudoline arrangement, but we still have a braided wiring diagram and can apply Proposition~\ref{symp-config} to extend it to an arrangement of  symplectic disks. In Figure~\ref{fig:orev-lines-more}, we take three curves replacing $\ell_3$. In the decorated germ, the complex line corresponding to $\ell_3$ will be replaced by  3 curvettas that  are tangent to order 4 (and  transverse to the other 10 branches of the germ). By~(\ref{weight-ineq}), the weight of each new curvetta must be 5 or greater. We take the weights to be exactly 5 for the three new curvettas. 
Consider the  symplectic curve arrangement given by the extension of the diagram in Figure~\ref{fig:orev-lines-more},
with  marked points at all intersections and one additional free marked point on each of the three new curves (to 
 account for higher weights).  The resolution graph for $\mathcal{Q}$ is star-shaped with 11 legs. The self-intersection of the central vertex is $-12$ and all other self-intersections are $-2$. The legs of the resolution graph for $\mathcal{Q}$ with minimal weights had two legs of length $3$, two of length $5$, and seven of length $4$. 
 For this revised arrangement, the corresponding singularity has an augmented graph. Specifically, one of the legs of length $3$ 
 (which corresponded to $\ell_3$) gains an additional vertex of self-intersection $-4$.
 If the three tangent  
 curvettas have higher weights, so they have additional free marked points in the deformed arrangement, the $-4$ vertex becomes a branching point with 3 additional legs 
 (each vertex on these legs has self-intersection $-2$).  See Figure~\ref{fig:orev-lines-more}. By Theorem~\ref{thm:nonMilnor}, the links of the corresponding singularities have 
 unexpected Stein fillings. 
 
 In general, if we replace $\ell_3$ with $k$ curves commonly intersecting at the four points where $\ell_3$ intersected other pseudolines as above, the additional vertex will have self-intersection 
 $-k-1$ and increased weights will yield $k$ additional legs with $(-2)$ vertices.
 \end{example}

 Further, we can replace each of the $k$ pseudolines  by a bundle of curves that go through the same intersections.
 
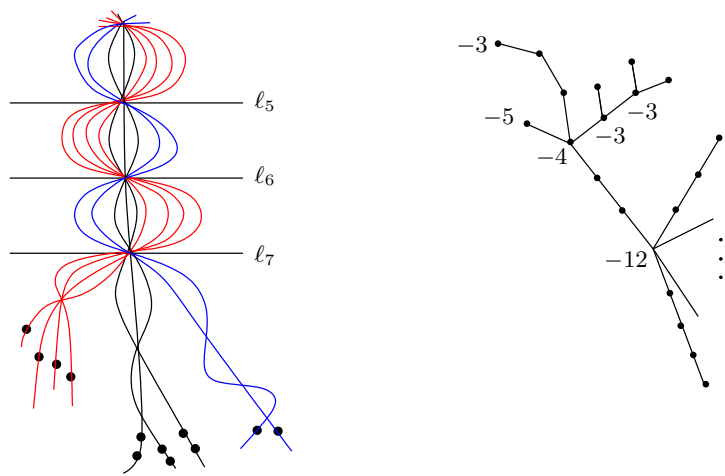
\begin{figure}[htb]
	\centering
	\bigskip
	\labellist
	\small\hair 2pt	
	\pinlabel $-12$ at 233 81
	\pinlabel $-4$ at 205 120
	\pinlabel $-5$ at 185 135
	\pinlabel $-3$ at 175 162
	\pinlabel $-3$ at 227 128
	\pinlabel $-3$ at 239 137
	\pinlabel $\ell_5$ at 97 141
	\pinlabel $\ell_6$ at 97 113
	\pinlabel $\ell_7$ at 97 83
	\endlabellist
	\includegraphics[scale=1.0]{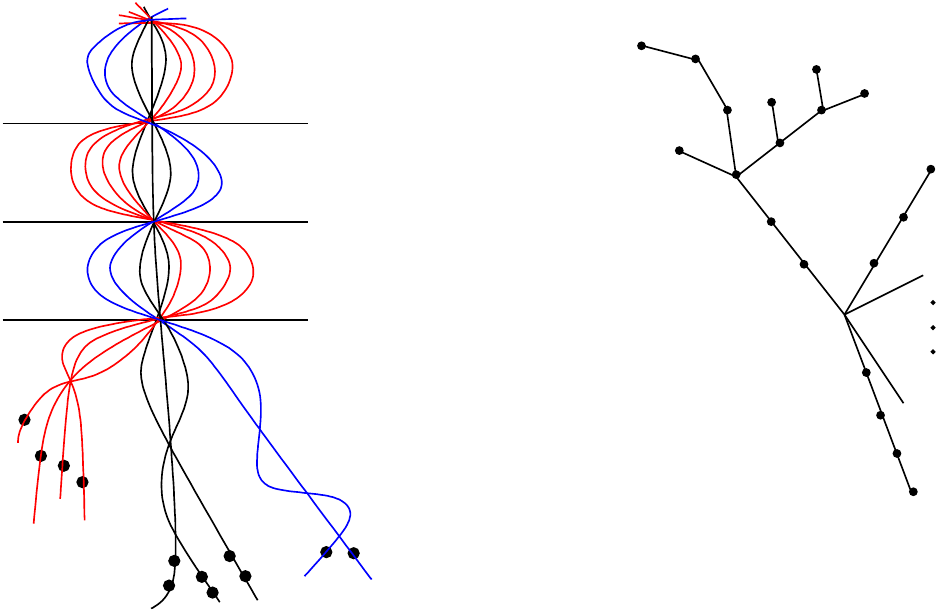}
	\caption{In the pseudoline arrangement $\mathcal{Q}$ of  Figure \ref{fig:orev-lines}, we replace $\ell_3$ with a bundle of curves passing through the exisiting intersections 
	of $\ell_3$ with $\ell_5$, $\ell_6$, $\ell_7$, and $\ell_0$. (Only part of the arrangement is shown.) Note that the additional curves create no extra intersections with 
	the pseudolines of $\mathcal{Q}$. All the intersection points are marked, and there are additional free marked points that correspond to higher weights.
	In the resolution graph of the  singularity, the leg corresponding to $\ell_3$ is replaced by a tree with additional branching, as shown. 
	 All 
	unlabeled vertices have self-intersection $-2$.}
	\label{fig:orev-manymore}
\end{figure}

\begin{example} \label{ex:orev-manymore}
Figure~\ref{fig:orev-manymore} shows  a possible {bundle replacing $\ell_3$, instead of the bundle of three curves in the previous  arrangement} of Figure~\ref{fig:orev-lines-more}. All the new curves run $C^1$ close to 
and are isotopic to the original pseudoline, and they pass through the same intersection points with the other pseudolines. Within each bundle, the curves may have additional 
intersections, which lead to higher order tangencies between the corresponding curvettas in the decorated germ. In particular, for the arrangement in
Figure~\ref{fig:orev-manymore}, the bundle of curves replacing $\ell_3$ will have three subbundles of curves intersecting each other 4 times, (and intersecting each of the other pseudolines once). 
One of these subbundles has four curves which are intersect each other a total of $5$ times, another has two curves which intersect a total of $7$ times, and the third has two curves intersecting each other a total of $6$ times, with an additional curve intersecting these two $5$ times.

The corresponding decorated germ (with the weights given by the number of intersection points in the disk arrangement) encodes the singularity whose graph has more branching 
and some vertices with higher negative self-intersections, as shown in Figure~\ref{fig:orev-manymore}. Note that if we vary the incidence pattern of the additional curves
(subject to the weight restrictions), we can obtain a number of unexpected Stein fillings with different topology.
\end{example}

Example~\ref{ex:orev-manymore} demonstrates how, once we have an unexpected symplectic line arrangement $\Gamma=\{\Gamma_i\}$, the star-shaped graph 
$G$ of the corresponding singularity can be extended to arbitrarily complicated graphs of rational singularities with reduced fundamental cycle. 
The following proposition explains how to form these bundles in general from a given extension of the graph, completing the proof of Theorem~\ref{thm:intro-examples}. 
It is not hard to see that under the hypotheses of the proposition, the extended graph $H$ corresponds to a singularity with reduced fundamental cycle. 

\begin{prop}\label{p:orev-manymore} 
	Let $G$ be the star-shaped resolution graph corresponding to the surface singularity associated to an unexpected symplectic line arrangement
	with minimal possible weights. Let $I$ be the set of leaves of $G$, and let $\{G_i\}_{i \in I}$ be a collection of (possibly empty)
	negative definite rooted trees; 
	assume that $G$ and $G_i$ have no $(-1)$ vertices. 
	
	Consider a graph $H$ constructed by attaching to $G$ the rooted trees $G_i$, $i \in I$,
	so that the root of $G_i$ is connected to the leaf $u_i$ by a single edge.
	{Assume that the resulting graph $H$ satisfies condition~(\ref{valency-weight}).}
	Let $(Y,\xi)$ be the link of a rational surface singularity with reduced fundamental cycle
	whose dual resolution graph is $H$. 
	
	Then $(Y,\xi)$ admits a Stein filling which is not strongly diffeomorphic to any Milnor filling. 
\end{prop}

\begin{remark}  Proposition~\ref{p:orev-manymore} provides a fairly general class of rational surface singularities with reduced fundamental cycle 
which admit unexpected fillings. The construction can be further generalized to include variations in the bundling structure and to apply to more general 
graphs $G$ as the input. Despite all variations, getting rid of the $(-2)$ vertices in the resolution graph seems difficult. Indeed, we could add a curve intersecting 
$\ell_3$ only twice in Example~\ref{ex:orev-manymore}, which would lower the self-intersection to $(-3)$ for one of the vertices on the leg of 
the star-shaped graph $G$. However, 
such a curve would 
intersect the other pseudolines in the arrangement $\mathcal{Q}$ at new points.
This would increase the weights on the curvettas corresponding to these other pseudolines, producing 
free marked points and yielding additional $(-2)$ vertices elsewhere  in the graph. In fact, we already know from Theorem~\ref{kollar-fill} that our strategy
must have limitations, as there are no unexpected fillings when  
 each vertex of the resolution graph has self-intersection $-5$ or lower. 
\end{remark}

\begin{proof}[Proof of Proposition~\ref{p:orev-manymore}]
The initial unexpected symplectic line arrangement $\{L_i\}$ consists of symplectic lines associated to the legs of the star-shaped graph $G$. As above, let $u_i$ denote 
the valency 1 vertex of the leg that corresponds to $L_i$. Choose a braided wiring diagram for the symplectic line arrangement such that a symplectic 
line~$L_i$ corresponds to the wire $\gamma_i$. The braided wiring diagram should be chosen such that $\gamma_i$ contains all the marked points of $L_i$ {(including free points)}. 
We will replace each wire $\gamma_i$,
	with a bundle of curves (with intersections but no braiding between the components of the bundle) constructed according to the tree $G_i$, as follows.
	
	All curves in the $i^{th}$ bundle must intersect at all marked points on $\gamma_i$. We will specify the
	additional intersections and explain how to determine the number of curves and free marked points in the bundle. The bundle will be described recursively, via its subbundles 
	and iterative (sub$)^k$-bundles, which we determine by moving through the graph $G_i$. We start at the root and move upward in the graph $G_i$ 
	with respect to the partial order induced by the root, stopping when  we either  reach either a vertex $v_0$  
	of self-intersection number $-s_0$ for $s_0\geq 3$ or exhaust the graph $G_i$.

       By Condition~(\ref{valency-weight}), $(-2)$ vertices can only occur
	in a linear chain. Thus, if we never reach a vertex with self-intersection $-s_0$ for $s_0\geq 3$, 
	 then all vertices of $G_i$ have self-intersection $-2$ {(and $G_i$ is a linear chain)}. Suppose there are $r_0\geq 0$ such $(-2)$ vertices. In that case,  
	the bundle for $G_i$ should consist of only a single curve, but with $r_0\geq 0$ additional free points. 
	(The weights of the decorated germ increase accordingly.)

	 If there exists a vertex $v_0$ of self-intersection $-s_0$ for $s_0\geq 3$ after passing through a linear chain of $r_0$ vertices of self-intersection 
	 $-2$, then the bundle will consist of exactly $s_0-1$ non-empty subbundles. The subbundles will be described as we travel further along $G_i$.
	 We require that all curves in the bundle intersect exactly $r_0$ additional times {(where each of these $r_0$ intersection points gets marked)} and
	 increase the weight of each curve by $r_0+1$, yielding one additional free marked point on each curve. 
	 Two curves in different subbundles will not intersect at any additional points beyond those specified {so far}.
	
	Note that $v_0$ can have at most $s_0-1$ vertices directly above it in $G_i$, since its valency is at most~$s_0$. 
	In particular, $G_i$ itself is built by attaching $s_0-1$ (potentially empty) trees onto the subgraph $\{v\leq v_0\}\subset G_i$.
	We associate  the $s_0-1$ subbundles to these $s_0-1$ {rooted} trees $G_1^1,\dots, G_{s_0-1}^1$, which may be empty or non-empty. {(Note that the partial order on $G$ induced by its root induces a partial order and root on each $G_j^1$.)}
	
	Now we will create subbundles and their  subsubbundles by iteratively repeating a slight modification of the process above.
	For each tree $G_j^1$, we construct a subbundle as follows. Starting at the root of $G_j^1$, we again have a linear chain of $r_1\geq 0$ vertices with self-intersection 
	$-2$, which either exhausts the graph $G_j^1$ or ends in a vertex $v_1$ of self-intersection number $-s_1$ for $s_1\geq 3$. 
	(Note that $r_1$ and $s_1$ depend on $j$, but we drop this index to avoid further notational clutter.) 
	If we are in the first case, where there is no such vertex $v_1$, the subbundle associated to $G_j^1$ will consist of a single curve 
	with $r_1$ additional free marked points. If we are in the second case, where the chain of length $r_1$ of $(-2)$-vertices ends at a vertex $v_1$ 
	with self-intersection $-s_1$ for $s_1\geq 3$, the subbundle itself will be a union of $s_1-1$ non-empty subsubbundles,
	intersecting at $r_1+1$ additional points. (Accordingly, the weights increase by $r_1+1$, but no new free marked points are added.) 
	Two curves in different subsubbundles will not intersect at any additional points beyond those previously specified.

	The $s_1-1$ subsubbundles correspond to the $s_1-1$ potentially empty
	trees $G_1^2,\dots, G_{s_1-1}^2$ attached above $v_1$. We determine these subsubbundles by iteratively repeating this process, 
	where $G_l^{2}$ takes the role of $G_j^1$ and the subsubbundle takes the role of the subbundle.
	Note the (sub$)^k$-bundles will generally have (sub$)^{k+1}$-bundles, leading to additional iterations of the procedure.
	The situation where a (sub$)^k$-bundle does not have a (sub$)^{k+1}$-bundle is when the (sub$)^k$-bundle consists of a single component
	(as in the first case of the procedure). Since the graph is finite, there will be a finite number of iterations,
	so this process will eventually describe the bundle completely.
	
	Having constructed such bundles individually for each $G_i$, we now superimpose them onto the wires $\gamma_i$ as satellites to get a new braided wiring diagram {by inserting them into a small neighborhood of $\gamma_i$ so that each wire of the bundle is $C^1$ close to the original wire $\gamma_i$}.
	Recall that all intersection points between wires are marked in the original diagram, and all curves from the $i^{th}$-bundle are required to intersect at all
	marked points. It follows that curves from the different bundles {are allowed to} intersect {\em only} at the marked points of the original diagram.
	
	We can apply Proposition~\ref{symp-config} to extend the new braided wiring diagram to an arrangement $\Gamma$ of symplectic disks. 
	We claim that via Lemma~\ref{construct-fibration}, the resulting arrangement $\Gamma$ provides a Stein filling for the link of the singularity with the resolution graph 
	$H$. To check the claim, we need to show that the open book decomposition on the boundary of the Lefschetz fibration constructed from $\Gamma$ supports the canonical contact structure for the link associated to $H$. 	Recall that $H$ is associated to a decorated germ $\mathcal{C}^H$ with smooth branches, by attaching $(-1)$~vertices and curvettas and blowing down. We will show that $\Gamma$ is related by a smooth graphical homotopy to another decorated germ $\mathcal{C}$ which is topologically equivalent to $\mathcal{C}^H$. The topological type of $\mathcal{C}$ will be determined by the intersections and marked points in $\Gamma$: the order of tangency between two components in $\mathcal{C}$ is equal to the number of intersections between the corresponding components of $\Gamma$. 
	The weight on each curve is the total number of marked points on the corresponding disk of $\Gamma$, including intersections and free marked points. After showing that $\Gamma$ and $\mathcal{C}$ are related by a smooth graphical homotopy, we will verify that $\mathcal{C}$ and $\mathcal{C}^H$ are topologically equivalent (with corresponding weights) to conclude that the open book decompositions are equivalent.
	
	To relate $\Gamma$ and $\mathcal{C}$, we first construct a smooth graphical homotopy from $\Gamma$ to a ``pencil of the bundles''. In the pencil of the bundles, 
	 all curves will intersect at one point, and curves from different bundles do not intersect anywhere else, but curves from the same bundle may intersect at other points 
	 along the corresponding line. We can use  a smooth graphical homotopy of the original symplectic line arrangement $\{L_i\}$ to a pencil as a guide to build the required 
	 homotopy of $\Gamma$, because 
	each bundle is $C^1$-close to the corresponding symplectic line inside the chosen Milnor ball. Essentially, at this step we treat each bundle as a whole, bringing  different bundles together without perturbing
	curves inside each bundle. More precisely, we satellite the bundle onto the family of wiring diagrams corresponding to the smooth graphical homotopy of the symplectic lines to the pencil. The intersection points within a bundle will 
	remain distinct in this smooth graphical homotopy.  Note that at intermediate times during the homotopy, we allow many additional intersection points in the arrangement,
	as curves from different bundles will intersect outside the common marked intersections.
	
	Next, we show that each bundle can be homotoped so that all the intersections come together to high order tangencies.	
	Let $\Gamma^i$ denote the $i^{th}$ bundle constructed above, and let $\mathcal{C}^i$ denote the curves in the germ $\mathcal{C}$ corresponding to those in $\Gamma^i$. 
 To show that $\Gamma^i$ and $\mathcal{C}^i$ are related by a smooth graphical homotopy, it suffices to check that 
 they have the same boundary braid.
 To verify this, we observe that the subbundling structure looks like the nested structure produced by the Scott deformation 
 of $\mathcal{C}^i$ as in the proof of Proposition~\ref{artin-ob}. The bundle, as drawn in $\R^2$,
 provides a wiring diagram which is planar isotopic to the wiring diagram of the Scott deformation, and thus their braid monodromy is the same.
 As a consequence, each bundle  $\Gamma^i$ is related by a smooth graphical homotopy to $\mathcal{C}^i$. Applying these homotopies to all bundles,
 we see that $\Gamma$ is related to $\mathcal{C}$ by a smooth graphical homotopy, and their induced open books agree.

	Now, we need to check that $\mathcal{C}$ and $\mathcal{C}^H$ are topologically equivalent.
	To this end, we will compare the weights and the parwise orders of tangency between curvettas in the two germs. For $\mathcal{C}$, these quantities are 
	computed from the 
	intersections and marked points in $\Gamma$, while  Remark~\ref{rem:toptype} shows how to compute them from the graph $H$.

	First, we make a few  observations to  relate the curvettas on the graph $H$ to the bundling construction above.  Before the star-shaped graph $G$ is extended, 
	the lines $L_i$ correspond to the legs of the graph. For each $i$, the $i^{th}$ leg is a chain of $(-2)$-vertices, with an end vertex $u_i$. We attach a single 
	$(-1)$ vertex to $u_i$ and put a curvetta on this vertex; this curvetta gives rise to the line $L_i$. By Remark~\ref{rem:toptype}, the  weight of $L_i$ is 
	$1+l(u_0,u_i)$, where $u_0$ is the root of $G$. {In this case, the root has been chosen to be the center of the star-shaped graph.}
	
	When $G_i$ is non-empty,  the symplectic line $L_i$ is replaced  by a collection of $m_i$ curves {(we compute $m_i$ below)} in the germ associated to $H$. These new curves come from 
	curvettas on the additional $(-1)$-vertices attached to $G_i$.  For each $v\in G_i$, $(v\cdot v+a(v))$ additional $(-1)$ vertices are attached to $v$, 
	and each $(-1)$ vertex has a curvetta attached, thus $$m_i = -\sum_{v\in G_i}(v\cdot v+a(v))$$
	as in Proposition~\ref{p:choices}. Note that $m_i$ agrees with the number of curves in the bundle $\Gamma_i$  constructed above for the graph $G_i$.
	This is because the sub-bundling process terminates when you reach a (sub$)^k$-bundle which is a single component.
	This occurs when the (sub$)^k$-bundle corresponds to a (sub$)^k$-tree consisting of only $r\geq 0$ vertices of self-intersection $-2$.
	When $r>0$, this means that there is a $(-2)$~vertex leaf which contributes one to $m_i$, and when $r=0$, 
	this means there is a $(-s)$ vertex $v$ with fewer than $(s-1)$ branches above it, and
	there are correspondingly $-(v\cdot v+a(v))=s-a(v)$ such 
	(sub$)^k$-bundles, each consisting of a single curve.

	Now, let  $C_x$ be one of the curvettas for  the graph $H$, and let $\tilde v_x$ be a vertex of $G$ such that $C_x$ intersects a $(-1)$ vertex attached to $\tilde v_x$.
	According to Remark~\ref{rem:toptype}, the weight of $C_x$ according to the graph $H$ is $1+l(\tilde v_x,u_0)$, where $l(\tilde v_x,u_0)$ 
	counts the number of vertices in the path from the root $u_0$ of $G$ to the vertex $\tilde v_x$. This path consists of several parts.
	 From the original graph $G$, the path contains the $l(u_i,u_0)$ vertices connecting the root $u_0$ to the vertex $u_i$ where $G_i$ is attached.
	 Next, there are vertices from $G_i$, which can be organized into $(K+1)$ chains as shown in Figure~\ref{fig:weights}.  For~$0\leq k\leq K-1$, the $k^{th}$ chain
	 consists of $r_k\geq 0$  vertices of self-intersection $(-2)$, followed by a vertex of self-intersection~$-s_k<-2$.
	 Finally, there may be a last chain of $(-2)$-vertices, of length~$r_{K}\geq 0$, such that $\tilde v_x$ is its last vertex. (If $\tilde v_x \cdot \tilde v_x<-2$, then $r_K=0$.) 
	  Therefore,
	$$1+l(\tilde v_x, u_0) = 1+l(u_i,u_0)+r_K+\sum_{k=0}^{K-1} (r_k+1).$$
	\begin{figure}
		\centering
		\includegraphics[scale=.85]{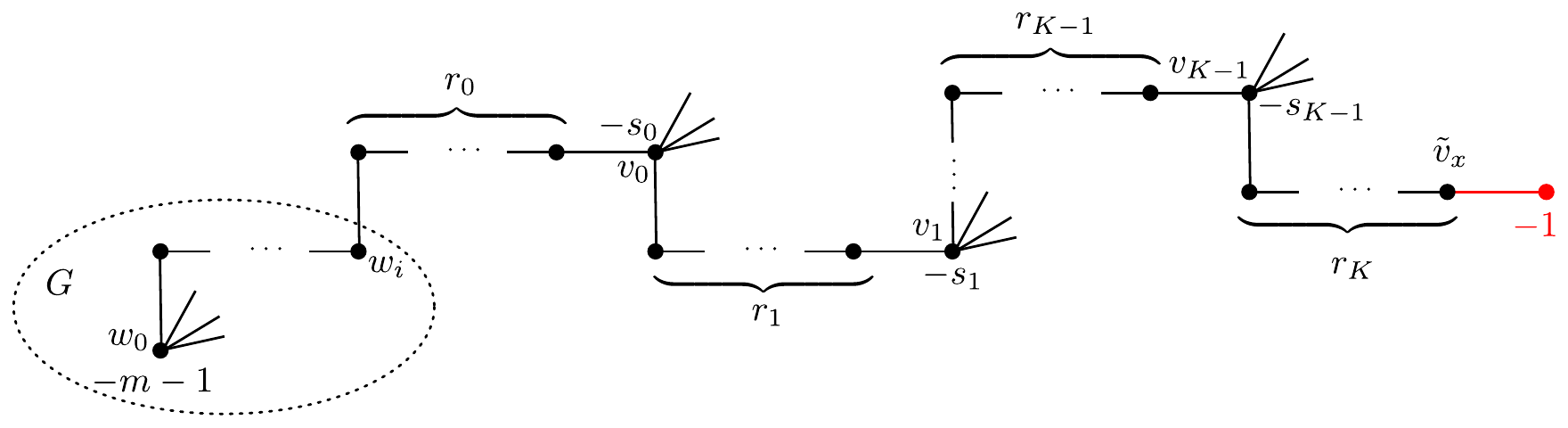}
		\caption{{How to compute the weights from the graph $G$ following the proof of Proposition~\ref{p:orev-manymore}.}}
		\label{fig:weights}
	\end{figure}
	On the other hand, in the construction of the bundle, the initial weight on each curve begins at $1+l(u_i,u_0)$. For each iterative (sub$)^k$-bundle 
	it is included in, the weight is increased by $r_k+1$, until we reach a stage $K$ where the (sub$)^{K}$graph consists of $r_K\geq 0$ vertices, all of self-intersection $-2$. For this $K^{th}$ stage,
	the weight is increased by $r_K$ (the increase is associated to free marked points). Therefore, the total weight on $C_x$ will be
	$$w(C_x)=1+l(u_i,u_0)+r_{K}+\sum_{k=0}^{K-1}(r_k+1),$$
	which agrees with $1+l(\tilde v_x, u_0)$, as required.

	Next, we compare the orders of tangency between the curves. According to Remark~\ref{rem:toptype}, the order of tangency between 
	two components $C_x$ and $C_y$ is $\rho(\tilde v_x,\tilde v_y;u_0)$, the number of common vertices in the path from $\tilde v_x$ 
	to $u_0$ with the path from $\tilde v_y$ to $u_0$. Note that by Condition~\ref{valency-weight}, the vertex $v_L$ where these two paths
	diverge has self-intersection $-s_L$ for $s\geq 3$. See Figure~\ref{fig:tangencies}.
	The path from $u_0$ to $v_L$ includes the path from $u_0$ to $u_i$ in $G$. This contributes $l(u_i,u_0)$ vertices. 
	The path continues into $G_i$, with sequential chains of $r_k$ vertices of self intersection $(-2)$, each ending in a vertex $v_k$ of 
	self-intersection~$-s_k<-2$,  for $0\leq k\leq L$. Therefore
	$$\rho(\tilde v_x,\tilde v_y; v_0) = l(u_i,u_0)+\sum_{k=0}^L(r_k+1).$$
		\begin{figure}
			\centering
			\includegraphics[scale=.85]{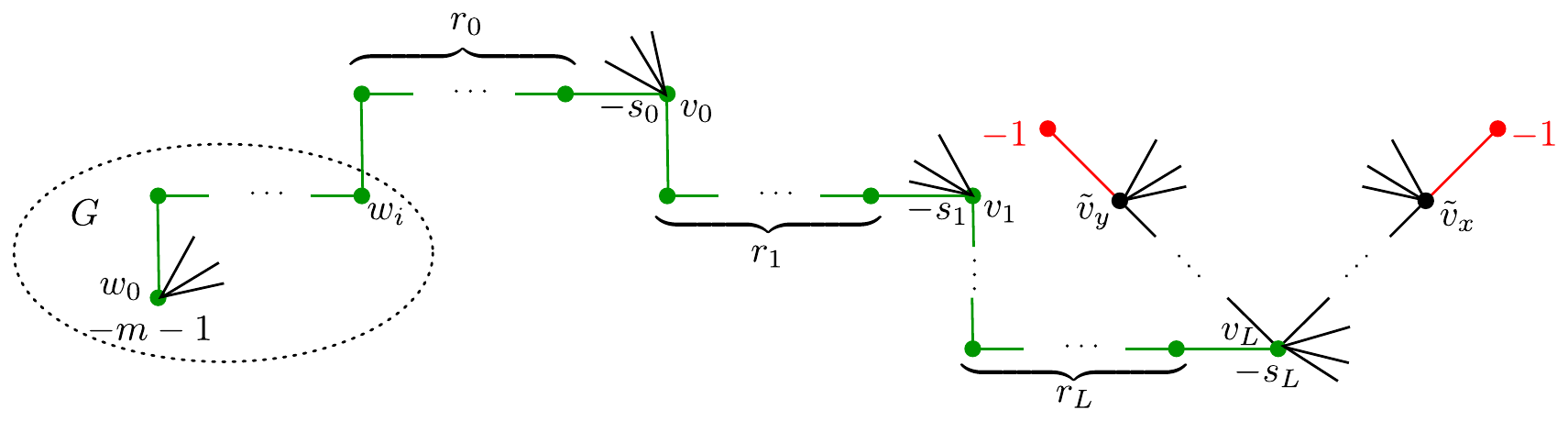}
			\caption{{How to compute the tangencies from the graph $G$ following the proof of Proposition~\ref{p:orev-manymore}.}}
			\label{fig:tangencies}
		\end{figure}
	 On the other hand, in the bundle construction, the curves $C_x$ and $C_y$ lie in two distinct (sub$)^{L+1}$-bundles created for
	 two of the distinct trees lying above vertex $v_L$.
	 No intersections between $C_x$ and $C_y$ will be created after the $L^{th}$ stage. At the beginning of the bundle construction,
	 all curves are required to intersect $1+l(u_i,u_0)$ times. All other intersections between $C_x$ and $C_y$ are created 
	 in the procedure above at some iteration $k$, for $0\leq k\leq L$. At the $k=0$ stage, we add $r_0$ intersections between
	 $C_x$ and $C_y$. At stage $k$ for $1\leq k\leq L$, we add additional $r_k+1$ intersections between $C_x$ and $C_y$. 
	 Therefore the total number of intersections between $C_x$ and $C_y$ is
	 $$1+l(u_0,u_i)+r_0+\sum_{k=1}^L (r_k+1),$$
	 which agrees with $\rho(\tilde v_x,\tilde v_y;v_r)$.

	 To complete the proof, observe that the arrangement $\Gamma$ contains the original unexpected symplectic line arrangement as a subarrangement 
	 (choose a single component of each bundle). By Theorem~\ref{thm:nonMilnor},
	 we obtain unexpected Stein fillings of the link of the sigularity corresponding to the graph $H$.	 
\end{proof}

\section{Further comments and questions on curvetta homotopies}   \label{s:further}

In the previous section we showed that 
Stein fillings of the link of a singularity do not always arise from the Milnor fibers,
even for the simple class of rational singularities with reduced fundamental cycle. Our examples of unexpected Stein fillings 
come from curvetta arrangements that do not arise as picture deformations of the decorated germ representing the singularity, although
 these arrangements  are still related to the decorated germ through a smooth graphical homotopy. In this section, we make 
a detailed comparison of de Jong--van Straten's picture deformations (Definition~\ref{pic-def}) to
smooth graphical homotopies (Definition~\ref{def:homotopy}). 
Observe that the two notions differ in several essential ways. 
Indeed, the curvetta branches are required to be algebraic 
resp. just smooth; positivity of all intersections and the weight restrictions must hold  at all times during a picture deformation but only at the end of a graphical homotopy; the topology of 
the arrangement may change at non-zero times during graphical homotopy but not during a picture deformation. This is summarized in Table~\ref{table}.  
We will explore each of these aspects and their role in differentiating Stein fillings from Milnor fibers. The most important aspect  seems to be  
the topology of the curvetta arrangement, and whether it is allowed to vary during the homotopy.

\begin{table}[tbh] \caption{\label{table}}
	\begin{tabular}{l||l|l|}
		\cline{2-3}      & smooth graphical homotopy     & picture deformation     \\ \hline \hline
		\multicolumn{1}{|l||}{type of curvetta branch $C^t_j$}                                                     & smooth graphical disk           
		& \begin{tabular}[c]{@{}l@{}}disk given by\\ (germ of) algebraic curve\end{tabular} \\ \hline
		\multicolumn{1}{|l||}{\begin{tabular}[c]{@{}l@{}}topology of curvetta \\ arrangement\end{tabular}} & may change with time                                                                                            & remains the same                                                               \\ \hline
		\multicolumn{1}{|l||}{\begin{tabular}[c]{@{}l@{}} weight restrictions: \\ $C^t_j$ has at most $w_j$ intersections \end{tabular}}                                                         & \begin{tabular}[c]{@{}l@{}}only hold for final arrangement,\\ may be violated during homotopy\end{tabular}      & hold at all times                                                              \\ \hline
		\multicolumn{1}{|l||}{ \begin{tabular}[c]{@{}l@{}} positivity of intersection points: \\ $C^t_i \cdot C^t_j>0$ \end{tabular}}                   
		& \begin{tabular}[c]{@{}l@{}}only hold for the final arrangement, \\ may be violated during homotopy\end{tabular} & hold at all times                                                              \\ \hline
	\end{tabular}
\end{table}


\subsection{Algebraic versus smooth}
The first difference between picture deformations and homotopies is that a smooth graphical homotopy includes curvettas which need not be complex algebraic curves, 
either during the course of the homotopy or at the end of the homotopy. It turns out that
this is not the key aspect contributing to the difference between Milnor fillings and Stein fillings in our examples. Indeed, adding higher-order terms, one can produce some surprising curvetta arrangements. 
Because the curvettas are open algebraic disks, possibly given by high degree algebraic equations, 
curvetta arrangements can be more general than arrangements of complex lines or global algebraic curves. 
To illustrate, we recall the example of pseudo-Pappus arrangement from \cite{dJvS}, see Figure~\ref{pseudopappus}:

\begin{figure}[htb]
	\centering
	\bigskip
	\labellist
	\small\hair 2pt
	\pinlabel $p$ at 56 45
	\pinlabel $q$ at 86 52
	\pinlabel $r$ at 109 53
	\endlabellist
	\includegraphics[scale=1.2]{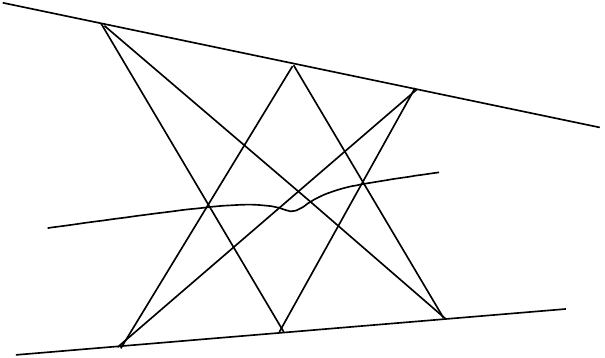}
	\caption{The pseudo-Pappus arrangement.}
	\label{pseudopappus}
\end{figure}

\begin{example} \label{pseudopapp} \cite{Gru, dJvS} Recall that the classical Pappus arrangement consists of 9 lines (we have already discussed this arrangement in Example~\ref{ex:pappus}.
	By the Pappus theorem, the points $p$, $q$, $r$ in the middle of Figure~\ref{pseudopappus} are collinear. In the pseudo-Pappus arrangement, the line through these three points is replaced 
	by a bent pseudoline that passes through two points but not through the third. The pseudo-Pappus arrangement cannot be realized by complex lines. However, the bent pseudoline can be given by a graph of a high-degree polynomial whose additional intersections with the other lines occur sufficiently far outside the ball we restrict to. Thus, the pseudo-Pappus arrangement can be realized by higher-degree open algebraic curves. 
	In fact, as mentioned in \cite{dJvS}, the pseudo-Pappus arrangement arises as a picture deformation of the pencil 
	of 9 lines, with the weights of each line given by the number of intersection points on the corresponding line in the arrangement. The picture deformation can be  obtained by adding small higher-order terms to the linear deformation of the pencil to the classical Pappus arrangement.
	Thus, the pseudo-Pappus arrangement gives rise to Milnor fibers of smoothings of the singularities given by the 
	corresponding decorated pencil of 9 lines.
\end{example}

In fact, all of the fillings produced via arrangements of real pseudolines can 
be obtained from an algebraic curvetta arrangement which can be deformed by a polynomial homotopy (through algebraic curves) to a pencil of lines. 
(However, this family does not constitute a picture deformation because the topology may vary at different $t\neq 0$,
and the weight constraints may fail at intermediate times.) Note that we only consider a portion of the algebraic curves in a chosen ball surrounding the origin. In particular, the algebraic curves may intersect additional times outside of this ball, but we do not need to count such intersections in the incidence data of our arrangement.

\begin{prop} \label{algebraicapprox} Let $\Lambda=\{\ell_1, \dots, \ell_m\}$ be an arrangement of real pseudolines in $\R^2$. Then 
there exists a family of complex algebraic curves $\{\Gamma_1^t,\dots, \Gamma_m^t\}$, given by polynomial equations
$$
\Gamma_i^t=\{y = p(x,t)\},
$$
and a smoothly embedded closed $4$-ball $B\subset \C^2$, such that $\{\Gamma_1^t,\dots, \Gamma_m^t\}$ is a symplectic line arrangement in $B$ (with intersections in the interior of $B$) for every $t\in [0,1]$, where
\begin{itemize}
 \item $B\cap (\Gamma_1^0\cup\cdots \cup \Gamma_m^0)$ has the incidences of a pencil of lines, 
 
 \item  $B\cap (\Gamma_1^1\cup \cdots \cup \Gamma_m^1)$ is isotopic in $B$ to the symplectic extension of 
 the pseudoline arrangement $\ell_1\cup\cdots \cup \ell_m$ given by Proposition~\ref{symp-config}.
 \end{itemize}
\end{prop}

Before proving the proposition, we discuss its consequences.

\begin{remark} \label{change-topology} Consider an arbitrary pseudoline arrangement $\ell_1, \dots, \ell_m$ and the corresponding 
symplectic line arrangement $\{\Gamma_1,\dots, \Gamma_m\}$. By Proposition~\ref{p:symp-lines-fill}, this arrangement gives Stein fillings of the spaces $(Y(m;w_1,\dots, w_k),\xi)$ whenever the 
weights satisfy inequalities $w_k\geq w(\Gamma_k)$, $k=1, \dots, m$.  Let $\Gamma^t=\{\Gamma_1^t,\dots, \Gamma_m^t \}$
be a polynomial 
homotopy  between  a pencil of lines  and the arrangement $\{\Gamma_1,\dots, \Gamma_m\}$; such a homotopy always exists by Proposition~\ref{algebraicapprox}.   
A priori, the homotopy may violate the weight constraints: at some moment $t$, the number of intersections may increase, so that $w(\Gamma_k^t)>w_k$. 
(In fact, the homotopy constructed in Proposition~\ref{algebraicapprox} converts all multiple intersections into double points and thus creates a lot of additional intersections.)
However, since $\Gamma_k^t$ intersects each of the other $m-1$ components exactly once, $w(\Gamma_k^t)$ will never exceed $m-1$. Thus, if  $w_k \geq m-1$ for all $k$, 
any homotopy as above will satisfy the weight constraints. By construction, intersections between any two components $\Gamma_i^t$ and $\Gamma_j^t$ remain positive for all $t$. 
Thus, the homotopy $\Gamma^t$ satisfies the requirements of the first, third, and fourth lines in Table~\ref{table}, sharing these properties with picture deformations, 
but it changes topology of the arrangement. Accordingly, the arrangement $\{\Gamma_1^t,\dots, \Gamma_m^t \}$ gives a Stein filling $W_t$ of $(Y(m;w_1,\dots, w_k),\xi)$ for every 
$t$, and $W_t$ carries a Lefschetz fibration as in Lemma~\ref{construct-fibration}, but topology of the fillings $W_t$ changes with $t$. Note also that for small $t>0$,
the defining polynomials for $\Gamma_k^t$ give an unfolding, and thus a 1-parameter deformation of $\mathcal{C}$. Equipped with marked points, this gives a picture deformation. 
Therefore, for small $t>0$ the Stein filling $W_t$ is given by a Milnor fiber. As $t$ increases and the topology of the arrangement changes, we obtain new fillings $W_t$ which may 
not be realizable by Milnor fibers. We will consider a specific example of such a topology change in Subsection~\ref{subs:change-top}.

The conclusion we wish to draw here is that the difference between algebraic curves and smooth curves is not the essential to our counterexamples, as 
we can realize the corresponding symplectic line arrangements by complex algebraic curves and construct polynomial homotopies. The positivity of intersections and
the weight constraints can often be trivially satisfied, although we further discuss the role of  weights in Subsection~\ref{subs:weight}.
In fact, the important difference comes from the second aspect in Table~\ref{table}, namely smooth graphical homotopies 
can vary their topology and singularities in various different ways during the homotopy, whereas picture deformations
must maintain the same topology for all non-zero parameters $t$. 

\end{remark}

We now turn to the proof of Proposition~\ref{algebraicapprox}.	Given any pseudoline arrangement, it can be isotoped in $\R^2$ to be in a standard wiring diagram form,
with the following properties. Each pseudoline is graphical $\ell_i = \{ y = f_i(x) \}$. Away from intersection points,
each pseudoline is horizontal with $f_i(x) = 2\delta n$ for some integer $1\leq n \leq m$ and a fixed constant $\delta>0$.
There are disjoint intervals $(a_1,b_1),\dots, (a_r,b_r)$ at which $f_i(x)$ is non-constant, such that there is a unique point in each interval $(a_k,b_k)$ at which $\ell_i$ intersects other pseudolines. Furthermore, we ask that $f_i$ and $f_j$ are linear whenever $|f_i(x)-f_j(x)|<\delta$, and each $f_i(x)$ is monotonic in each interval $(a_k,b_k)$. We will assume after a planar isotopy of $\Lambda$ our pseudoline arrangement is initially given in this form. To construct our algebraic family, we first require a smooth family of pseudolines connecting this given pseudoline arrangement in standard wiring diagram form to a pencil, and satisfying a quantitative transversality property as follows.
	
	\begin{lemma}\label{l:pseudofam}
		Let $\Lambda=\{\ell_1, \dots, \ell_m\}$ be an arrangement of real pseudolines in $\R^2$ in standard wiring diagram form with constant $\delta$, such that all intersections occur in $[-M,M]\times \R$. Then there exist smooth functions $f_i: [-M,M]\times [0,1]\to \R$ with the following properties:
		\begin{enumerate}
			\item $\ell_i = \{y = f_i(x,1) \}$ (at time $1$ the graphs of the functions give the chosen pseudoline arrangement).
			\item $f_i(x,0) = c_i x$ (at time $0$ the graphs of the functions give a linear pencil).
			\item For any $t_0\in [0,1]$, and any $i\neq j$, there is a unique point $\bar{x}\in [-M,M]$ such that $f_i(\bar{x},t_0) = f_j(\bar{x},t_0)$ and an interval $\bar{x}\in (a,b)\subset [-M,M]$ such that $|f_i(x,t_0)-f_j(x,t_0)|< \delta$ if and only if $x\in (a,b)$ (the pseudolines remain at least distance $\delta$ apart except in a neighborhood of their unique intersection).
			\item \label{i:qtrans} For any $t_0\in[0,1]$, and any $x_0\in [-M,M]$ such that $|f_i(x_0,t_0)-f_j(x_0,t_0)|<\delta$, we have that
			$$\left| \frac{\partial f_i}{\partial x}(x_0,t_0) - \frac{\partial f_j}{\partial x}(x_0,t_0) \right| > \eta:=\delta/2M$$
			(whenever the pseudolines become close enough to intersect, their slopes are quantitatively far from each other to ensure isolated transverse intersections).
		\end{enumerate}
	\end{lemma}
	
	\begin{proof}
		Note that when the original pseudoline arrangement $\{\ell_i\}$ is in standard wiring diagram form, 
		it does satisfy property (\ref{i:qtrans}) of the lemma when $t_0=1$. This is because
		whenever $|f_i(x,1)-f_j(x,1)|<\delta$, the function $f_i-f_j$ is linear, and it interpolates a height difference greater 
		than $\delta$ over an interval smaller than $2M$, so its slope is greater than $\eta$.
		
		It was proven that any arrangement of pseudolines in standard wiring diagram form can be related through 
		a family of pseudolines to a pencil in~\cite[Proposition 6.4]{RuSt}. In that paper, what is needed
		is that the pseudolines maintain transverse intersections throughout the family, whereas we need a quantitative measure of this transversality.
		We demonstrate here that this stronger condition is in fact  satisfied by the family in ~\cite{RuSt}.
		
		We briefly recall the key aspects in the construction of the family and refer the reader to~\cite[Proposition 6.4]{RuSt} for further details. 
		This family is graphical and thus can be written as $\ell_i^t=\{y=f_i(x,t)\}$, for $i=1,\dots, m$ where $\ell_i^1=\ell_i$.
		The key move to modify the pseudoline arrangement into a pencil through a family is shown in Figure~\ref{fig:pseudohomotopy}
		(this figure is a slight modification of that appearing in \cite[Figure 8]{RuSt}). This move is used iteratively to break 
		up $k$-tuple points into a sequence of double points in a particular order. This procedure can be reversed to form an $m$-tuple point from a collection of appropriately ordered double points at the end to obtain a pencil. The order of the double points can be modified through the moves shown in Figures~\ref{fig:braid1} and~\ref{fig:braid2}, by a classical theorem of Matsumoto and Tits \cite{Matsumoto}.
		
		\begin{figure}
			\centering
			\includegraphics[scale=.5]{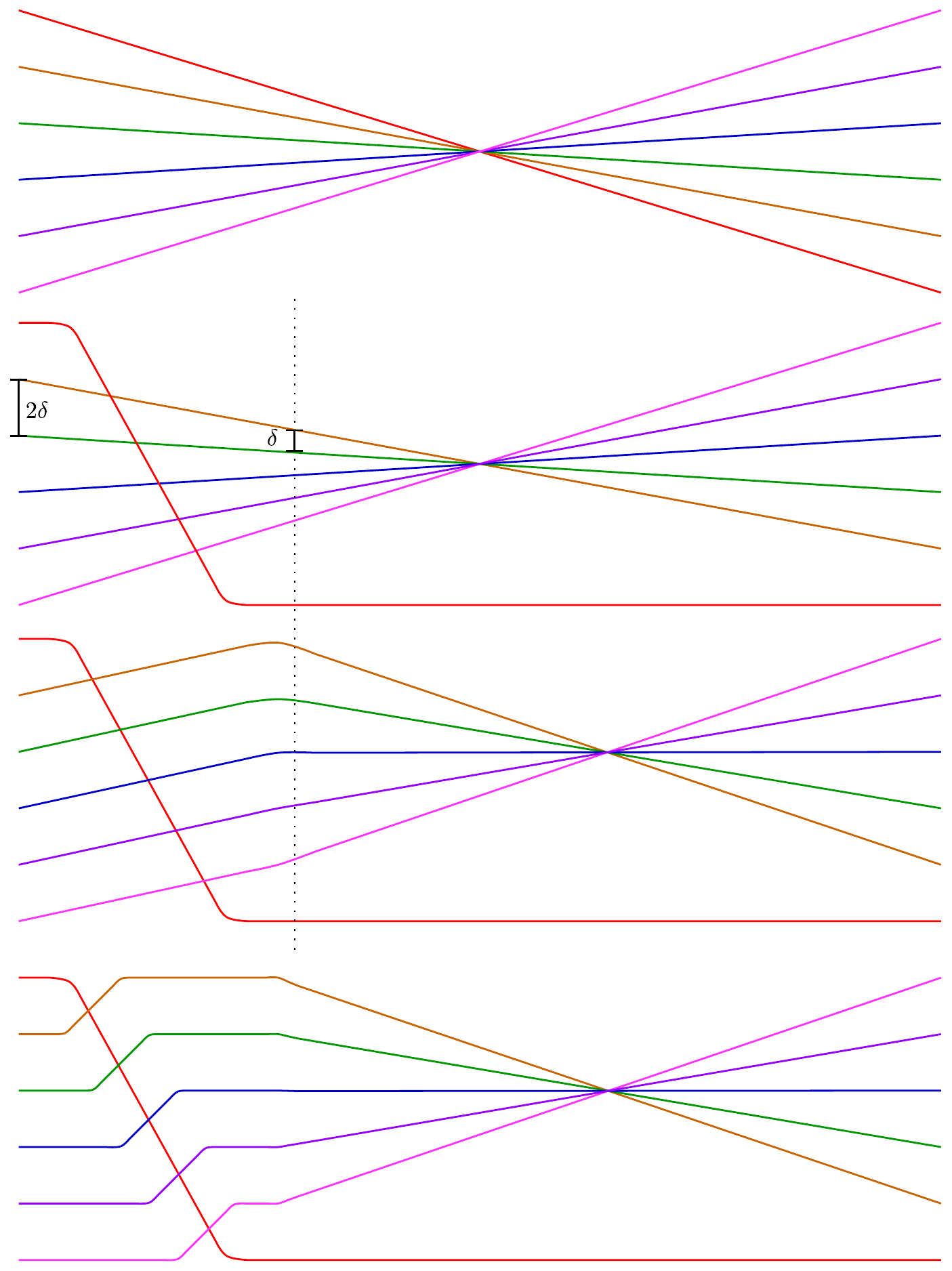}
			\caption{Key move used to construct a family of pseudolines, slightly modified from~\cite{RuSt}.}
			\label{fig:pseudohomotopy}
		\end{figure}
		
		\begin{figure}
			\centering
			\includegraphics[scale=.5]{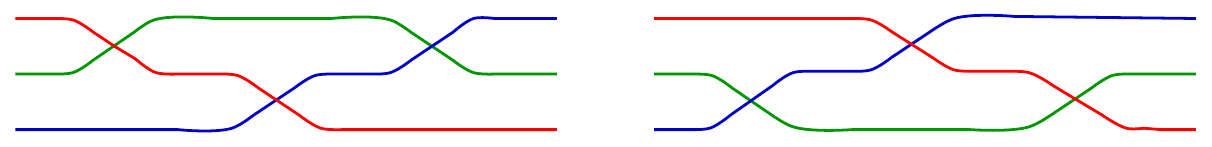}
			\caption{First reordering move.}
			\label{fig:braid1}
		\end{figure}
		
		\begin{figure}
			\centering
			\includegraphics[scale=.5]{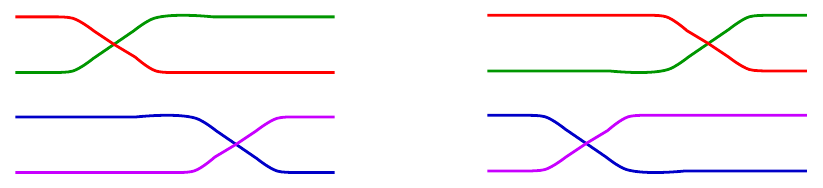}
			\caption{Second reordering move.}
			\label{fig:braid2}
		\end{figure}
		
		If a pseudoline arrangement satisfies the transversality property (\ref{i:qtrans}) before
		the move in Figure~\ref{fig:braid2}, then it will continue to satisfy the same property throughout the move,
		because the relative slopes remain the same, only the interval where they occur is translated.
		
		For the move from Figure~\ref{fig:braid1}, this can be realized using Figure~\ref{fig:pseudohomotopy} once in 
		reverse to form a triple point and then again in the forwards time direction but mirrored to break up the triple point 
		in the opposite manner (see~\cite[Figure 10]{RuSt}). Therefore it suffices to ensure that property (\ref{i:qtrans}) is 
		satisfied throughout the move shown in Figure~\ref{fig:pseudohomotopy}. 
		Indeed, throughout this move, whenever a pair of pseudolines have height difference less 
		than $\delta$ (recall that the spacing between the heights of the strands at the left and right ends of the figure is $2\delta$), 
		both pseudolines are linear in this interval. The difference of pairwise slopes whenever $|f_i(x,t)-f_j(x,t)|<\delta$ is always greater 
		than $\eta$ throughout this family,
		because each crossing  changes the difference in $f_i-f_j$ by at least $2\delta$ across the interval,
		whereas the interval has length at most $2M$. Moreover, this move preserves the property 
		that there is a unique interval at which a given pair satisfies $|f_i(x,t)-f_j(x,t)|<\delta$.
	\end{proof}

\begin{proof}[Proof of Proposition~\ref{algebraicapprox}]
	We use the functions $\{f_i(x,t)\}$, representing a family of pseudolines through their graphs at a fixed time $t$, 
	and approximate these by real polynomials intersecting in somewhat controlled ways. 
	We assume that $x\in[-M,M]$ and that $M\geq 1$. Recall that our final pseudoline arrangement is given by $\ell_i = \{y=f_i(x,1)\}$.  Let $x_1,\dots, x_n$ be the points at which $f_i(x_k,1) = f_j(x_k,1)$ for some $i\neq j$.
	
	Let $\varepsilon>0$. Let $\zeta=\min\{1,\min_{i\neq j}\{|x_i-x_j|\} \}$. In particular, $\zeta\leq 1$.
	
	Using the Stone-Weierstrass approximation theorem, choose polynomials $\widetilde{p}_i(x,t)$ such that
	$$\left| \frac{\partial f_i}{\partial x}(x,t) - \widetilde{p}_i(x,t)  \right| < \frac{\varepsilon \zeta^{n-1}}{4n^2(2M)^{n}}.$$
	Then by integrating $\widetilde{p}_i(x,t)$ and shifting by a constant, we can find $\overline{p}_i(x,t)$ such that $\frac{\partial \overline{p}_i}{\partial x}(x,t) = \widetilde{p}_i(x,t)$ and 
	$$\left|\overline{p}_i(x,t) - f_i(x,t)\right| < \frac{\varepsilon\zeta^{n-1}}{4n^2(2M)^{n-1}}.$$
	Now for $k=1,\dots, n$ let
	$$a_k^i = \frac{\left(f_i(x_k,1) - \overline{p}_i(x_k,1)\right)}{(x_k-x_1)\cdots(x_k-x_{k-1})(x_k-x_{k+1})\cdots (x_k-x_n)}.$$
	Let $a_0^i = \overline{p}_i(0,0)$.
	Define
	\begin{align*}
	p_i(x,t) = \overline{p}_i(x,t) +a_0^i(t-1) &+ a_1^it(x-x_2)\cdots (x-x_n)\\
	&+ a_2^i t(x-x_1)(x-x_3)\cdots (x-x_n)+ \cdots + a_n^it(x-x_1)\cdots (x-x_{n-1}). 
	\end{align*}
	Then we have that for every $k=1,\dots, n$, $p_i(x_k,1) = f_i(x_k,1)$ and $p_i(0,0)=p_j(0,0)=0$ for all $i,j$. 
	In particular, for every multi-intersection point of the pseudolines $\ell_1,\dots, \ell_m$ there is a multi-intersection point of 
	the corresponding $\{p_1(x,1)=0\},\dots, \{p_m(x,1)=0\}$. We will show that 
	the curves $\gamma_1^{t_0}:=\{p_1(x,t_0)=0\},\dots, \gamma_m^{t_0}:=\{p_m(x,t_0)=0\}$ form a pseudoline arrangement at each time $t_0$ (namely every pair of components intersects exactly once). In particular, this suffices to show that at $t_0=1$, 
	the algebraic arrangement has the same intersections as the smooth pseudoline arrangement. For this, we use the following bounds:
	\begin{eqnarray*}
	\left| p_i(x,t) - f_i(x,t)\right | &\leq& \left| p_i(x,t) - \overline{p}_i(x,t) \right| + \left| \overline{p}_i(x,t) - f_i(x,t)\right|\\
	& \leq & a_0^i+ \sum_{k=1}^n a_k^i (2M)^{n-1} + \frac{\varepsilon\zeta^{n-1}}{4n^2(2M)^{n-1}}\\
	& \leq & \frac{\varepsilon\zeta^{n-1}}{4n^2(2M)^{n-1}} + \sum_{k=1}^n \frac{\varepsilon}{4n^2(2M)^{n-1}} \cdot (2M)^{n-1} + \frac{\varepsilon \zeta^{n-1}}{4n^2(2M)^{n-1}}\\
	& < &\varepsilon.
	\end{eqnarray*}
	We can similarly bound the difference of the derivatives with respect to $x$:
		$$\left| \frac{\partial p_i}{\partial x}(x,t) - \frac{\partial f_i}{\partial x}(x,t) \right| \leq a_0^i+
		\sum_{k=1}^n a_k^i n(2M)^{n-2} +\frac{\varepsilon\zeta^{n-1}}{4n^2(2M)^{n-1}} <\varepsilon.$$
	
	Now we want to show that the graphs $\lambda_i^t:=\{y=p_i(x,t) \mid x\in [-M,M] \}$ provide a family of algebraic pseudoline arrangements 
	whose incidences agree with $\{\ell_i\}$ at $t=1$, and agree with the incidences of a pencil at $t=0$. 
	We will use the intersection and quantitative transversality properties of Lemma~\ref{l:pseudofam}, to verify that for each time $t_0\in [0,1]$, 
	there is a unique transverse intersection between $\lambda_i^{t_0}$ and $\lambda_j^{t_0}$ where $p_i(x,t_0)=p_j(x,t_0)$ for $x\in [-M,M]$.
	
	Since we could choose $\varepsilon>0$ arbitrarily in the argument above, we now set $\varepsilon = \min\{\delta/3, \eta/3\}$. For each $t_0\in [0,1]$ and each pair $i\neq j$, there is an interval $(a,b)$ such that for $x\in [-M,M]\setminus (a,b)$, $|f_i(x,t_0)-f_j(x,t_0)|\geq \delta$. By the triangle inequality, for $x\in [-M,M]\setminus (a,b)$,
	$$|p_i(x,t_0)-p_j(x,t_0)|\geq |f_i-f_j|-|f_i-p_i|-|p_j-f_j| > \delta - 2\varepsilon \geq \delta/3>0.$$
	Therefore $p_i(x,t_0)\neq p_j(x,t_0)$ for $x\in [-M,M]\setminus (a,b)$. Now for $x\in(a,b)$, we have that $|f_i(x,t_0)-f_j(x,t_0)|<\delta$ so
	by the last property of Lemma~\ref{l:pseudofam},
	$$\left| \frac{\partial f_i}{\partial x}(x,t_0) - \frac{\partial f_j}{\partial x}(x,t_0) \right| > \eta.$$
	Again by the triangle inequality and the bounds above we get that
	$$\left| \frac{\partial p_i}{\partial x}(x,t_0) - \frac{\partial p_j}{\partial x}(x,t_0) \right| > \eta/3.$$
	Since the difference of the derivatives is bounded away from zero, this implies that there can be \emph{at most} one value $x\in (a,b)$ such that $p_i(x,t_0) = p_j(x,t_0)$. 
	
	Because $f_i(x,t_0)$ and $f_j(x,t_0)$ intersect once in the interval $(a,b)$ and their distance is $\delta$ at the end-points $a$ and $b$, up to switching $i$ and $j$, we have $f_i(a,t_0)-f_j(a,t_0) = \delta = f_j(b,t_0)-f_i(b,t_0)$. Since $|p_i(x,t)-f_i(x,t)|<\delta/3$ and $|p_j(x,t)-f_j(x,t)|<\delta/3$, this implies that $p_i(a,t_0)>p_j(a,t_0)$ and $p_j(b,t_0)>p_i(b,t_0)$. Therefore there must exist \emph{at least} one value $x\in (a,b)$ such that $p_i(x,t_0) = p_j(x,t_0)$. Therefore the arrangement $\{\lambda_i^{t_0}\}_{i=1}^m$ is a pseudoline arrangement for all $t_0\in [0,1]$.
	
	Finally, view the variable $x$ as a complex variable. Let $B = [-M,M]\times i[-\alpha,\alpha]\times D_R \subset \C^2$ where $D_R$ 
	is a disk of sufficiently large radius $R$ such that all $|p_i(x,t)|<R$ for $x\in [-M,M]\times i[-\alpha,\alpha]$. We consider 
	the locus $\{\prod_{i=1}^m (y-p_i(x,t))=0\}\subset B$ for each $t\in [0,1]$, and label its irreducible components as
	$\Gamma_i^t = \{ y-p_i(x,t)=0 \mid (x,y)\in B \}$. If $\alpha>0$ is chosen sufficiently small, then all of
	the intersections where $p_i(x,t)=p_j(x,t)$ with $x\in [-M,M]\times i[-\alpha,\alpha]$ occur at real values of $x$. 
	Therefore this complexification of the $\lambda_i^{t_0}$ restricted to $B$ gives an algebraic family of curves, which for 
	any $t_0\in [0,1]$ is a symplectic line arrangement, at $t_0=0$ has the incidences of a pencil, and at $t_0=1$ has the incidences of the original pseudoline arrangement $\{\ell_i\}$.
\end{proof}

\begin{remark} To prove Proposition~\ref{algebraicapprox}, we started with a particular smooth homotopy between the given pseudoline arrangement and the pencil; this homotopy 
was provided by Lemma~\ref{l:pseudofam}. Note that the same argument applies to an {\em arbitrary} smooth graphical homotopy that has the properties stated 
in  Lemma~\ref{l:pseudofam}. In many examples such as those in Section~\ref{s:examples}, a homotopy with required properties can be easily constructed directly, thus we can 
find its polynomial approximation without resolving all multiple intersections into double points as required by the algorithm of  Lemma~\ref{l:pseudofam}. However, we are 
unable to do the polynomial approximation while preserving all the incidence relations during the homotopy (we only guarantee the required incidences agree with those of the homotopy for $t=0$ and $t=1$ 
but not for $0<t<1$.) 

\end{remark}

\subsection{Smooth graphical homotopies imitating picture deformations}

Even without the algebraic condition, we can define a subclass of smooth graphical homotopies which produce Stein fillings 
constrained in a similar way as Milnor fibers. We now isolate these key properties of a picture deformation needed to detect 
 the examples of unexpected Stein fillings in Section~\ref{s:examples}.

We can describe a smooth graphical homotopy with branches $C_k^t \subset \C^2$ via equations 
\begin{equation} \label{eq:param}
f_k(x_1, x_2, t) - y=0.
\end{equation}
where $(x,y)$ are the complex coordinates on $\C^2$, $x=x_1+i x_2$, and $t$ is the real homotopy parameter.
At~$t=0$, we assume that $\cup_{i=1}^k C_k^0=\mathcal{C}$ is the germ of a complex algebraic curve where each branch passes through the origin. 
In particular, $f_k(0, 0, 0)=0$ for all $k$. Additionally, any two branches of $\mathcal{C}$ have positive total algebraic intersection number, so any two deformed branches $C_i^t$, $C_j^t$ intersect for small $t>0$. Composing the homotopy with a $t$-dependent translation, we can also assume that the first two branches always intersect at the origin, $C_1^t \cap C_2^t =0$. 

As before, we will assume that the deformed branches $C^t_k$ are not all concurrent for $t>0$. 
This means that for $t>0$, at least one of the functions $f_k(0, 0, t)$, $k>2$, is non-zero. We need a non-degenerate version of non-concurrence:
\begin{equation} \label{nonzero-deriv}
\exists k \in \{3, \dots, m\}, \exists r>0 \text{ such that }  \frac{\d^r f_k}{\d t^r}(0, 0, 0) \neq 0.
\end{equation}
In other words, if we set $\ord_t f_k = \min \{r:  \frac{\d^r f_k}{\d t^r}(0, 0, 0) \neq 0\}$, then $\ord_t f_k$ is finite for at least some values $k=3, \dots, m$. 
Intuitively, this condition says that the branches move away from being concurrent at the infinitesimal level.

In addition to the above non-degeneracy hypothesis, assume that for all $t>0$ the arrangements $\{C_1^t, C_2^t, \dots, C_m^t\}$ are
topogically equivalent. It follows that each curvetta $C_i^t$ has a finite number of intersections with the other curvettas $C_j^t$, $i \neq j$; the incidence pattern, and the number of intersections, remains constant during the homotopy. We can add decorations so that  all intersection points on  $\cup_{i=1}^m C^t_i$ are marked; as 
for picture deformations, we allow free marked points as well. Let $w_k$ be the total number of marked points on the branch $C^t_k$, for any $t>0$, and set
$w=(w_1, w_2, \dots, w_m)$. We will use the term {\em small smooth deformation} to refer to a smooth graphical homotopy of the decorated germ $(\mathcal{C},w)$
with special properties as above. Small smooth deformation  mimic picture deformations in the 
smooth category, using smooth graphical instead of algebraic curvettas: they preserve the topology of the curvetta arrangement and
satisfy the same weight restrictions and positivity of intersection properties.

\begin{prop} \label{p:smallsmooth} Lemma~\ref{get-lines} holds for small smooth deformations of plane curve germ $\mathcal{C}$ with smooth branches. 
\end{prop}

\begin{proof} 
The proof remains almost the same, but we  have to use Taylor approximations of smooth functions instead of power series for analytic functions.

In complex coordinates $(x, y)$ on $\C^2$, the complex tangent line to $C_k$ at $0$ has the form $a_k x - y =0$  for $a_k \in \C$. 
Setting $x=x_1+ix_2$ and identifying $\C^2$ with $\R^2 \times \C$, the complex tangent line becomes the 2-plane $a_k x_1 + i a_k x_2 - y =0$.  
Set $b_k(t)=f_k(0, 0, t)$ and  $g_k(x, y, t)= f_k(x, y, t) - a_k x_1 - i a_k x_2 - b_k(t)$.  Since  $g_k(0, 0, t)=0$ for all $t$, 
	we have $\frac{\d^\gamma g_k}{\d t^ \gamma}(0, 0, 0)=0$ for all $\gamma$; additionally, $\frac{\d g_k}{\d x}(0, 0, 0)=0$ and $\frac{\d g_k}{\d y}(0, 0, 0)=0$. 
	Equation~(\ref{eq:param}) for the deformed branch $C^t_k$ becomes 
	\begin{equation}
	a_k x_1 + i a_k x_2 + b_k(t) + g_k(x_1, x_2, t) -y=0.
	\end{equation}
Using~\ref{nonzero-deriv},  $r=\min_k \ord_t b_k(t)= \ord_t b_{k_0}(t) < +\infty$, and write $b_k(t) = t^r \bar{b}_k(t)$ for all $k$. 

We now use the Taylor formula for each function $g_k(x_1, x_2, t)$ at $(0, 0, 0)$, writing out the terms up to $r$-th order, followed by the remainder.  	
	This gives 
	\begin{equation}  \label{eq:taylor}
	\begin{split}
	a_k x_1 + i a_k x_2 + t^r \bar{b}_k(t) + & \sum_{\substack{1<\alpha+\beta+\gamma\leq r \\ \alpha>0 \text{ or } \beta>0}} 
	\frac{\d^{\alpha+\beta+\gamma}g_k}{\d x_1^{\alpha} \d x_2^{\beta} \d t^\gamma} (0,0,0) x_1^\alpha x_2^\beta t^\gamma \\
	+&\sum_{\substack{\alpha+\beta+\gamma= r \\ \alpha>0 \text{ or } \beta>0}} h_{k; \alpha, \beta, \gamma}(x_1, x_2, t) x_1^\alpha x_2^\beta t^\gamma + 
	h_{k; 0, 0, r}(0, 0, t) t^r - y =0. 
	\end{split}
	\end{equation}
	The remainder function $h_{k; \alpha, \beta, \gamma}$ is continuous
	for each $(k; \alpha, \beta, \gamma)$, and 
	$h_{k; \alpha, \beta, \gamma}(x_1, x_2, t) \to 0$ when $(x_1, x_2, t) \to (0, 0, 0)$.
	Now make a change of variables
	$$
	x_1= t^r x_1', \qquad x_2 = t^r x_2', \qquad y= t^r y'.
	$$
	It is not hard to see that, as in Lemma~\ref{get-lines}, after the change of variables we can divide Equation~(\ref{eq:taylor}) $t^r$ for $t \neq 0$ and take
	the limit as $t \to 0$. The result is an arrangement of non-concurrent complex lines given by equations   
	$a_k x' + \bar{b}_k(0)- y'=0$. Since we have assumed that the incidence relations for $C^t_1, \dots, C^t_m$ remain the same for all $t\neq 0$, 
	the same relations must hold for the lines. 
\end{proof}

	
	

As a consequence, small smooth deformations cannot 
produce the unexpected symplectic line arrangements that gave unexpected Stein fillings in Section~\ref{s:examples}.
In such examples, to obtain deformations which produce only Milnor fibers, the algebraic condition on the curves and deformation 
is less important than keeping the topology of the curves constant for $t\neq 0$. For rational singularities with
reduced fundamental cycle, small smooth deformations give a symplectic analogue of smoothings, picking out the Stein fillings 
which are ``closest'' to the singularity and its resolution.

\subsection{Smooth graphical homotopies changing topology} \label{subs:change-top}

The key difference between picture deformations and smooth graphical homotopies in Table~\ref{table} is that
the topology of the union of the curves is allowed to change multiple times during a smooth graphical homotopy (for picture deformations, 
the only change happens at time $0$).  In other words, the types of singularities where the curves intersect can vary during the homotopy. 

Here we provide an explicit example to illustrate the topology change in the family of Lefschetz fibrations. Our 
example is related to the configuration $\mathcal{Q}$ from Example~\ref{example-orevkov}, but with a careful choice of weights.

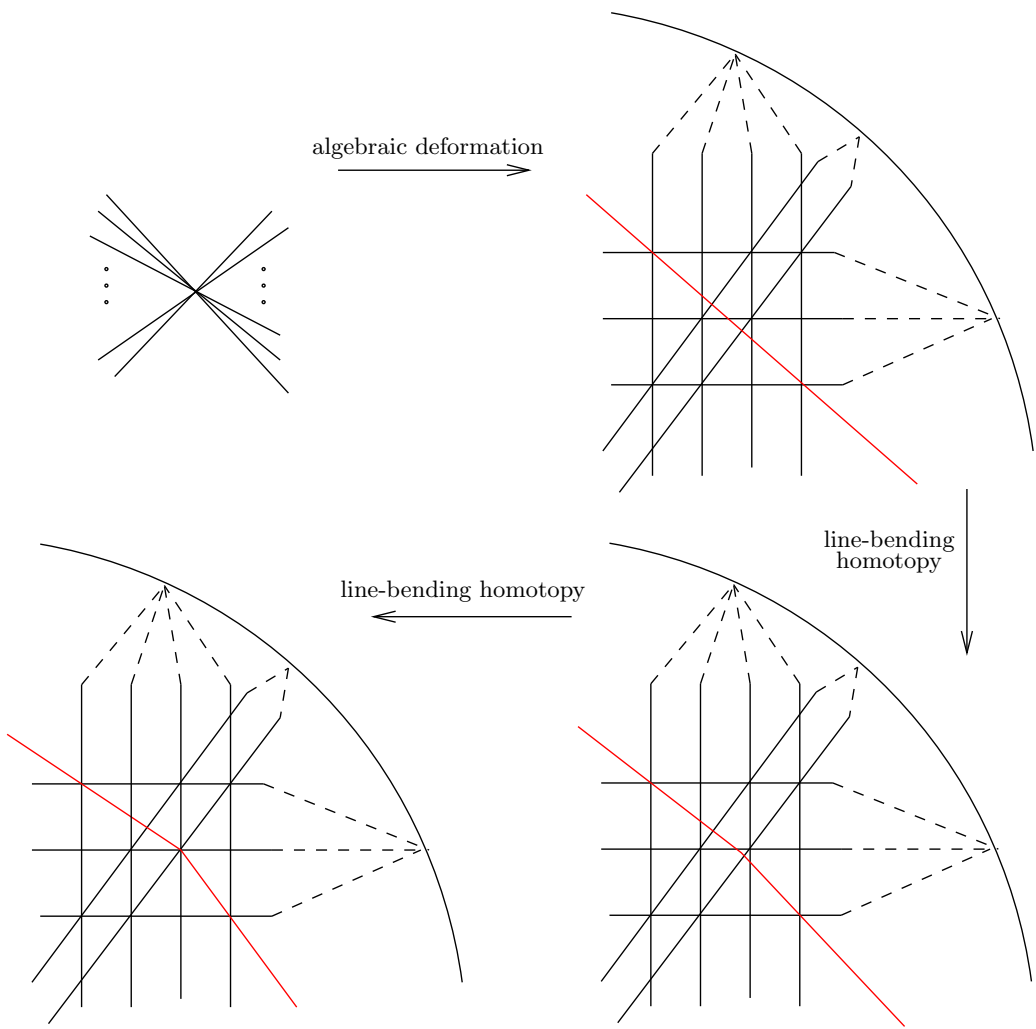
\begin{figure}[htb]
	\centering
	\bigskip
	\labellist
	\small\hair 2pt
	\pinlabel $\text{algebraic deformation}$ at 145 302
	\pinlabel $\text{line-bending homotopy}$ at 157 150
	\pinlabel $\text{line-bending}$ at 303 167
	\pinlabel $\text{homotopy}$ at 303 160
	\endlabellist
	\includegraphics[scale=1.1]{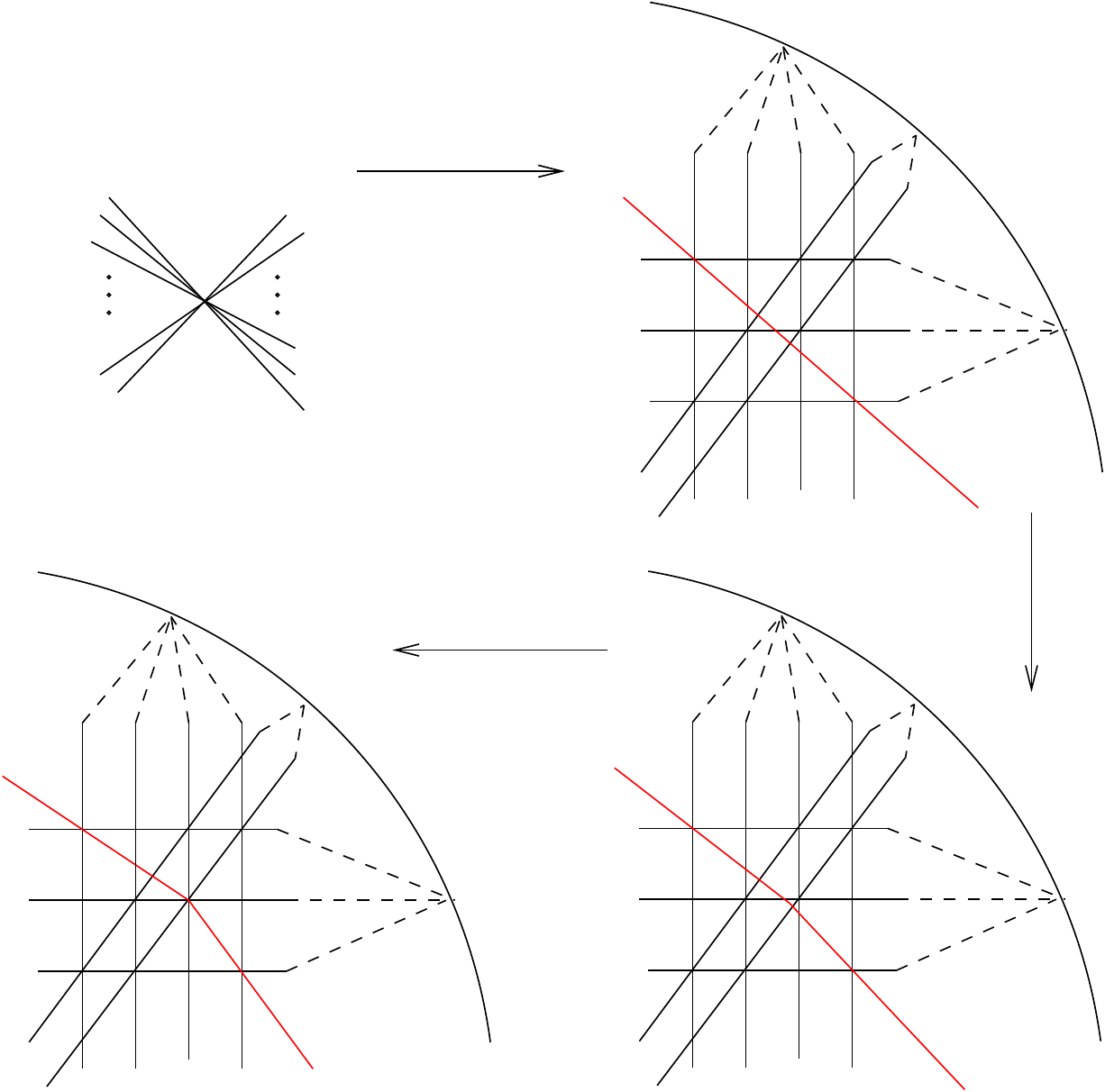}
	\caption{A long-term homotopy from a pencil of lines to $\mathcal{Q}$.}
	\label{fig:orline2}
\end{figure}

\begin{example} \label{long-homotopy} Consider the pencil of 11 lines indexed from $0$ to $10$, with weights $w_0=4$, $w_1=w_2=w_3=w_4=w_5=w_7=5$, $w_6=w_8=w_9=6$, and $w_{10}=8$. Observe that any arrangement of straight lines is related to the pencil by linear deformation (scaling the constant terms of the linear equations to $0$).  
	Using such a deformation, let   
	$\mathcal{Q}_{t_0}$ be the arrangement shown in Figure~\ref{fig:orline2}, where  $\ell_{10}$ is a straight line. Unlike the arrangement $\mathcal{Q}$, $\ell_{10}$ 
	does not pass through the intersection point $b$ of $\ell_3$, $\ell_6$ and $\ell_9$.  The corresponding picture deformation of the weighted 
	pencil gives a deformation of the surface singularity. We can extend the picture deformation to a smooth graphical homotopy which for $t_0<t<1$ bends the pseudoline $\ell_{10}$ towards the intersection $\ell_3\cap \ell_6\cap \ell_9$, and at $t=1$ realizes the configuration $\mathcal{Q}$. (We implicitly use Proposition~\ref{symp-config} to symplectify the family of pseudolines to a smooth graphical homotopy of symplectic line arrangements.)
		
	Now, consider the Stein fillings $W_t$ correspond to the arrangements $\mathcal{Q}_t$, $0 \leq t \leq 1$. For $0< t<1$, the Stein fillings 
	are diffeomorphic to Milnor fibers of the corresponding smoothings of the singular complex surface. 
	Indeed, the Lefschetz fibrations given by Lemma~\ref{construct-fibration} are all equivalent, and for $t$ close to $0$ the smooth graphical homotopy is a picture deformation.
	When $t=1$, Lemma~\ref{Yorevkov} says that the Stein filling $W$ arising from $\mathcal{Q}$ is not strongly diffeomorphic to any Milnor fiber. 
	The topology of $W$ is different from that 
	of $W_t$: as a smooth manifold, $W_t$ for $t<1$ is obtained from $W$ by rational blow-down.
	The corresponding Lefschetz fibrations are related via the positive monodromy substitution 
	given by the daisy relation~\cite{HMVHM}, see Figure~\ref{daisy}.
\end{example}
\begin{figure}[htb]
	\centering
	\includegraphics[scale=1.1]{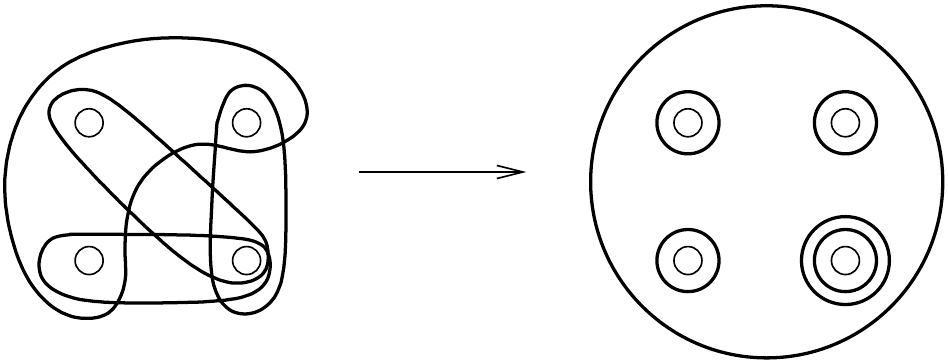}
	\caption{The Stein filling $W$ is related to the Milnor fibers $W_t$ by the monodromy substitution as shown.}
	\label{daisy}
\end{figure}

\subsection{Violating positivity of intersections and weight constraints} \label{subs:weight}

Although we have seen that we can produce many examples of unexpected Stein fillings using smooth graphical deformations which satisfy positivity of intersections and the weight constraints, we also can construct examples where a Stein filling arises from a configuration of curves such that every smooth graphical homotopy from the germ curvetta violates the weight constraints.

\begin{example}\label{example:weights}
	Consider again the configuration $\mathcal{Q}$ from Example~\ref{example-orevkov} of $11$ symplectic lines $\{L_k\}_{k=1}^{11}$. We compare this to a pencil of lines with weights 
	\begin{equation} \label{eq:weights}
	w_0=w_3=4, \quad w_1=w_2=w_4=w_5=w_6=w_7=w_9=5, \quad w_8=w_{10}=6.
	\end{equation}
	These are chosen such that $w_k=w(L_k)$, so they are the minimal possible weights satisfying the hypotheses of Corollary~\ref{cor:nonMilnor}. 
	We can show that there is no smooth graphical homotopy from this pencil to $\mathcal{Q}$ satisfying these weight constraints.
\end{example}

\begin{prop} The arrangement $\mathcal{Q}$ cannot be obtained from the pencil of lines by a smooth graphical homotopy satisfying the weight constraints as above,
	if we consider homotopies that are analytic in $t$ or satisfy a non-degeneracy condition such as (\ref{nonzero-deriv}).
\end{prop}

This statement follows from the following lemma, which shows that for combinatorial reasons, there are no ``intermediate'' arrangements between the pencil and $\mathcal{Q}$, 
so if a homotopy existed, it would have to deform the pencil immediately into an arrangement with the same incidence relations as $\mathcal{Q}$.

\begin{lemma} Let  $\mathcal{Q}_t = \cup_{k=0}^{10} L^t_k$ be a smooth graphical homotopy such that $\mathcal{Q}_0$ is a 
	pencil of 11 lines, and $\mathcal{Q}_1=\mathcal{Q}$ (after an appropriate choice of coordinates).  Suppose that all intersections 
	$L_i^t \cdot L_j^t$ are positive, and each $L_k^t$ has no more than $w_k$ intersection points  at all times $t \in [0,1]$. Then, the homotopy $\mathcal{Q}_t$ immediately deforms
	the pencil of lines into an arrangement combinatorially equivalent to $\mathcal{Q}$, perhaps after restricting to a smaller time interval: there exists $\tau\geq 0$ such 
	that $\mathcal{Q}_\tau$ is a pencil, and  $\mathcal{Q}_t$ is combinatorially equivalent to $\mathcal{Q}$ for all $t \in (\tau, 1]$.
\end{lemma}

\begin{proof} 
	Any two lines in the pencil have algebraic intersection number 1.  Since intersections remain inside the Milnor ball during the homotopy and remain positive at all times, throughout the homotopy any two components  $L_i^t$ and $L_j^t$ of $\mathcal{Q}_t$ intersect exactly once. This allows us to work 
	with $\mathcal{Q}_t$ as with pseudoline arrangements in  Proposition~\ref{orevkov-exotic}.

	We examine possible combinatorics of an arrangement with the weight restrictions as above. The analysis below works at any time $t$.  
	For each individual line $L_k$, we write $L_k^t$ for its image under the homotopy at time $t$. For $t=0$, the lines $L_k^0$ form a pencil; for $t=1$, we have 
	$\mathcal{Q}=\cup L^1_k$. 
	
	In the arrangement $\mathcal{Q}$, the line $L_0$ 
	contains 4 intersection points. These are points where $L_0$ meets the pencil $L_1, L_2, L_3, L_4$ of vertical lines, the pencil $L_5, L_6, L_7$ of horizontal lines, 
	the two diagonal lines $L_8, L_9$, and the bent line $L_{10}$.
	The weight condition then implies that $L^t_0$ can never have more than 
	4 intersection points. Note that $L_3$ also has only 4 intersection points, so the same is true for $L^t_3$.  It follows that  at most one intersection point on $L^t_0$ can have multiplicity 5 or greater:
	if there are two such points, there would be two pencils of 5 or more lines. Even if $L^t_3$ is in one of these pencils, it would intersect the lines of the other pencil
	in 5 or more distinct points, a contradiction. 
	Next, observe that no line has more than 6 intersection points, so no pencil can contain more than 6 lines unless all the 
	lines are concurrent. We conclude that $L^t_0$ must have at least 3 intersection points for all~$t$, because it is not possible to distribute the 10 other lines 
	into two intersection points on $L_0$ subject to  these conditions. 
	
	Observe that $\mathcal{Q}_t$ must be combinatorially equivalent to $\mathcal{Q}$ for $t$ close to $1$. Indeed, for $t$ sufficiently close to 1, the four distinct 
	intersection points on $L_0$ remain distinct on $L^t_0$. Similarly, for $t$ close to 1, each of $L^t_5$, $L^t_6$, and $L_7^t$ have {\em at least} 5 distinct intersection points 
	with the other curves in the arrangement $\mathcal{Q}_t$. On the other hand, due to weight restrictions each of these curves has {\em at most} 5 intersection points. 
	It follows that $L^t_5$, $L^t_6$, and $L_7^t$ have exactly 5 intersection points each, and the curves of $\mathcal{Q}_t$ meeting at each intersection have the same 
	incidence relations as the corresponding lines in $\mathcal{Q}$. Thus, the incidences involving $L_0^t$, as well as the incidences for 
	the ``grid'' intersections between $L^t_1, L^t_2, L^t_3, L^t_4$ and   $L^t_5$, $L^t_6$, $L_7^t$, are the same as in $\mathcal{Q}$ for $t$ close to 1. All the remaining 
	intersections in $\mathcal{Q}_t$ are double points, and they cannot merge with other intersections if $t$ is sufficiently close to 1.  
	
	The above argument shows that $\{t\in [0,1]: \mathcal{Q}_t \text{ is combinatorially equivalent to } \mathcal{Q} \}$ is open. Now, suppose that  
	$\mathcal{Q}_t$ is equivalent to $\mathcal{Q}$ for $t >t_0$. We examine the combinatorial possibilities for $\mathcal{Q}_{t_0}$, assuming that 
	this arrangement is not a pencil. Consider two cases, 1) $L_0^{t_0}$ has 4 distinct intersection points, and 2) $L_0^{t_0}$ has 3 distinct intersection points. 
	In the first case, it follows that $\mathcal{Q}_{t_0}$ must be combinatorially equivalent 
	to $\mathcal{Q}$. This is because all the incidence relations valid for $t>t_0$ still hold by taking a limit as $t \to t_0$.
	As in the proof of Proposition~\ref{orevkov-exotic}, we see that no two intersection points can collapse (if they do, all the curves must be concurrent).
	It follows that in this case, all the incidence relations  in $\mathcal{Q}_{t_0}$  are the same as in $\mathcal{Q}$. 
	
	In the second case, there are 3 intersection points on $L_0$. Again, because all incidences hold after taking limits as $t \to t_0$, the arrangement 
	$\mathcal{Q}_{t_0}$ satisfies all the incidence relations of $\mathcal{Q}$. Additionally, two of the intersection points on $L_0$ collapse. It follows from the proof of 
	Proposition~\ref{orevkov-exotic} that in this case $\mathcal{Q}_{t_0}$ must be a pencil, contradicting the assumption that $L_0^{t_0}$ has 3 distinct intersection points.

	We conclude that if  $\mathcal{Q}_t$ is combinatorially equivalent to $\mathcal{Q}$ for all $1\geq t>t_0$, and   $\mathcal{Q}_{t_0}$ is different, 
	then $\mathcal{Q}_{t_0}$ must be a pencil. 
\end{proof}

We have just seen that there are examples of Stein fillings arising from graphical smooth homotopies which do not satisfy the weight constraint 
(and such that there is no possible graphical smooth homotopy which does satisfy the weight constraint). On the other hand, 
we do not have examples of Stein fillings associated to a configuration of graphical curves which cannot 
be related to the curvetta germ by a smooth graphical homotopy satisfying positivity of intersections between the curve components. 
We suspect that in fact, there may always be a smooth graphical homotopy maintaining positivity of intersections.

\begin{question}\label{question:positive} Suppose  $C^0=\{ C^0_1, C^0_2, \dots, C^0_m \}$ and $C^1=\{ C^1_1, C^0_2, \dots, C^1_m \}$ are two collections of symplectic disks in $B^4_r$ such that $C^t_i$ intersects $C^t_j$ positively transversally or with a local holomorphic model. Further assume that the boundaries of $C^0$ and $C^1$ are isotopic braids in  $S^3_r$.  Does 
	there exists a continuous family $\{ C^t_1, C^t_2, \dots, C^t_m \}$ of symplectic disks, all with isotopic boundary braid for $t\in [0,1]$ 
	extending this pair of arrangements, such that for each $t$, $C^t_i$ and $C^t_{i'}$ have positive intersections?
\end{question}

To prove existence of such a homotopy, one could realize $C^0$ and $C^1$ as $J_0$ and $J_1$-holomorphic curves respectively, for
almost complex structures $J_0$ and $J_1$ which are compatible with the standard symplectic structure, with appropriate convexity conditions at the boundary of the ball. One could connect $J_0$ and $J_1$ through a family $J_t$ of almost complex structures with the same properties, and then try to find a family $C^t_i$ of $J_t$ holomorphic disks interpolating between $C^0_i$ and $C^1_i$ for each $i$. The difficulty arises in analyzing the moduli spaces of $J$-holomorphic curves with appropriately chosen boundary conditions (either using an SFT set-up or a totally-real boundary condition). Compactness issues in the moduli space must be overcome to obtain a positive answer to Question~\ref{question:positive}. Because such techniques are far beyond the scope of this article, and the answer to the question is not central to our investigations, we leave this open.

\begin{remark} 
	If a smooth graphical homotopy fails to satisfy the weight constraints or positivity of intersections, we cannot construct a sequence of Stein fillings
	using Lemma~\ref{construct-fibration}. However, we can ``connect'' the singular complex surface $(X, 0)$ to the Stein filling $W$ via a family 
	of {\em achiral} Lefschetz fibrations 
	(see \cite[Section 8.4]{GS}).
	
	Consider Example~\ref{example:weights}. We will use the homotopy of pseudoline arrangements given in Example~\ref{long-homotopy}. For $0<t<1$, the pseudolines $\ell_3$, $\ell_6$, $\ell_9$ and $\ell_{10}$ have more intersection points than the weights~(\ref{eq:weights}) allow. We need to compensate the higher weights to obtain the required open book monodromy, so we place {\em negative} free marked points on these lines:  $\ell_3$, $\ell_6$, $\ell_9$ need one negative marked point each to compensate for one extra positive intersection, and $\ell_{10}$ needs 2 negative points. In the open book monodromy, every negative marked point contributes a negative Dehn twist around the corresponding hole. It follows from the proof of Lemma~\ref{same-monodromy} that with these additional negative twists, the resulting open book supports $(Y, \xi)$. The corresponding vanishing cycles determine an achiral Lefschetz fibration. The negative Dehn twists corresponds to a ``negative'' blow-up in the smooth category (the 4-manifold changes by taking a connected sum with $\cptwo$).
\end{remark}

\bibliography{references}
\bibliographystyle{alpha}

\end{document}